\documentclass[10pt,reqno,a4]{amsart}
\usepackage[margin=2.5cm]{geometry}
\usepackage[utf8]{inputenc}
\DeclareUnicodeCharacter{00A0}{~}
\usepackage{amssymb,amsmath,amscd,amsfonts,amsthm,bbm}
\usepackage[usenames, dvipsnames]{color}
\usepackage{graphicx}
\usepackage{caption}
\usepackage{subcaption}
\usepackage{enumitem}
\usepackage[colorlinks=true]{hyperref}
\usepackage{array}
\usepackage{soul}

\usepackage[nocorrection]{optional} 

\newcommand{\correction}{\opt{correction}}

\newtheorem{theo}{Theorem}[section]

\newtheorem{lemma}[theo]{Lemma}
\newtheorem{coro}[theo]{Corollary}
\newtheorem{Def}[theo]{Definition}
\newtheorem{prop}[theo]{Proposition}

\newtheorem{claim}[theo]{Claim}
\theoremstyle{remark}
\newtheorem{rem}{Remark}[section]
\newtheorem{example}{Example}[section]

\theoremstyle{definition}
  \newtheorem{notation}[theo]{Notation}

\newcommand\nc\newcommand
\newcommand\dmo\DeclareMathOperator

\nc{\dz}{{\bf d}_z}
\nc{\N}{\mathbb{N}}
\nc{\R}{\mathbb{R}}
\nc{\C}{\mathbb{C}}
\nc{\E}{\mathbb{E}}
\nc{\PP}{\mathbb{P}}
\nc{\bdm}{\begin{displaymath}}
\nc{\edm}{\end{displaymath}}
\nc{\bea}{\begin{eqnarray*}}
\nc{\eea}{\end{eqnarray*}}
\nc{\la}{\langle}
\nc{\ra}{\rangle}
\nc{\Cplus}{\mathbb{C}_+}
\nc{\Rplus}{\mathbb{R}^+}
\nc{\pitilde}{\tilde{\pi}}
\nc{\tv}{\mathrm{tv}}
\nc{\Cnabla}{\mathbb{C}^{\nabla}}
\nc{\im}{\mathrm{Im}}
\nc{\bs}{\boldsymbol}
\nc{\ti}{\tilde}
\nc{\Msub}{M_{\mathrm{sub}}}
\nc{\diag}{\mathrm{diag}}
\nc{\ii}{\mathrm{i}}
\numberwithin{equation}{section}
\nc{\bv}{\boldsymbol{\varepsilon}}
\nc{\dzz}{\boldsymbol{\delta}_z}
\nc{\tr}{\mathrm{tr}\,}
\nc{\cvgP}[1]{\xrightarrow[#1]{\mathcal P}}
\nc{\cvgD}{\xrightarrow[]{\mathcal D}}
\nc{\eqdef}{\stackrel{\triangle}{=}} 
\nc{\lrn}{\left|\!\left|\!\left|} 
\nc{\rrn}{\right|\!\right|\!\right|} 
\nc{\bell}{\boldsymbol{\ell}}
\nc{\blambdaR}{\boldsymbol{\lambda}_R}
\nc{\comment}[1]{\textcolor{blue}{[{\it #1}]}}

\nc{\rhoVn}{\rho(V_n)}	
\nc{\rhoV}{\rho(V)}	
\nc{\tkappa}{{\tilde \kappa}}

\nc{\Ncenter}{\stackrel{\circ}{\mathcal N}_n}
\nc{\Sn}{{\mathcal S}_n}
\nc{\Kn}{{\mathcal K}_n}

\dmo{\Imm}{Im}
\dmo{\Real}{Re}
\dmo*{\dist}{dist}
\dmo{\var}{var}
\dmo{\trace}{Tr}
\dmo{\rank}{rank}

\nc{\Rd}{R_n^{1/2}}
\nc{\Rdb}{\bar{R}_n^{1/2}}
\nc{\V}{\mathcal V}
\nc{\VV}{|\mathcal V|^2}
\nc{\dlp}{d_{\mathcal LP}}
\nc{\Qij}{Q_{[ij]}}

\nc{\cU}{{\mathcal U}}

\nc{\smax}{\sigma_{\max}} 
\nc{\smin}{\sigma_{\min}} 

\nc{\Xn}{X^{\mathcal N}}
\nc{\Yn}{Y^{\mathcal N}}
\nc{\Hn}{H^{\mathcal N}}
\nc{\Gn}{G^{\mathcal N}}
\nc{\tGn}{\widetilde{G}^{\mathcal N}}
\nc{\Fn}{F^{\mathcal N}}
\nc{\cFn}{{F}^{'\mathcal N}}

\nc{\vxi}{\vec{\xi}}
\nc{\vxin}{\vec{\xi}^{\mathcal N}}
\nc{\xin}{\xi^{\mathcal N}}
\nc{\bxin}{\bar{\xi}^{\mathcal N}}

\nc{\Ec}{\E_{\{1\}}}

\nc{\Oeta}[1]{{\mathcal O}_{\eta} \left( {#1}\right)}
\nc{\vOeta}[1]{\vec{\mathcal O}_{\eta} \left( {#1}\right)}
\nc{\specnorm}[1]{\left\| {#1}\right\|_{\mathrm{sp}}}
\nc{\posneq}{\succcurlyeq_{\neq}}

\nc{\bq}{\bs{q}}
\nc{\qt}{\widetilde{q}}
\nc{\bqt}{\bs{\qt}}
\nc{\qvec}{\vec{\bs q}}
\nc{\qvecstar}{\qvec_*}
\nc{\bqstar}{\bq_*}
\nc{\bqtstar}{\bqt_*}

\nc{\rr}{r}
\nc{\rt}{\widetilde{r}}
\nc{\br}{\bs{r}}
\nc{\brt}{\bs{\rt}}
\nc{\rvec}{\vec{\bs r}}
\nc{\rvecstar}{\rvec_*}
\nc{\brstar}{\br_*}
\nc{\brtstar}{\brt_*}

\nc{\pt}{\widetilde{p}}
\nc{\bp}{\bs{p}}
\nc{\bpt}{\bs{\pt}}
\nc{\pvec}{\vec{\bp}}

\nc{\vp}{\varphi}
\nc{\vpt}{\widetilde{\varphi}}
\nc{\bphi}{\bs{\varphi}}
\nc{\bphit}{\bs{\widetilde \varphi}}
\nc{\phivec}{\bs{\vec{\varphi}}}

\nc{\q}{\bs{q}}
\nc{\qtilde}{\bs{\widetilde q}}	
\nc{\p}{\bs{p}}
\nc{\ptilde}{\bs{\widetilde p}}
\nc{\tp}{\widetilde{\varphi}}
\nc{\bphitilde}{\bs{\widetilde \varphi}}
\nc{\bphivec}{\bs{\vec{\varphi}}}

\nc{\vecepsilon}{\vec{\boldsymbol{\varepsilon}}}

\nc{\Qa}{{\mathcal Q}(\bs{\alpha}, A, \bs{a})}
\nc{\Qb}{{\mathcal Q}(\bs{\beta}, B, \bs{b})}

\nc{\Sab}{{\mathcal S}_{AB}}
\nc{\Sa}{{\mathcal S}_{A}}
\nc{\Sb}{{\mathcal S}_{B}}

\nc{\Sig}{A}

\dmo{\eps}{\varepsilon}
\dmo{\ls}{\lesssim}
\dmo{\gs}{\gtrsim}

\nc{\expo}[1]{\exp \left( #1 \rule{0mm}{3mm}\right)}
\dmo{\e}{\mathbb{E}}
\dmo{\pr}{\mathbb{P}}
\dmo{\un}{\mathbbm{1}}
\dmo{\1}{\mathbf{1}}
\nc{\tran}{\mathsf{T}} 	
\nc{\scut}{\snot}		
\nc{\snot}{\sigma_0}	
\nc{\mN}{\mathcal{N}}
\nc{\vpmax}{\|\bphi\|_\infty}
\nc{\tpmax}{\|\bphitilde\|_\infty}
\nc{\bS}{\bs{S}}
\nc{\bd}{\bs{d}}
\nc{\bdt}{\bs{\tilde{d}}}
\nc{\dt}{\tilde{d}}
\nc{\HS}{\mathsf{HS}}
\nc{\Mo}{M_0}
\nc{\Res}{\bs{R}}
\nc{\Y}{\bs{Y}}
\nc{\Am}{\bs{A}^{(m)}}
\nc{\Vm}{\bs{V}^{(m)}}
\nc{\Ym}{\bs{Y}^{(m)}}
\nc{\Lm}{\check{\bs{L}}^{(m)}}
\nc{\wY}{{\widetilde Y}}
\nc{\wA}{{\widetilde A}}
\dmo{\Leb}{Leb}

\nc{\Blue}[1]{\textcolor{blue}{#1}}
\nc{\comN}[1]{\textcolor{ForestGreen}{#1}}

\title[ Non-Hermitian random matrices with a variance profile]{
Non-Hermitian random matrices with a variance profile (I): Deterministic equivalents and limiting ESDs} 

\author[N. Cook, W. Hachem, J. Najim, D. Renfrew]{Nicholas Cook, Walid Hachem, Jamal Najim and David Renfrew} 

\date{\today}


\begin{document}

\begin{abstract} 
For each $n$, let $A_n=(\sigma_{ij})$ be an $n\times n$ deterministic matrix and let $X_n=(X_{ij})$ be an $n\times n$ random matrix with i.i.d.\ centered entries of unit variance.
We study the asymptotic behavior of the empirical spectral distribution $\mu_n^Y$ of the rescaled entry-wise product
\[
Y_n = \left(\frac1{\sqrt{n}} \sigma_{ij}X_{ij}\right).
\]
For our main result we provide a deterministic sequence of probability measures $\mu_n$, each described by a family of \emph{Master Equations}, such that the difference $\mu^Y_n - \mu_n$
converges weakly in probability to the zero measure. A key feature of our results is to allow some of the entries $\sigma_{ij}$ to vanish, provided that the standard deviation profiles $A_n$ satisfy a certain quantitative irreducibility property. 
An important step is to obtain quantitative bounds on the solutions to an associate system of Schwinger--Dyson equations, which we accomplish in the general sparse setting using a novel \emph{graphical bootstrap} argument.
\end{abstract}

\maketitle

\tableofcontents


\section{Introduction} 
\label{sec:intro}
For an $n\times n$ matrix $M$ with complex entries and eigenvalues $\lambda_1,\dots,\lambda_n\in \C$ (counted with multiplicity and labeled in some arbitrary fashion), the \emph{empirical spectral distribution (ESD)} is given by
\begin{equation}	\label{def:esd}
\mu_M = \frac1n\sum_{i=1}^n \delta_{\lambda_i}\;.
\end{equation}
A seminal result in non-Hermitian random matrix theory is the \emph{circular law}, which describes the asymptotic global distribution of the spectrum for matrices with i.i.d.\ entries of finite variance.
The following strong form of the circular law was established by Tao and Vu \cite{tao2010random}, and is the culmination of the work of many authors \cite{Ginibre-1965, mehta1967, girko1985circular, Bai-circular-law-1997, bai-sil-book, gotze2010circular, pan-zhou-circular-2010, TaVu:circ} -- see the survey \cite{2012-bordenave-chafai-circular} for a detailed historical account. 

\begin{theo}[Circular law]	\label{thm:circ}
Let $\xi$ be a complex random variable of zero mean and unit variance, and for each $n$ let $X_n = (X_{ij}^{(n)})$ be an $n\times n$ matrix whose entries are i.i.d.\ copies of $\xi$.
Then almost surely, the ESDs $\mu_{\frac1{\sqrt{n}}X_n}$ converge weakly to the circular measure
\[
\mu_{\mathrm{circ}}(dx\, dy) := \frac1{\pi}1_{\{ |x|^2+|y|^2\le 1\}} \, dx\, dy.
\]
\end{theo}

One of the remarkable features of the circular law is that the asymptotic behavior of ESDs is insensitive to specific details of the entry distributions, apart from the first two moments. This is an instance of the \emph{universality phenomenon} in random matrix theory.

The circular law has been an important tool for understanding the stability of dynamical systems on complex networks, going back to work of May in ecology \cite{may1972will}, and later work of Sompolinski et al.\ in neuroscience \cite{SCS}. 
May used an i.i.d.\ matrix $X_n$ to model the \emph{community matrix} for a food network of $n$ species, where the entry $X_{ij}$ determines the rate of growth (or decay) of the population of species $i$ due to species $j$.
The stability of the system is determined by the spectrum of $X_n$ -- specifically by whether it has eigenvalues with sufficiently large real part -- and May used the circular law\footnote{At the time the circular law was only known to hold in the complex Gaussian case thanks to work of Ginibre and Mehta \cite{Ginibre-1965, mehta1967}. Strictly speaking, May's argument assumes that there are asymptotically no eigenvalues outside the limiting support, which is now known to hold under some moment hypotheses \cite{bai-sil-rmta12,BCCT}.} to derive a criterion for stability.

Recently there has been increasing interest in extending the arguments of \cite{may1972will,SCS} to matrix models with more structured distributions.
In neural networks, where random matrices are used to model the \emph{synaptic matrix}, the work \cite{RaAb} considered perturbed i.i.d.\ matrices of the form $X_n+M_n$, where $M_n$ is a fixed matrix with all entries within a fixed proportion of columns taking a fixed positive value $\mu^+$, and all remaining entries taking a fixed negative value $\mu^-$. Their motivation was to conform to Dale's Law, stating that neurons are either inhibitory or excitatory. In this case $M_n$ is a rank-one perturbation; as was later shown rigorously in work of Tao \cite{Tao:outliers}, low rank perturbations do not affect the limiting spectral distribution, but may lead to the creation of outlier eigenvalues. 

Several recent works have studied the limiting spectral distribution for random matrices of the form 
\begin{equation}	\label{sdev.profile}
A_n\odot X_n = (\sigma_{ij}X_{ij})
\end{equation}
(suitably rescaled) where $A_n=(\sigma_{ij})$ is a fixed (deterministic) \emph{standard deviation profile}. 
From a modeling perspective the $\sigma_{ij}$ can reflect the varying degrees of interaction between species/neurons.
In theoretical ecology the works \cite{allesina2015predicting,Allesina:2015ux} considered asymmetric standard deviation profiles, i.e.\ taking $\sigma_{ij}\ne \sigma_{ji}$, in order to create more realistic predator-prey cascading relationships.
In neuroscience the works \cite{ASS,ARS} considered matrices $A_n$ partitioned into a bounded number of block submatrices having constant entries within each block, in order to model networks with a bounded number of cell types. 
We also note that predating these works, Girko \cite[Chap. 25, 26]{girko-canonical-equations-I} (see also the references therein) studied non-Hermitian matrices with standard deviation and mean profiles and provided {\em canonical equations} to describe the limiting spectral densities.

Some works have also gone beyond matrices with independent entries of specified mean and variance, for instance considering products and sums of deterministic matrices with a random matrix having i.i.d.\ entries \cite{Ahmadian:2015xw}, or allowing correlations between entries $X_{ij}, X_{ji}$ \cite{AlTa12:nature}. 
We also mention that parallel to the study of non-Hermitian matrices there have been many works devoted to the study of Hermitian random matrices with a
variance profile, both Wigner and Gram-type -- see for instance
Girko \cite[Chapter 7, 8]{girko-canonical-equations-I}, Shlyakhtenko \cite{shlyakhtenko-96}, Guionnet \cite{Guionnet:band}, Anderson and Zeitouni \cite{AnZe:band},
Hachem et al.\ \cite{hachem-et-al-2007}, Ajanki et al.\ \cite{ajanki2015universality}.

As has been pointed out in the ecology 
literature \cite{Allesina:2015ux}, a key feature that is missing from the literature on models of the form \eqref{sdev.profile} is to allow $A_n$ to have zero entries. Indeed, the nodes in large real-world ecological or neural networks do not interact with all other nodes.
One fix has been to take $A_n$ to have i.i.d.\ Bernoulli($p$) indicator entries, independent of $X_n$, i.e.\ to model the support of the network by a sparse Erd\H{o}s--R\'enyi digraph. As was shown by Wood \cite{Wood:sparse} the circular law still holds for $A_n\odot X_n$ (after rescaling by $(pn)^{-1/2}$) if $p\ge n^{\alpha-1}$ for any fixed $\alpha\in (0,1]$. However, the valence of the nodes in the resulting network is highly concentrated around $pn$, while the valence distribution for real-world networks is highly non-uniform \cite{Allesina:2015ux}.
With the ability to set $A_n$ deterministically one can reflect some known underlying geometry of the network.

In the present work, our main focus is to understand the asymptotic global spectral distribution under minimal assumptions on the variance profile.
Key features of our results, which also present significant challenges in the proofs, are to allow a large number of the entries $\sigma_{ij}$ to be zero, as well as to allow asymmetry (i.e.\ $\sigma_{ij}\ne \sigma_{ji}$).
As an example, we are able to handle variance profiles satisfying a \emph{robust irreducibility} condition (see Definition \ref{def:robust}), which includes matrices that are close in the cut norm to an irreducible graphon (Lemma \ref{lem:graphon}).

As compared with the proof of the circular law, the identification and description of a limiting measure is significantly more involved.
In this article we prove the existence of a sequence of deterministic measures -- called \emph{deterministic equivalents} -- which asymptotically approximate the random ESDs.
In particular we obtain non-trivial information even when the ESDs themselves do not converge to a limit.
The identification of the deterministic equivalents involves analysis of a (cubic) polynomial system of \emph{Master Equations} determined by the variance profile. A relative of the Master Equations known as the Quadratic Vector Equation was studied in recent work of Ajanki, Erd\H os and Kr\"uger and Alt, Erd\H os and Kr\"uger on the spectrum of Hermitian matrices with a variance profile \cite{AEKquadequations,AEKgram}. 

Since the initial release of this paper, a local law version of our main statement (Theorem \ref{thm:main}) was proved in \cite{alt2016local} under the restriction that the standard deviation profile $\sigma_{ij}$ is uniformly strictly positive and that the distribution of the matrix entries possesses a bounded density and has all its moments finite. 
In this case, it is also proved that the density of the deterministic equivalents is positive and bounded on its support. 
The latter properties may no longer hold if the standard deviation profile has zero entries or is not uniformly lower bounded. In these cases, the limiting distribution may offer a wider variety of behavior, such as a blowup or vanishing density at zero, or a point mass at zero; see \cite{cook-et-al-in-preparation}. 

It is by now well known that the study of ESDs for non-Hermitian random matrices is intimately connected with proving the quantitative invertibility of such matrices -- that is, establishing lower tail estimates for small singular values. The possible sparsity of the matrices considered here gives rise to significant challenges for this task. 
Bounds on the smallest singular value sufficient for our purposes were established by the first author in \cite{Cook:ssv}.
In the present work we obtain control on the remaining small singular values from Wegner-type bounds, which are established by a quantitative analysis of the Master Equations.
Specifically, the key is to show solutions to the \emph{Regularized Master Equations} (see \eqref{def:MEt}) are uniformly bounded in the spectral scale parameter $t$ and the dimension $n$. 
For this task we use an iterative \emph{graphical bootstrap} argument 
(reminiscent of bootstrap arguments from the theory of differential equations) that exploits expansion properties of a directed graph naturally associated to the variance profile.
(Graph expansion properties were also key for the analysis of the smallest singular value in \cite{Cook:ssv}.)
We discuss these aspects of the proof in more detail after presenting the results in Section \ref{sec:results}.

\subsection{The model} 

In this article we study the following general class of random matrices with non-identically distributed entries.

\begin{Def}[Random matrix with a variance profile]		\label{def:model}
For each $n\ge 1$, let $A_n$ be a (deterministic) $n\times n$ matrix with entries $\sigma_{ij}^{(n)}\ge0$, let $X_n$ be a random matrix with i.i.d.\ entries $X_{ij}^{(n)}\in \C$ satisfying
\begin{equation}	\label{standardized}
\E X_{11}^{(n)}=0\, , \quad \E|X_{11}^{(n)}|^2=1
\end{equation}
and set
\begin{equation}	\label{def:Yn}
Y_n = \frac1{\sqrt{n}}A_n\odot X_n
\end{equation}
where $\odot$ is the matrix Hadamard product, i.e.\ $Y_n$ has entries $Y_{ij}^{(n)} = \frac1{\sqrt{n}} \sigma_{ij}^{(n)} X_{ij}^{(n)}$.
The empirical spectral distribution of $Y_n$ is denoted by $\mu_n^Y$.
We refer to $A_n$ as the \emph{standard deviation profile} and to $A_n\odot A_n = \big((\sigma_{ij}^{(n)})^2\big)$ as the \emph{variance profile}.
We additionally define the \emph{normalized variance profile} as
$
V_n=n^{-1} A_n\odot A_n$. 
When no ambiguity occurs, we will drop the index $n$ and simply write $\sigma_{ij}, X_{ij}, V$, etc.  
\end{Def}

\begin{rem}
Note we do not assume the $(X_{ij}^{(n)})$'s are independent or identically distributed for different $n$'s.
\end{rem}

Our goal is to describe the asymptotic behavior of the ESDs $\mu_n^Y$ given a sequence of standard deviation profiles $A_n$ which can be sparse, and may not converge in any sense to a limiting standard deviation profile.

\subsection{Master equations and deterministic equivalents} 

The main result of this article states that under certain assumptions on the sequence of standard deviation profiles $A_n$ and the distribution of the entries of $X_n$, there exists a sequence of deterministic probability measures $\mu_n$ that are \emph{deterministic equivalents} of the
spectral measures $\mu_n^Y$ such that 
\[
\mu_n^Y \sim \mu_n\qquad \text{in probability} \quad (n\to\infty),
\]
see Section \ref{subseb:convofmeas} for notation.


The measures $\mu_n$ are described by a polynomial system of \emph{Master Equations}.
Denote by $V_n^\tran$ the transpose matrix of $V_n$ and by $\rhoVn$ its spectral radius. 
For a parameter $s\ge0$, the Master Equations are the following system of $2n+1$ equations in $2n$ unknowns $q_1,\dots,q_n,\qt_1,\dots,\qt_n$:
\begin{equation}	\label{def:ME2}	
\begin{cases}
q_i&= \dfrac{ (V_n^\tran {\bq})_i} {s^2+(V_n\bqt)_i(V_n^\tran\bq)_i  }  \\
\\
\qt_i&= \dfrac{ (V_n \bqt)_i} {s^2+(V_n\bs {\qt})_i(V_n^\tran\bq)_i  } \\ 
\\
&\sum_{i\in [n]} q_i  = \sum_{i\in [n]} {\qt}_i
\end{cases}\, , 
\quad q_i, \qt_i\ge0,\ i\in [n],
\end{equation}
where $\bq,\bqt$ are the $n\times 1$ column vectors with components $q_i,\qt_i$, respectively. 

If $s\ge \sqrt{\rhoVn}$, it can be proved that the Master Equations admit the unique trivial solution 
$( \bq,  \bqt)=0$. 
Provided that $0<s<\sqrt{\rhoVn}$ and that the matrix $V_n$ is irreducible, it can be shown that the Master Equations admit a unique positive solution $( \bq,  \bqt)$
which depends only on $s$. 
This solution $s\mapsto (\bq(s),\bqt(s))$ is continuous on $(0,\infty)$.
With this definition of $\bq$ and $\bqt$, the deterministic equivalent $\mu_n$ is defined as the radially symmetric probability distribution over $\C$ satisfying
\[
\mu_n\{ z \in \C\, , \ |z|\le s\} = 1 - \frac 1n \bq^\tran (s) V_n \bqt(s)\ ,\quad s>0\ .
\]
It readily follows that the support of $\mu_n$ is contained in the disk of radius $\sqrt{\rhoVn}$.

\subsection*{Acknowledgements} 
The work of NC was partially supported by NSF grants DMS-1266164 and DMS-1606310.
The work of WH and JN was partially supported by the program ``mod\`eles
num\'eriques'' of the French Agence Nationale de la Recherche under the grant
ANR-12-MONU-0003 (project DIONISOS). Support of the Labex BEZOUT from 
the Universit\'e Paris Est is also acknowledged. The work of DR was partially supported by Austrian Science Fund (FWF): M2080-N35. The authors thank the referee for their careful reading and insights which enabled to substantially shorten the paper.


\section{Presentation of the results}
\label{sec:results}

\subsection{Notational preliminaries}		\label{sec:notation}

Denote by $[n]$ the set $\{1,\cdots,n\}$ and let 
$\C_+=\{ z\in \C\, ,\ \im(z)>0\}$. For ${\mathcal X} = \C$ or $\R$, let 
$C_c(\mathcal X)$ (resp.~$C_c^\infty(\mathcal X)$) the set of 
$\mathcal X \to\R$ continuous (resp.~smooth) and compactly supported functions. 
Let $\mathcal B(z,r)$ be the open ball of $\C$ with center $z$ and 
radius $r$. 
If $z\in \C$, then $\bar{z}$ is its complex conjugate; let $\ii^2=-1$. 
The Lebesgue measure on $\C$ will be either denoted by $\ell(\, dz)$ or $dx dy$.
For $x,y\in \R$ we write $\max(x,y)= x\vee y $ and $\min(x,y)=x\wedge y$. The cardinality of a finite set $S$ is denoted by $|S|$.
For $S\subset[n]$ and when clear from the context we will abbreviate $S^c= [n]\setminus S$.
 
\subsubsection{Matrices} 

We denote by $\1_n$ the $n\times 1$ vector of 1's. Given two $n\times 1$ vectors $\bs u,\bs v$, we denote their scalar product $\langle \bs u,\bs v\rangle=\sum_{i\in [n]} \bar{u}_i v_i$.

For a given matrix $A$, denote by $A^\tran$ its transpose and by $A^*$ its
conjugate transpose. Denote by $I_n$ the $n\times n$ identity matrix. If clear from
the context, we omit the dimension. For $a\in \C$ and when clear from the
context, we sometimes write $a$ instead of $a\, I$ and similarly write $a^*$
instead of $(aI)^*=\bar{a}I$.  For matrices $A,B$ of the same dimensions we denote by $A\odot B$ their Hadamard, or entrywise, product (i.e.\ $(A\odot B)_{ij} =
A_{ij} B_{ij}$). 

Given two Hermitian matrices $A$ and $B$, the notations $A \ge B$ and $A > B$ 
refer to the usual positive semidefinite ordering. 
Notations $\succ$ and $\succcurlyeq$ refer to the elementwise inequalities for 
real matrices or vectors. Namely, if $A$ and $B$ are real matrices, 
\[
A\succ B\quad \Leftrightarrow\quad A_{ij} > B_{ij}\quad \forall i,j\qquad \text{and}\qquad A\succcurlyeq B\quad \Leftrightarrow\quad A_{ij} \ge B_{ij}\ \quad\forall i,j .
\]
The notation $A \posneq 0$ stands for $A\succcurlyeq 0$ and $A\neq 0$. 
Given a matrix $A$, $\|A\|$ refers to its spectral norm, and 
$\lrn A \rrn_{\infty}$ to its max-row norm, defined as:
$
\lrn A \rrn_{\infty} := \max_{i\in[n]} \sum_{j=1}^n |A_{ij}|=\max_{\|{\bs u}\|_{\infty}\le 1} \| A{\bs u}\|_{\infty}
$.
We denote the spectral radius of an $n\times n$ matrix $A$ by
$
\rho(A) := \max\big\{\, |\lambda|\colon \text{ $\lambda$ is an eigenvalue of $A$}\,\big\}
$. 
If $A$ is a square matrix, we write $\im(A)=(A-A^*)/(2\ii)$. For an $n$-dimension vector $a$,
$\diag(a_i)$ denotes the $n\times n$ diagonal matrix with $a$ as its diagonal elements. 

\subsubsection{Convergence of measures} \label{subseb:convofmeas}
Given probability distributions $\nu_n,\nu$ over some set ${\mathcal X}$ ($=\R$
or $\C$), we will denote the weak convergence of $\nu_n$ to $\nu$ by
$\nu_n\xrightarrow[n\to\infty]{\bf w} \nu$. If $\nu_n$ is random,
$\nu_n\xrightarrow[n\to\infty]{\bf w} \nu$ almost surely (resp.~in probability)
stands for the fact that for all $f\in C_c({\mathcal X})$, 
\[
\int f\, d\nu_n \xrightarrow[n\to\infty]{} \int f\, d\nu 
\quad \text{almost surely} \ \text{(resp. in probability)} .  
\] 
Let $(\mu_n)$ and $(\nu_n)$ be deterministic sequences of probability distributions over ${\mathcal X}$, and let $(\nu_n)$ be tight, i.e.\ for all $\varepsilon>0$, one can find a compact set ${\mathcal K}_{\varepsilon}$ such that 
$
\sup_{n} \nu_n({\mathcal X} \setminus {\mathcal K}_{\varepsilon}) \le \varepsilon$.
We will denote by 
\[
\mu_n \sim \nu_n\quad \text{as}\quad n\to\infty
\]
the fact that the signed measure $\mu_n - \nu_n$ weakly converges to zero, i.e.\
$\int f\, d\mu_n - \int f\, d\nu_n \to 0$ for all $f\in C_c({\mathcal X})$. If
the sequence $(\mu_n)$ is random while $(\nu_n)$ is deterministic and tight,
then 
\[
\mu_n \sim \nu_n \quad \text{almost surely (resp. in probability)} 
\]
stands for
\[
\int f\, d\mu_n - \int f\, d\nu_n \xrightarrow[n\to\infty]{} 0 
\quad \text{almost surely} \ \text{(resp. in probability)} , 
\]
for all $f\in C_c({\mathcal X})$. Notice that $\mu_n \sim \nu_{\infty}$ is equivalent to $\mu_n\xrightarrow[n\to\infty]{\bf w} \nu_{\infty}$.

\subsubsection{Stieltjes transforms}
Let $\mu$ be a nonnegative finite measure on $\R$ and
\begin{equation}\label{def:Stieltjes}
g_\mu(\eta)=\int \frac{\mu(d\lambda)}{\lambda-\eta}\ ,\quad \eta\in \C_+
\end{equation}
its Stieltjes transform. Then the following properties are standard
$$
\mathrm{(i)}\quad g_{\mu}(\eta)\in \C_+\,,\qquad \mathrm{(ii)}\quad | g_{\mu}(\eta)| \le \frac {\mu(\R)}{\im(\eta)}\,,\qquad \mathrm{(iii)}
\ \lim_{y\to+\infty} -\ii y g_{\mu}(\ii y)= \mu(\R)\ .
$$
Moreover, $-(z+g_{\mu}(z))^{-1}$ is the Stieltjes transform of a probability measure, see for instance \cite[Theorem B.3]{book-weidmann}. In particular
\begin{equation}\label{eq:ST-estimate}
\left| z +g_{\mu}(z)\right|^{-1} \le {\im^{-1}(z)}\ ,\quad z\in \C_+ .
\end{equation} 

\subsubsection{Graph theoretic notation}\label{sec:graph-notations}

Given an $n\times n$ non-negative matrix $A=(\sigma_{ij})$ we form a directed graph $\Gamma=\Gamma(A)$ on the vertex set $[n]$ that puts an edge $i\to j$ whenever $\sigma_{ij}>0$.
We denote the out-neighborhood of a vertex $i\in [n]$ in the graph $\Gamma$ by
\begin{equation}	\label{def:nbrs}
\mN_{A}(i) := \{j\in [n]: \sigma_{ij}>0\}.
\end{equation}
Consequently, the in-neighborhood is denoted $\mN_{A^\tran}(i)$.
For a set $S\subset[n]$ we write 
\begin{equation}	\label{def:nbrs2}
\mN_{A}(S) := \bigcup_{i\in S} \mN_A(i) = \{j\in [n]: \mN_{A^\tran}(j)\cap S \ne \emptyset\}.
\end{equation}
For $\delta\in (0,1)$ we denote the associated \emph{densely-connected out-neighbors} of a set $S\subset[n]$ by
\begin{align}
\mN_{A}^{(\delta)}(S) &= \{j\in [n]: |\mN_{A^\tran}(j)\cap S| \ge \delta |S|\}. \label{def:dense-nbrs}
\end{align}
To obtain quantitative results we will generally work with the graph associated to the matrix
\begin{equation}	\label{def:Acut}
A(\scut) = (\sigma_{ij}1_{\sigma_{ij}\ge \scut})
\end{equation}
which only keeps the entries exceeding a fixed cutoff parameter $\scut>0$, setting the remaining entries to zero.

\subsection{Model assumptions}

We will establish results concerning sequences of matrices $Y_n$ as in Definition \ref{def:model} under various additional assumptions on $A_n$ and $X_n$, which we now summarize.
We note that many of our results only require a subset of these assumptions.

For our main result we need the following additional assumption on the distribution of the entries of $X_n$.

\vspace{.1cm}
\begin{enumerate}[leftmargin=*, label={\bf A0}]
\item\label{ass:moments} 
(Moments). We have $\E|X_{11}^{(n)}|^{4+\varepsilon}\le \Mo$ for all $n\ge1$ and some fixed $\varepsilon>0$, $\Mo<\infty$.
\end{enumerate}

\begin{rem}
\label{rmk:moment}
Assumption \ref{ass:moments} is needed to apply the results from \cite{Cook:ssv} to bound the smallest singular value as in \eqref{smin.tail}. It is also used in Section \ref{sec:svd} to quantitatively bound the difference between our random measures and their deterministic equivalents, which is crucial for obtaining logarithmic integrability of singular value distributions. This latter step can likely be accomplished with fewer moments, but we do not pursue this. 
\end{rem}


\vspace{.1cm}
\begin{enumerate}[leftmargin=*, label={\bf A1}]
\item\label{ass:sigmax}		
(Bounded variances).
There exists $\smax\in (0,\infty)$ such that
\[
\sup_n \max_{1\le i,j\le n} \sigma_{ij}^{(n)} \le \smax.
\]
\end{enumerate}

\begin{rem}[Convention]
While we will keep the value $\smax$ generic in the statements, we will always set it to 1 in the proofs, with no loss of generality. 
\end{rem}


In order to express the next key assumption, we need to introduce the following {\em Regularized Master Equations} which are a specialization of the Schwinger--Dyson equations of Girko's Hermitized model associated to $Y_n$ (see Section \ref{outline} for further discussion). The following is proved in Section \ref{sec:RME}.

\begin{prop}[Regularized Master Equations]
\label{prop:MEt}	
Let $n\ge 1$ be fixed, let $A_n$ be an $n\times n$ nonnegative matrix and write $V_n=\frac1nA_n\odot A_n$. Let $s,t > 0$ be fixed, and consider the following system of equations 
\begin{equation}
\label{def:MEt} 
\left\{ 
\begin{array}{ccc}
\rr_i&=& \dfrac{ (V_n^\tran  \br)_i+t} {s^2+((V_n\brt)_i+t) ((V_n^\tran \br)_i +t) }  \\
\\
\rt_i&=& \dfrac{ (V_n \brt)_i+t} {s^2+((V_n\brt)_i+t) ((V_n^\tran \br)_i +t) } \\ 
\end{array}
\right.\ ,
\end{equation}
where $\br=(\rr_i)$ and $\brt=(\rt_i)$ are $n\times 1$ vectors. Denote by $\rvec=\begin{pmatrix} \br \\ \brt\end{pmatrix}$. Then this 
system admits a unique solution $\rvec=\rvec(s,t) \succ 0$. This solution
satisfies the identity 
\begin{equation}	\label{q_t:trace}
\sum_{i\in [n]} \rr_i \ =\ \sum_{i\in [n]} \rt_i \, .
\end{equation}
\end{prop}
\vspace{.1cm}
\begin{enumerate}[leftmargin=*,label={\bf A2}]
\item\label{ass:admissible}		
(Admissible variance profile).
Let $\rvec(s,t)=\rvec_n(s,t)\succ 0$ be the solution of the Regularized Master Equations for given $n\ge 1$. For all $s>0$, there exists a constant $C=C(s)>0$ such that
$$
\sup_{n\ge 1} \sup_{t\in (0,1]} \frac 1n \sum_{i\in [n]} \rr_i(s,t)\ \le \ C\ . 
$$
A variance profile $V_n$ for which the previous estimate holds is called {\em admissible}.
\end{enumerate}
\begin{rem} Assumption \ref{ass:admissible} may seem obscure at first sight as it necessitates to solve the regularized master equations to check if a variance profile is admissible. In particular, it is not clear if this assumption is compatible with some sparsity. In Section \ref{sec:sufficient}, we provide sufficient conditions on the variance profile $V_n$ which imply \ref{ass:admissible}, namely \ref{ass:sigmin} (lower bound on $V_n$), \ref{ass:symmetric} (symmetric $V_n$) and \ref{ass:expander} (robust irreducibility for $V_n$). 
\end{rem}

\subsection{Statement of the results} 	 \label{results} 



Recall the Master Equations \eqref{def:ME2}, and notice that these equations are obtained from the Regularized Master Equations \eqref{def:MEt} by letting the parameter $t$ go to zero. Notice however that condition $\sum q_i=\sum \tilde q_i$ is now required for uniqueness and not a consequence as in \eqref{def:MEt}.



In order to prove existence of solutions $\q,\bqt$ to the Master Equations we need to assume the standard deviation profile $A_n$ is irreducible. 
This is equivalent to assuming the associated digraph $\Gamma(A_n)$ is strongly connected (recall the notations from Section \ref{sec:graph-notations}).
This will cause no loss of generality, as we can conjugate $Y_n$ by an appropriate permutation matrix to put $A_n$ in block-upper-triangular form with irreducible blocks on the diagonal -- the spectrum of $Y_n$ is then the union of the spectra of the block diagonal submatrices.

\vspace{.2cm}
\begin{theo} [Master equations]		\label{thm:master}	
Let $n\ge 1$ be fixed, let $A_n$ be an $n\times n$ nonnegative matrix and write $V_n=\frac1nA_n\odot A_n$. Assume that $A_n$ is irreducible. 
Then the following hold:
\begin{enumerate}
\item\label{q:sys1} 
For $s \geq \sqrt{\rho(V_n)}$ the system \eqref{def:ME2} has the unique solution $\qvec(s) = 0$.
\item\label{q:sys2} 
For $s\in (0,\sqrt{\rhoV_n})$ the system \eqref{def:ME2} has a unique non-trivial solution $\qvec(s) \posneq 0$. Moreover, this solution satisfies $\qvec(s)\succ 0$. 
\item\label{q:contdif} 
The function $s\mapsto \qvec(s)$ defined in parts (1) and (2) is continuous on 
$(0,\infty)$ and is continuously differentiable on 
$(0,\sqrt{\rhoV}) \cup (\sqrt{\rhoV},\infty)$. 
\end{enumerate} 
\end{theo}

\begin{rem}[Convention]
Above and in the sequel we abuse notation and write $\qvec=\qvec(s)$ to mean a solution of the equation \eqref{def:ME2}, understood to be the nontrivial solution for $s\in (0,\sqrt{\rho(V_n)})$. 
\end{rem}


The main result of this paper is the following.

\begin{theo}[Main result]	\label{thm:main}	
Let $(Y_n)_{n\ge1}$ be a sequence of random matrices as in Definition \ref{def:model}, and assume \ref{ass:moments}, \ref{ass:sigmax} and \ref{ass:admissible} hold. Assume moreover that $A_n$ is irreducible for all $n\ge 1$.
\begin{enumerate}
\item There exists a sequence of deterministic measures $(\mu_n)_{n\ge 1}$ on $\C$ such that
\[
\mu_n^Y \sim \mu_n\quad \mbox{ in probability}.
\]
\item Let $\qvec(s)^\tran = (\bq(s)^\tran\; \bqt(s)^\tran)$ be as in Theorem \ref{thm:master}, and for $s\in (0,\infty)$ 
put 
\begin{equation}	\label{expF}
F_n(s) = 1-\frac1n\langle \bq(s),V_n \bqt(s)\rangle.
\end{equation}
Then $F_n$ extends to an absolutely continuous function on $[0,\infty)$ which is the CDF of a probability measure with support contained in $[0,\sqrt{\rho(V_n)}]$ and continuous density on $(0,\sqrt{\rho(V_n)})$.
\item For each $n\ge1$ the measure $\mu_n$ from part (1) is the unique radially symmetric probability measure on $\C$ with $\mu_n(\{z:|z|\le s\})= F_n(s)$ for all $s\in (0,\infty)$.
\end{enumerate}
\end{theo}

\begin{rem}[Almost sure convergence under different hypotheses]	\label{rmk:almostsure}
As with Theorem \ref{thm:circ}, a key component of the proof is a lower tail estimate for the smallest singular value of scalar shifts $Y_n-zI_n$ of the form 
\begin{equation}	\label{smin.tail}
\pr(s_{n}(Y_n-zI_n)\le n^{-\beta}) = O(n^{-\alpha})
\end{equation}
holding for a.e.\ fixed $z\in \C$. (Crucially, we do not need such an estimate for \emph{every} $z\in \C$, as \ref{ass:admissible} only requires $s=|z|>0$ and allows variance profiles for which \eqref{smin.tail} is false when $z=0$.) 
Such a bound is available for arbitrary fixed $z\ne 0$ and $\alpha>0$ a small constant by recent work of the first author \cite{Cook:ssv}; see Proposition~\ref{prop:nick}.  
Obtaining \eqref{smin.tail} with $\alpha>1$ would immediately improve conclusion (1) to almost sure convergence by an application of the Borel--Cantelli lemma.
Such improvements 
are already available under stronger assumptions on $A_n$ and $X_n$.
For instance, under \ref{ass:sigmin} and replacing \ref{ass:moments} with a bounded density assumption, an easy argument gives \eqref{smin.tail} for any fixed $\alpha>0$ and some $\beta=\beta(\alpha)>0$; see \cite[Section 4.4]{2012-bordenave-chafai-circular}.
\end{rem}

%
%

\begin{rem}[Density of $\mu_n$ versus density of $F_n$]\label{rem:density-mu}
In the previous theorem, by the classical polar change of coordinates, the density $\varphi_n$ of $\mu_n$ for $0<|z|<\sqrt{\rho_V}$ is given by the formula:
\[
\varphi_n(|z|) \ =\ \frac 1{2\pi  |z|} \frac{d}{ds}  F_n(s) \Big|_{s=|z|}\ =\  - \frac 1{2\pi n |z|} \frac{d}{ds} \langle \bq(s),  V \bqt(s) \rangle \Big|_{s=|z|}\ .
\]
\end{rem}


As an illustration we show how our results recover the circular law for matrices with i.i.d.\ entries.

\begin{example}[The circular law]\label{example:circular}
Consider a standard deviation profile $A_n$ with all elements equal to 1 and assume that \ref{ass:moments} holds. It is well known in this case that $\mu_n^Y \xrightarrow[]{w} \mu_{\mathrm{circ}}$ in probability (and even almost surely), where $\mu_{\mathrm{circ}}$ stands for the circular law with density $\pi^{-1} 1_{\{ |z|\le 1\}}$. We can recover this result with Theorem \ref{thm:main}. 
In this case, both systems \eqref{def:MEt} and \eqref{def:ME2} simplify into a single equation:
\begin{equation}	\label{ME:scalar}
\rr_i \equiv r = \frac{r+t}{s^2 + (r+t)^2}\ ,\quad r>0\qquad \textrm{and} \qquad  q_i \equiv q = \frac{q}{s^2 + q^2}\ , \quad q \ge 0\ .
\end{equation}
From the first equation, one can prove that $r(s,t)\le 1$ for $t\in (0,1]$.  In fact,
$$
r = \frac{r+t}{s^2 + (r+t)^2} \le \frac 1{r+t} \quad \Rightarrow \quad r^2 + rt \le 1 \quad \Rightarrow \quad r^2 \le 1\ .
$$ 
Hence \ref{ass:admissible} is fulfilled. The second equation has the unique nontrivial solution
\begin{equation}	\label{ME:scalar.soln}
q(s) = \begin{cases} \sqrt{1-s^2} & 0\le s\le 1\\ 0 & s\ge 1 \end{cases}.
\end{equation}
Consequently, $F_n(s)= s^2$ for $s\le 1$. From Remark \ref{rem:density-mu}, we conclude the desired convergence.
\end{example}
In the next example, we prove that a doubly stochastic normalized variance profile is admissible and that the associated deterministic equivalent $\mu_n$ is the circular law.
\begin{example}[Doubly stochastic variance profile]\label{ex:bistochastic}
Assume that matrix $V_n$ is doubly stochastic, i.e.\  $\frac 1n \sum_i \sigma_{ij}^2 =\frac 1n \sum_j \sigma_{ij}^2=1$ for all $1\le i,j\le n$.
Then, one quickly verifies that the vectors $\rvec=r\bf{1}$ and $\qvec= q \bf{1}$ with $r,q$ as in \eqref{ME:scalar} respectively satisfy the Regularized Master Equations and the Master Equations. As a consequence \ref{ass:admissible} can be established as in Example \ref{example:circular}.
Let now  \ref{ass:sigmax} hold and assume that the variance profile $V_n$ is irreducible for all $n\ge 1$ then one can apply Theorem \ref{thm:main} with 
$\mu_n$ equal to the circular law. 
\end{example}

\begin{rem} Note that under \ref{ass:sigmax} the doubly stochastic condition implies that the number of non-zero entries in each row and column is linear in $n$. 
\end{rem}

In the following theorem, we relax the irreducibility assumption, which requires some additional argument.

\begin{theo}[The circular law for doubly stochastic variance profiles]  \label{thm:bistochastic} 
Let $(Y_n)_{n\ge1}$ be a sequence of random matrices as in Definition \ref{def:model}, and assume \ref{ass:moments} and \ref{ass:sigmax} hold. 
Suppose also that the normalized variance profiles $V_n$ are doubly stochastic, i.e.\  $\frac 1n \sum_i \sigma_{ij}^2 =\frac 1n \sum_j \sigma_{ij}^2=1$ for all $1\le i,j\le n$.
Then $\mu_n^Y \xrightarrow[]{w} \mu_{\textrm{circ}}$ in probability.
\end{theo} 

This parallels results of Girko \cite[\S 7.11, \S 8.2]{girko-canonical-equations-I} and Anderson and Zeitouni \cite{AnZe:band} (see also \cite{shlyakhtenko-96}, \cite{Guionnet:band} for the Gaussian case) for random Hermitian matrices with a doubly-stochastic variance profile which obey the Marchenko--Pastur or Wigner semi-circle laws.

%


\subsection{Sufficient conditions for admissibility}\label{sec:sufficient}


Hereafter we introduce a series of assumptions directly checkable over matrices $(V_n)$ without solving a priori the regularized master equations.
These assumptions enforce \ref{ass:admissible}.

The simplest such assumption is to enforce uniform positivity of the variances, which allows one to bypass some of the most technical portions of our argument. This assumption was also made in the recent work \cite{alt2016local}.

\vspace{.1cm}
\begin{enumerate}[leftmargin=*, label={\bf A3}]
\item\label{ass:sigmin}		
(Lower bound on variances).
There exists $\smin>0$ such that
\[
\inf_n \min_{1\le i,j\le n} \sigma_{ij}^{(n)} \ge \smin.
\]
\end{enumerate}
\vspace{.1cm}

We generalize \ref{ass:sigmin} below with the expansion-type condition \ref{ass:expander}.

\begin{prop}\label{prop:sigmin}
Let $A=(\sigma_{ij})$ be an $n\times n$ matrix with entries $\sigma_{ij}\ge \smin>0$ for some $\sigma>0$. 
Let $\rvec\succ 0$ be the unique solution of the Regularized Master Equations \eqref{def:MEt}. Then 
\[
\frac1n \sum_{i=1}^n \rr_i \le \frac1{\smin}.
\]
In particular, if $A_n=(\sigma_{ij}^{(n)})$ is a sequence of standard deviation profiles as in Definition \ref{def:model} for which \ref{ass:sigmin} holds, then \ref{ass:admissible} is satisfied, i.e. $V_n$ is admissible.
\end{prop}

\vspace{.1cm}
\begin{enumerate}[leftmargin=*, label={\bf A4}]
\item\label{ass:symmetric}		
(Symmetric variance profile). For all $n\ge 1$, the normalized variance profile (or equivalently the standard deviation profile) is symmetric: 
$$V_n= V_n^\tran\ .$$
\end{enumerate}
\vspace{.1cm}

\begin{prop}\label{prop:symmetry}
Let $A=(\sigma_{ij})$ be a symmetric matrix with nonnegative entries, and let $\rvec\succ 0$ be the unique solution of the Regularized Master Equations \eqref{def:MEt}. Then 
\[
\frac1n \sum_{i=1}^n \rr_i \le \frac1{2s}.
\]
In particular, if $A_n=(\sigma_{ij}^{(n)})$ is a sequence of standard deviation profiles as in Definition \ref{def:model} for which \ref{ass:symmetric} holds, then \ref{ass:admissible} is satisfied.
\end{prop}

We now introduce the following strengthening of the irreducibility assumption, which can be understood as a kind of expansion condition on an associated directed graph. 

\begin{Def}[Robust irreducibility]	\label{def:robust}
For $\delta,\kappa\in (0,1)$ we say that a nonnegative $n\times n$ matrix $A$ is \emph{$(\delta,\kappa)$-robustly irreducible} if the following hold:
\begin{enumerate}
\item For all $i\in [n]$,
\begin{equation}	\label{mindegree}
\big|\mN_{A}(i)\big|,\,\big|\mN_{A^\tran}(i)\big| \ \ge\  \delta n.
\end{equation}
\item For all $S\subset[n]$ with $1\le |S|\le n-1$,
\begin{equation}	\label{lb:expand}
\big|\mN_{A^\tran}^{(\delta)}(S) \cap S^c\big|\ \ge\  \min(\kappa|S|, |S^c|).	
\end{equation}
\end{enumerate}
\end{Def}

For comparison, a nonnegative $n\times n$ matrix $A$ is irreducible if and only if $\mN_{A^\tran}(S)\cap S^c\ne \emptyset$ for all $S\subset [n]$ with $1\le |S|\le n-1$.
Thus, a matrix $A$ satisfying the conditions of Definition \ref{def:robust} is ``robustly irreducible" in the sense that $A$ remains irreducible even after setting a small linear proportion of entries equal to zero.

\begin{rem}[Relation to broad connectivity]	\label{rmk:broad}
In their work on permanent estimators, Rudelson and Zeitouni assume a stronger expansion-type condition on $A(\scut)$ which they call \emph{broad connectivity} \cite{rudelson-zeitouni-2016}. 
The conditions for a nonnegative square matrix $A$ to be $(\delta,\kappa)$-broadly connected are the same as in Definition \ref{def:robust}, except \eqref{lb:expand} is replaced by the stronger condition
\begin{equation}	\label{def:broad}
\big| \mN_{A^\tran}^{(\delta)}(S)\big|  \ge \min(n,(1+\kappa)|S|)
\end{equation}
for all nonempty $S\subset[n]$.

\end{rem}
Recall the definition \eqref{def:Acut} of $A(\sigma_0)$.

\vspace{.1cm}
\begin{enumerate}[leftmargin=*,label={\bf A5}]
\item\label{ass:expander}		
(Robust irreducibility).
There exist constants $\scut,\delta,\kappa\in (0,1)$ such that for all $n\ge 1$, $A_n(\scut)$ is $(\delta,\kappa)$-robustly irreducible.
\end{enumerate}
\vspace{.1cm}

Note that \ref{ass:expander} enables variance profiles with a large proportion of vanishing entries and is implied by \ref{ass:sigmin}.

\begin{theo}\label{th:sufficient} Consider a sequence of standard deviation profiles $A_n=(\sigma_{ij}^{(n)})$ as in Definition \ref{def:model}, and assume that \ref{ass:sigmax} and \ref{ass:expander} hold. Then \ref{ass:admissible} holds, i.e.\ $V_n$ is admissible. 
\end{theo}

It turns out that the mere irreducibility of $V_n$ provides a weaker form of \ref{ass:admissible}. 

\begin{prop}\label{prop:qualitative} Let $V_n$ be an irreducible variance profile and let $\rvec=\rvec(s,t)$ be the solution of the associated Regularized Master Equations \eqref{def:MEt}. Then there exists $C=C(s,n)$ such that 
$$
\sup_{t\in (0,1]}\ \frac 1n \sum_{i\in [n]} r_i(s,t) \ \le \ C\ .
$$
\end{prop}
The main difference here is that constant $C$ depends on $n$ and may blow up with $n$. Depending on the variance profile, this proposition is sometimes sufficient to verify \ref{ass:admissible}.

\begin{example}[Variance profile with a block structure]	\label{example:blockVP}
Let $k\ge 1$ be a fixed integer, and $M=(m_{ij})_{i,j\in [k]}$ be a $k\times k$ irreducible matrix with nonnegative elements. 
Let $J_m=\bs{1}_m \bs{1}_m^\tran$.
Assume that $n = km$ ($m\ge 1$) and consider the $n\times n$ matrix
\begin{equation}	\label{Asingular}
V_n  = \frac 1n \begin{pmatrix} m_{11} J_m & \cdots  & m_{1k} J_m   \\ \vdots \\ m_{k1} J_m &\cdots & m_{kk} J_m
\end{pmatrix} . 
\end{equation} 
Then $V_n$ is admissible, i.e.\ \ref{ass:admissible} is fulfilled. In fact, $V_n$ is irreducible and its block structure implies that  
$$
\boldsymbol{r}^\tran=(\underbrace{\rho_1,\cdots, \rho_1}_{m\ \textrm{times}}, \cdots, \underbrace{\rho_k,\cdots, \rho_k}_{m\ \textrm{times}})\ ,\quad 
\boldsymbol{\tilde r}^\tran=(\underbrace{\tilde \rho_1,\cdots, \tilde \rho_1}_{m\ \textrm{times}}, \cdots, \underbrace{\tilde \rho_k,\cdots, \tilde \rho_k}_{m\ \textrm{times}})
$$
where $\boldsymbol{\rho}=(\rho_i)$ and $\boldsymbol{\tilde \rho}=(\tilde \rho_i)$ satisfy the $2k$ equations
$$
\rho_i = \frac{( M_k^\tran \boldsymbol{\rho})_i +t}{s^2 + ((M_k^\tran \boldsymbol{\rho})_i +t) (M_k \boldsymbol{\tilde \rho})_i +t)}\,,\quad 
\tilde \rho_i = \frac{(M_k \boldsymbol{\tilde \rho})_i +t}{s^2 + ((M_k^\tran \boldsymbol{\rho})_i +t) (M_k \boldsymbol{\tilde \rho})_i +t)}\,,\quad i\in[k]
$$
with $M_k = \frac 1k M$. In particular, 
$$
\sup_{t\in(0,1]} \frac 1n \sum_{i\in [n]} r_i(s,t) = \sup_{t\in(0,1]} \frac 1k \sum_{i\in [k]} \rho_i(s,t)\, ,
$$
where the latter is finite by Proposition \ref{prop:qualitative} and does not depend on $n$, hence \ref{ass:admissible}.
\end{example}

We now state two general classes of standard deviation profiles satisfying \ref{ass:expander} that include many examples not covered by \ref{ass:sigmin}, \ref{ass:symmetric} or Example \ref{example:blockVP}.
The classes are in a similar spirit, essentially saying that $A=(\sigma_{ij})$ is approximately controlled from below by a block matrix whose nonzero blocks form an irreducible ``pattern", and which are ``regular" in the sense that the nonzero entries in the blocks are uniformly distributed in a certain sense.
In what follows we write
\[
e_A(S,T)= \sum_{i\in S,j\in T} 1_{\sigma_{ij}>0}, \qquad \rho_A(S,T)= \frac{e_A(S,T)}{|S||T|}.
\]
We recall a common assumption from the extremal combinatorics literature (see for instance \cite[Definition 1.6]{KoSi:survey}).
For $\eps,\rho_0\in (0,1)$, we say that an $m\times n$ matrix $A$ is \emph{$(\eps,\rho_0)$-super-regular} if for every $S\subset[m], T\subset[n]$ with $|S|\ge \eps m, |T|\ge \eps n$ we have $\rho_A(S,T) \ge \rho_0$. 
Informally, this says that all sub-matrices of $A$ of relative size $\eps$ have density (proportion of non-zero entries) at least $\rho_0$. 

\begin{Def}[Block-pseudorandom irreducibility]
For $\delta,\eps,\rho_0\in (0,1)$ and $K\in \N$, we say that a non-negative $n\times n$ matrix $A$ is \emph{$(\delta, K, \eps, \rho_0)$-block-pseudorandomly irreducible}, or \emph{$(\delta, K, \eps, \rho_0)$-BPI} for short, if \eqref{mindegree} holds, and if there is a partition of $[n]$ into sets $T_1,\dots, T_K$ with $T_q\ge n/(2K)$ for each $1\le q\le K$, and a $K\times K$ irreducible 0/1 matrix $M=(m_{pq})$, such that for each $(p,q)\in [K]^2$ with $m_{pq}=1$, the sub-matrix $A_{T_p,T_q}$ is $(\eps,\rho_0)$-super-regular. 
\end{Def}

Thus, the BPI condition is met if we can partition $[n]$ into a bounded number of parts of roughly equal size, such that the partitioned matrix contains an irreducible ``pattern" of super-regular blocks. 
We note the ways in which the assumption that $A(\sigma_0)$ is BPI generalizes Example \ref{example:blockVP}:
\begin{enumerate}
\item The blocks do not need to be the same size.
\item The matrix $A$ does not have to be constant within blocks (and can in fact have a large proportion of zero entries in every block, as well as an arbitrary proportion of entries larger that $\sigma_0$).
\item $A$ is arbitrary outside the sub-matrices $A_{T_p,T_q}$ with $m_{pq}=1$ (whereas in Example \ref{example:blockVP} the entries of $A$ were zero there).
\end{enumerate}
We also note that this assumption generalizes the \emph{block fully indecomposable} assumption from \cite{AEKquadequations} (see Definition 2.9 there), which makes a stronger assumption on the connectivity pattern $M$ than irreducibility (that it is \emph{fully indecomposable}) and also requires entries to be uniformly bounded below in blocks for which $m_{pq}=1$.

\begin{lemma}	\label{lem:BPItoRI}
Let $A=(\sigma_{ij})$ be an $n\times n$ non-negative matrix that is $(\delta, K, \eps, \rho)$-BPI for some $\delta,\rho\in (0,1)$, $\eps\in (0,\delta/8)$, and $K\in \N$. Then $A$ is $(\delta_0, \kappa)$-robustly irreducible with $\delta_0=\min(\delta/2, \rho/(4K))$ and $\kappa = \min(\delta/4,1/(8K))$. 
\end{lemma}

\begin{proof}
The condition \eqref{mindegree} is immediate. 
For \eqref{lb:expand}, we first consider the case that $|S|\le \delta n/4$. From \eqref{mindegree} we have
\[
\delta n |S|\le \sum_{i=1}^n \sum_{j\in S} 1_{\sigma_{ij}>0} = \sum_{i=1}^n |\mN_A(i)\cap S| \le |\mN^{(\delta/2)}_{A^\tran}(S)||S| + n(\delta/2)|S|
\]
which rearranges to give
\[
|\mN^{(\delta/2)}_{A^\tran}(S)\cap S^c| \ge |\mN^{(\delta/2)}_{A^\tran}(S)| - |S| \ge \delta n/2-|S|\ge \delta n/4 
\]
whenever $|S|\le \delta n/4$. Since $n\ge |S|$ we conclude \eqref{lb:expand} holds in this case with $\kappa = \delta/4$. 

Assume now that $|S|>\delta n/4$. Let $G=\{q\in [K]: |S\cap T_q|\ge \frac 12|S||T_q|/n\}$. 
From the pigeonhole principle we know that $G$ is nonempty. Suppose first that $G\ne [K]$. 
Then since $M$ is irreducible, there exists $(p,q)\in G^c\times G$ such that $m_{pq}=1$. 
Now since $|S\cap T_q| \ge |S||T_q|/(2n) \ge (\delta/8)|T_q|$, and since $\eps<\delta/8$, we deduce from the super-regularity of $A_{T_p,T_q}$ that $|\mN_A(i)\cap (S\cap T_q)|\ge \rho|S\cap T_q|$ for at least $(1-\eps)|T_p|$ values of $i\in T_p$. Indeed, supposing otherwise, if $T_p'$ is the set of $i\in T_p$ for which this lower bound fails then we have $e_A(T_p', S\cap T_q) = \sum_{i\in T_q} |\mN_A(i)\cap (S\cap T_q)|<\rho |T_p'||S\cap T_q|$, which is only possible if $|T_p'|<\eps|T_p|$. Thus, we have 
\[
|\mN_A(i)\cap S| \ge |\mN_A(i)\cap (S\cap T_q)|\ge \rho|S\cap T_q| \ge \rho|S||T_q|/(2n) \ge \frac{\rho}{4K}|S|
\]
for at least
\[
(1-\eps)|T_p| - |S\cap T_p| \ge \left( 1-\eps - \frac{|S|}{2n}\right)|T_p| \ge \left(\frac12-\eps\right)|T_p| \ge \frac{1}{2K}  \left(\frac12-\eps\right)n \ge \frac1{8K}n
\]
values of $i\in S^c$, where in the penultimate inequality of the former display we used that $q\in G$, and in the first inequality of the latter display we used that $p\in G^c$. The claim follows in this case.

It remains to handle the case that $|S|>\delta n/4$ and $G=[K]$. In this case, repeating the lines above shows that $|N_A(i)\cap S|\ge \frac\rho{4K}|S|$ for at least $(1-\eps)|T_p|$ values of $i\in T_p$ for every $1\le p\le K$ (because for every $p\in [K]$ we must have $m_{pq}=1$ for at least one value of $q$). 
We hence have $|\mN^{(\rho/(4K))}_{A^\tran}(S)|\ge (1-\eps)n$. But since $\eps< \delta/8$ and $|S|>\delta n/4$ we conclude that $|\mN^{(\rho/(4K))}_{A^\tran}(S)\cap S^c|\ge \delta n/4\ge (\delta/4)|S|$ as desired. 
\end{proof}

Our next sufficient condition for robust irreducibility involves notions from the theory of graph limits, in particular the concept of a graphon. For further background we refer to \cite{Lovasz:book}. 
For our purposes, let us say that a \emph{graphon} is a (Lebesgue) measurable function $W:[0,1]^2\to [0,1]$.
We equip the set of graphons with the cut norm:
\[
\|W\|_\Box = \sup_{S,T\subseteq[0,1]} \left|\int_{S\times T} W(x,y) dxdy\right|,
\]
where the integral is taken with respect to product Lebesgue measure, and the supremum is taken over all measurable subsets of $[0,1]$. 
To a non-negative $n\times n$ matrix $A=(\sigma_{ij})$ we associate the graphon $W_A$ which is equal to $\sigma_{ij}$ on the set $[\frac{i-1}{n}, \frac{i}{n})\times [\frac{j-1}{n}, \frac{j}{n})$ for each $1\le i,j\le n$ (and set $W_A(x,1)=W_A(1,x)\equiv0$).

We say that a graphon $U$ is a \emph{stepfunction} if there is a measurable partition $\mathcal{P}=\{S_1,\dots, S_K\}$ of $[0,1]$ and a $K\times K$ matrix $M=(m_{pq})$ such that $U(x,y)=m_{pq}$ for all $p,q\in [K]$ and all $(x,y)\in S_p\times S_q$. 
We call $\mathcal{P}, M$ the \emph{partition} and \emph{pattern}, respectively, associated with $U$.

\begin{Def}
Let $K\in \N$, and $\alpha,\sigma_*\in (0,1)$. Say that a graphon $W$ is \emph{$(K,\alpha,\sigma_*)$-irreducible} if there exists a stepfunction $U$ with $W\ge U$ pointwise, and such that the following hold for the partition $\mathcal{P}=\{S_1,\dots, S_K\}$ and pattern $M$ associated with $U$:
\begin{enumerate}
\item $\Leb(S_p)\ge \alpha/K$ for all $p\in [K]$.
\item $m_{pq}\in \{0\}\cup[\sigma_*,1]$ for all $p,q\in [K]$.
\item $M$ is irreducible. 
\end{enumerate}
\end{Def}

\begin{lemma}	\label{lem:graphon}
Let $A=(\sigma_{ij})$ be a non-negative $n\times n$ matrix and let $K\in N$ and $\delta,\alpha,\sigma_*\in (0,1)$. Assume \eqref{mindegree} holds, and that there is a $(K,\alpha,\sigma_*)$-irreducible graphon $W$ such that 
$
\|W_A-W\|_\Box \le \frac{\delta^2 \alpha^2\sigma_*}{32 K^2}=:\delta_0.
$
Then $A$ is $(\delta_0/2,\delta_0\delta/4)$-robustly irreducible. 
\end{lemma}

\begin{proof}
Arguing as in the proof of Lemma \ref{lem:BPItoRI} it suffices to verify the condition \eqref{lb:expand} for $|S|\ge \delta n/4$. If $|S|\ge (1-\delta/2)n$ then from \eqref{mindegree} it follows from the triangle inequality that $|\mN_A(i)\cap S|\ge \delta n/2 $ for all $i\in [n]$, and \eqref{lb:expand} follows in this case.
We henceforth assume $|S^c|\ge \delta n/2$.
We have
\[
\frac1{n^2}e_A(S^c,S)= \int_{\hat{S}\times \hat{S}^c} W_A(x,y)dxdy
\ge \int_{\hat{S}\times \hat{S}^c} W(x,y)dxdy -\delta_0
\ge  \int_{\hat{S}\times \hat{S}^c} U(x,y)dxdy -\delta_0,
\]
where for $S\subset[n]$ we denote the dilation $\hat{S}:=\frac1n\cdot S$.
Let $L\subset [K]$ such that $1\le |L|\le K-1$, and such that $\Leb(\hat{S}\cap S_p)\ge (\delta/4)\Leb(S_p)$ for all $p\in L$ and $\Leb(\hat{S}^c\cap S_p)\ge (\delta/4)\Leb(S_p)$ for all $p\in L^c$. It is possible to choose $L$ such that both $L$ and $L^c$ are nonempty since $|\hat{S}|, |\hat{S}^c|\ge \delta/4$. Since $M$ is irreducible we must have $m_{p'q'}\ge \sigma_*$ for some $(p,q)\in L\times L^c$. Now we have
\[
 \int_{\hat{S}\times \hat{S}^c} U(x,y)dxdy \ge \sum_{p, q=1}^K m_{p,q} \Leb(\hat{S}\cap S_p) \Leb(\hat{S}^c\cap S_q) \ge \sigma_* (\delta/4)^2\Leb(S_{p'})\Leb(S_{q'})
 \ge \frac{\delta^2\alpha^2\sigma_*}{16K^2} = 2\delta_0.
\]
Combined with the previous display, we conclude $e_A(S^c,S)\ge \delta_0n^2$. On the other hand,
\[
e_A(S^c,S) = \sum_{i\in S^c} |\mN_A(i)\cap S| \le |S||\mN_{A^\tran}^{(\delta_0/2)}(S)| + \frac{\delta_0}2|S||S^c|
\]
which rearranges to $|\mN_{A^\tran}^{(\delta_0/2)}(S)|\ge (\delta_0/2)|S^c|\ge (\delta \delta_0/4)|S|$, as desired.
\end{proof}

\subsection{Outline of the proof}
\label{outline} 

As is well known in the literature devoted to large non-Hermitian random
matrices, the spectral behavior of such matrices can be studied with the help
of the so-called Girko's Hermitization procedure, which is intimately related
with the \emph{logarithmic potential} of their spectral measure
\cite{girko1985circular}. By definition, the logarithmic potential $U_\mu$ of a
probability measure $\mu$ on $\C$ which integrates $\log|\cdot|$ near infinity is the
function from $\C$ to $(-\infty,\infty]$ defined by
\[
U_\mu(z) = -\int_{\C} \log|\lambda-z|\, \mu(d\lambda)\, .
\]
Writing $z=x+\ii y$ the Laplace operator is 
$
\Delta = \partial_{xx}^2 + \partial_{yy}^2  = 4 \, \partial_z \, \partial_{\bar{z}}
$
where $\partial_z = \frac 12\left(
\partial_x -\ii \partial_y\right) $ and $\partial_{\bar{z}}   =  \frac 12\left(
{\partial_x} +\ii \partial_y \right)$. The probability measure $\mu$ can be recovered by the formula
$
\mu =  -\frac 1{2\pi} \Delta U_{\mu}\ ,
$ 
valid in the set $\mathcal D'(\C)$ of the Schwartz distributions, which means 
that 
\[
\int \psi(z) \, \mu(dz)  =  
- \frac 1{2\pi} \int_{\C} \Delta \psi(z) \, U_\mu(z) \, \ell(dz) \qquad \textrm{for all}\quad \psi \in C_c^\infty(\C)\ .
\]
Now, setting $\mu = \mu^Y_n$, the logarithmic potential can be written as 
\begin{align*} 
U_{\mu_n^Y}(z) &= - \frac 1n \sum_{i=1}^n \log|\lambda_i -z| 
= -\frac 1n \log |\det(Y_n-z)|
= -\frac 1{n} \log\sqrt{\det (Y_n-z)(Y_n-z)^*} \\
&=  - \int_0^\infty \log (x) \,  L_{n,z}(dx) \, , 
\end{align*} 
where 
$
L_{n,z} := \frac 1n \sum_{i=1}^n \delta_{s_{i,z}} 
$
is the empirical distribution of the singular values $s_{1,z}\ge \cdots\ge s_{n,z}\ge 0$ of
$Y_n-z$. 
For technical reasons, it is easier to consider the \emph{symmetrized} empirical distribution of the singular values 
\begin{equation} \label{eq:symsingularvalues}
\check{L}_{n,z}\ :=\ \frac 1{2n} \sum_{i=1}^n \delta_{-s_{i,z}} + \frac 1{2n} \sum_{i=1}^n \delta_{s_{i,z}}
\end{equation}
for which a similar identity holds:
\[
U_{\mu_n^Y}(z) = - \int_{\R} \log|x|  \,  \check{L}_{n,z}(dx) \, . 
\]
This identity is at the heart of Girko's strategy. In a word, it shows
that in order to evaluate the asymptotic behavior of the spectral distribution
$\mu_n^Y$, we can focus on the asymptotic behavior of $\check{L}_{n,z}$ for almost 
all $z \in \C$. By considering $\check{L}_{n,z}$, we are in the more familiar world of 
Hermitian matrices. Informally, for all $z\in\C$, we will find a sequence 
$(\check \nu_{n,z})_{n\in\N}$ of deterministic probability measures such that 
$\check L_{n,z} \sim \check \nu_{n,z}$, and 
\[
\int_\R \log |x| \,  \check L_{n,z}(dx) \approx 
\int_\R \log |x| \,  \check \nu_{n,z}(dx)  \quad \text{for large } n \, .
\]
Setting \[h_n(z) := - \int_\R \log |x| \, \check \nu_{n,z}(dx),\] gives  
that for all $\psi\in C_c^\infty(\C)$, 
\[
\int \psi(z) \, \mu^Y_n(dz) \approx 
- \frac 1{2\pi} \int_{\C} \Delta \psi(z) \, h_n(z) \, \ell(dz)  
  \quad \text{for large } n \, , 
\]
showing that $h_n(z)$ is the logarithmic potential of a probability measure 
$\mu_n$, and in particular that $\mu_n^Y\sim \mu_n$. Further smoothness properties of $h_n(z)$ will finally yield the properties of 
$\mu_n$ stated in Theorem~\ref{thm:main}. 

We provide more details hereafter with precise pointers to the article's results.

\subsubsection*{1. Study of the associated Hermitian model} 

This topic is covered in Section \ref{sec:svd}.

Given $z\in\C$, we establish the existence of a sequence 
$(\check \nu_{n,z})_{n\in\N}$ of deterministic probability measures such that 
$\check L_{n,z} \sim \check \nu_{n,z}$ almost surely. To this end, we 
introduce the $2n\times 2n$ Hermitian matrix 
\begin{equation}	\label{Y2n}
\Y_n^z := \begin{pmatrix} 
0 & Y-z \\
Y^* -z^* & 0
\end{pmatrix} ,\quad z\in \C , 
\end{equation}
with spectral measure $\check L_{n,z}$ \cite[Theorem~7.3.7]{book-horn-johnson}. 
The asymptotic analysis of $\check L_{n,z}$ relies on the \emph{resolvent}:
\begin{equation}	\label{def:res0}
\Res_n(z,\eta):= \begin{pmatrix} 
-\eta & Y-z\\
Y^* - z^* & -\eta
\end{pmatrix}^{-1} , \quad \eta \in \C_+ , 
\end{equation}
of $\Y_n^z$. The mere definition \eqref{def:Stieltjes} of a Stieltjes transform yields
\[
g_{\check L_{n,z}}(\eta)  =  \frac 1{2n} \tr \Res_n(z,\eta)\,.
\]

The rigorous use of the Stieltjes transform for the study of ESDs of Hermitian random matrices goes back to Pastur \cite{Pastur73}, and was further developed by Bai to obtain quantitative results \cite{Bai93a, Bai93b}. Beginning with the seminal works \cite{ESY:local2, ESY:local1} of Erd\H{os}, Schlein and Yau this approach has been used to show that the semicircle law governs the spectral distribution for Wigner matrices down to near-optimal scales. In these works, the basic strategy is to use resolvent identities to show that the Stieltjes transform approximately satisfies a fixed-point equation, sometimes called the \emph{Schwinger--Dyson} (or \emph{master-loop}) equation. This approach was extended to Hermitian matrices with doubly stochastic variance profile \cite{EYY:generalizedWigner}. 
However, for Hermitian matrices with more general variance profiles and non-zero mean it becomes necessary to consider a \emph{system} of equations that are approximately satisfied by individual diagonal entries of the resolvent.

In Section~\ref{sec:svd} we derive the Schwinger--Dyson equations for our setting, cf.\ Proposition \ref{prop:deteq}:
\begin{equation}	\label{cve0}
\begin{cases}
\,p_i &=\, \dfrac{(V_n^\tran\bs{p})_i +\eta}{|z|^2-((V_n\bs{\tilde p})_i+\eta) ((V_n^\tran\bs{p})_i +\eta) }  \\
\,\tilde p_i &=\, \dfrac{(V_n\bs{\tilde p})_i +\eta}{|z|^2-((V_n\bs{\tilde p})_i+\eta) ((V_n^\tran\bs{p})_i +\eta) }  \\
\end{cases}
\qquad \bs{p}=( p_i)\, , \ \bs{\tilde p}=( \tilde p_i )\, ,
\end{equation}
for $\eta\in \C_+$,
with unique solution the $2n\times 1$ vector $\bs{\vec p}=(\bs{p}\; \bs{\tilde p})$ satisfying $\Imm \bs{\vec p} \succ 0$. We prove that there exists a probability distribution $\check \nu_{n,z}$ whose Stieltjes transform $g_{\check\nu_{n,z}}$ is defined as  $g_{\check\nu_{n,z}}=\frac 1n \sum_{i\in [n]} p_i$.

In Theorem \ref{th:L->nu}, it is established that for all $\eta\in\C_+$, 
$
g_{\check L_{n,z}}(\eta) - g_{\check\nu_{n,z}}(\eta) \to 0$ almost surely, which in particular implies that $\check L_{n,z} \sim \check \nu_{n,z}$ a.s.
In Proposition \ref{prop:QR} we obtain a quantitative estimate 
along the imaginary axis
of the form
\begin{equation}\label{eq:conv-rate}
\E g_{\check L_{n,z}}(\ii t) - g_{\check\nu_{n,z}}(\ii t)= {\mathcal O} \left(\frac 1{t^c\, \sqrt{n}}\right) ,\quad t>0,
\end{equation}
for some integer $c$.

We note that recently there has been much work analyzing the Schwinger--Dyson equations corresponding to Hermitian random matrices with mean and variance profiles satisfying a range of assumptions. 
For the centered case one is led to the so-called \emph{Quadratic Vector Equation}, which has been thoroughly analyzed in works of Ajanki, Erd\H os and Kr\"uger and Alt, Erd\H os and Kr\"uger \cite{AEKquadequations,AEKgram}; the application to universality for local spectral statistics was carried out in \cite{ajanki2015universality}.
Very recently they have made the extension to matrices with correlated entries, which involves the study of the so-called \emph{Matrix Dyson Equation} \cite{AEK16:mde}.
In another recent work, He, Knowles and Rosenthal prove an approximate (matrix-valued) self-consistent equation for resolvents of Hermitian random matrices with arbitrary mean and variance profile, which covers the structure of the model \eqref{Y2n} \cite{HKR16:local}. However, they assume the entries have all moments finite, and their aim is to obtain a local law at the optimal scale. In the present work, a sub-optimal quantitative analysis of the system \eqref{cve0} under few moments will suffice for our purposes of understanding the spectrum of the associated non-Hermitian model $Y_n$ at global scale.

\subsubsection*{2. From the spectral measures $\check L_{n,z} \sim \check \nu_{n,z}$ to the spectral measures $\mu_n^Y\sim \mu_n$ via the associated logarithmic potentials} 

This topic is covered in Sections \ref{sec:svd} (partly) and \ref{sec:logpotential}.

The fact that $\check L_{n,z} \sim \check \nu_{n,z}$ a.s.\ does not 
ensure that the random logarithmic potential $U_{\mu^Y_n}(z)$ becomes close to the deterministic logarithmic potential $h_n(z)$ (assuming the latter is
well defined). 
Essentially, this is due to the fact that $x\mapsto \log|x|$ is unbounded 
near zero and infinity. 
While the singularity at infinity is easily handled using the almost sure tightness of the measures $\check L_{n,z}$, the singularity at zero presents a major technical challenge (indeed, this hurdle was the reason it took decades to establish the circular law under the optimal hypotheses).
We show that under the admissibility assumption \ref{ass:admissible},
$x\mapsto \log|x| $ is $\check \nu_{n,z}$-integrable, and that 
for all $\tau,\tau' > 0$, there is $\varepsilon > 0$ small enough such that
\begin{equation} 
\label{0epsilon}
\PP\left\{ \Bigl| \int_{-\varepsilon}^\varepsilon 
  \log|x| \,  \check L_{n,z}(dx) \Bigr| 
 > \tau \right\}   < \tau' \quad \text{for all large } n .
\end{equation} 
Together with the almost sure tightness and weak convergence $\check L_{n,z} \sim \check \nu_{n,z}$, we can show
that $U_{\mu^Y_n}(z) - h_n(z) \to_n 0$ in probability. The almost sure convergence is an open problem not covered in this article.

The proof of~\eqref{0epsilon} is based on two ingredients. The first one is a result from \cite{Cook:ssv} by the first author giving a lower tail estimate for the smallest singular value of $Y_n-z$
for arbitrary fixed $z\in\C\setminus\{0\}$ under the sole assumption \ref{ass:sigmax} on the 
standard deviation profile $A_n$.
For the second result, established in Section \ref{sec:logpotential}, we show that 
there exist two constants $C,\gamma_0 > 0$ such that for all $x > 0$, 
\[
\E \check L_{n,z}((-x,x)) \ \leq\  C (x \vee n^{-\gamma_0})\ .
\]
Bounds of this form on the expected density of states for random Hermitian operators are sometimes referred to as \emph{local Wegner estimates} (after the work \cite{Wegner:bounds}).
Their application to the convergence of the empirical spectral distribution for non-Hermitian random matrices goes back to Bai's proof of the circular law \cite{Bai-circular-law-1997}; our presentation of the argument is closer to the one in \cite{MR2831116}.

Such an estimate follows from control on $\Imm \E g_{\check L_{n,z}}(\ii t)$ for 
small $t > 0$, is obtained in two steps. First, sufficient control of 
$\E g_{\check L_{n,z}}(\ii t) - g_{\check \nu_{n,z}}(\ii t)$ is already 
provided by the estimate \eqref{eq:conv-rate}.

Second, we rely on \ref{ass:admissible} to state that $g_{\check \nu_{n,z}}(\ii t)$ is bounded independent of $n$ and $t$. For this task, a variation of the Schwinger--Dyson equations, namely the {\em Regularized Master Equations}, introduced in Proposition \ref{prop:MEt}, 
is obtained simply by setting $\eta=\ii t$ in the Schwinger--Dyson equations \eqref{cve0} -- in this case, $\pvec\in \ii \R^{2n}$ and $\rvec=({\br}^\tran\ \brt^\tran)^\tran$ is defined as $\rvec(t) = \Imm \pvec(\ii t)$. Hence by \ref{ass:admissible},
\begin{equation}	\label{bdd:sec2}
\Imm g_{{\check \nu}_{n,z}}(\ii t) = \frac 1n \sum_{i\in [n]} \rr_i(t) \le C
\end{equation}
for some $C<\infty$ independent of $n$ and $t$ (depending only on $|z|$ and the parameters in our assumptions). 


\subsubsection*{3. Description of the deterministic probability measure $\mu_n$} This is covered in Sections \ref{sec:master} and \ref{sec:end}.

We have proved so far that $\mu^Y_n \sim \mu_n$ in probability, where
$\mu_n$ is the probability measure whose logarithmic potential $U_{\mu_n}(z)$
coincides with $h_n(z)$. It
remains to establish the properties of $\mu_n$ that are stated by
Theorem~\ref{thm:main}. 

In Section \ref{sec:master} we prove Theorem \ref{thm:master}.
Our approach to obtaining the solution $\qvec(s)$ of \eqref{def:ME2} is through the Regularized Master Equations \eqref{def:MEt}.
Since these equations are obtained by a simple transformation of the Schwinger--Dyson equations \eqref{cve0}, by our work in Section \ref{sec:svd} we know that \eqref{def:MEt} has a unique solution $\rvec(s,t)$ satisfying $\rvec(s,t)\succ 0$, where we write $s=|z|$.
We then show that the pointwise limit $\rvecstar(s) := \lim_{t\downarrow0}\rvec(s,t)$ exists, and moreover that $\qvec=\rvecstar$ and is the unique solution to \eqref{def:ME2}.
Having properly defined $\qvec(s)$, our main task is to show that the distribution $-(2\pi)^{-1}
\Delta h_n(z)$ in fact defines a density on the set 
${\mathcal D} := \{ z \in \C \, : \, |z| \neq 0, \ |z| \neq \sqrt{\rho(V_n)} \}$
and to provide an expression for this density. The general approach towards
solving this problem can be found in the physics literature 
(see \cite{FeZe97}). Define on $\C \times (0,\infty)$ the functions  
\begin{equation}\label{cU} 
\cU_{n}^{\Y}(z,t) := 
- \frac{1}{2n} \log\det ( (Y-z)^* (Y-z) + t^2 ) \qquad \text{and} \qquad 
\cU_n(z,t) := 
- \frac 12 \int_{\R} \log(x^2 + t^2)  \check{\nu}_{n,z}(dx)\,  . 
\end{equation}
For fixed $t > 0$, these functions can be seen as regularized versions of the 
logarithmic potentials $U_{\mu^Y_n}(z)$ and $U_{\mu_n}(z)$ respectively, which
converge back as $t\downarrow 0$, in $\mathcal D'(\C)$. 

On the other hand, let us consider again the resolvent $\Res_n(z,\eta)$ introduced in \eqref{def:res0}. Using the well-known formula for 
the inverse of a partitioned matrix \cite[\S 0.7.3]{book-horn-johnson} and writing  
\begin{equation} \label{def:FGdef}
\Res_n(z,\eta) =: \begin{pmatrix}
G_n(z,\eta) & F_n(z,\eta)\\
F_n'(z,\eta)&\widetilde G_n(z,\eta)
\end{pmatrix} , 
\end{equation}
we get that by setting $\eta = \ii t$, that $\partial_{\bar{z}} \cU^{Y}_{n}(z,t)$ coincides with 
$n^{-1} \tr F_n(z,\ii t)$. Relying on the asymptotic analysis made in 
Section~\ref{sec:svd} on the resolvent $\Res_n$, we can easily obtain an 
expression for $\partial_{\bar{z}} \cU_n(z,t)$ by considering the asymptotic 
behavior of $n^{-1} \tr F_n(z,\ii t)$. We then conclude by studying the 
equation 
\[
\Delta U_{\mu_n} = 
4 \lim_{t\downarrow 0} \partial_z \partial_{\bar{z}} \cU_n(z,t).
\]
Section \ref{sec:density} is devoted to these questions. 

In Section \ref{sec:bistochastic} we conclude the proof of Theorem~\ref{thm:bistochastic}. As we noted above, the key is that \eqref{bdd:sec2} is easily obtained under the double stochasticity assumption by examining the explicit solution $\rvec$ to the Regularized Master Equations.

\subsubsection*{4. Sufficient conditions for \ref{ass:admissible} to hold} This topic is covered in Section \ref{sec:bddness}.

While \eqref{bdd:sec2} can be proved in a few lines under \ref{ass:sigmin} (see Proposition~\ref{prop:sigmin}) or \ref{ass:symmetric} (see Proposition ~\ref{prop:symmetry}), establishing such a bound under the more general robust irreducibility assumption \ref{ass:expander} is significantly more technical. 
Here it is helpful to view the standard deviation profile in terms of the associated directed graph $\Gamma(A(\scut))$ (which was defined in Section \ref{sec:notation}). 
The basic idea is that the equations \eqref{def:MEt} encode relationships between the size of components $\rr_i(t),\rt_i(t)$ at a vertex $i$ to the sizes of the components at neighboring vertices.
Assuming toward a contradiction that $\rr_{i_0}(t)$ is large at some vertex $i_0$, we can use the robust irreducibility assumption to propagate the property of having large $\rr_i(t)$ to most of the other vertices $i$.
We can also use the equations \eqref{def:MEt} to show that $\rt_i(t)$ will consequently be small for most $i$. However, this contradicts the crucial trace identity $\sum_{i=1}^n \rr_i(t)=\sum_{i=1}^n \rt_i(t)$, which derives from the fact that the matrix $\Res$ in \eqref{def:res0} satisfies \[\sum_{i=1}^n \Res_{ii} = \sum_{i=1}^{n}\Res_{n+i,n+i}.\]
See Section \ref{sec:bddness} for further details.

We remark that under the stronger broad connectivity assumption on the standard deviation profile (see Remark \ref{rmk:broad}), Wegner-type estimates that are sufficient for the purposes of this paper were obtained by the first author by a completely different argument, following a geometric approach introduced by Tao and Vu in \cite{tao2010random} -- see \cite[Theorem 4.5.1]{Cook:thesis}.

\subsection{Open questions}	\label{sec:open}

\correction{
\textcolor{red}{
\subsubsection*{Further properties of the density of deterministic equivalents}
Lemma \ref{q:dif} provides an expression for the derivative of the solution $\qvec(s)$ to the Master Equation which could shed some light on further properties of $\mu_n$. However, this expression appears difficult to analyze, and we have not pursued this. Recently \cite{alt2016local} showed the density is strictly positive on the closed disk with radius $\sqrt{\rho(V_n)}$ under assumptions \ref{ass:sigmin}. For instance, is the support of $\mu_n$ always connected? Is it possible to describe the density of $\mu_n$ near zero? 
}}

\paragraph*{\em Relaxing the robust irreducibility assumption}
While control on the smallest singular value is proved under very general conditions (see Proposition~\ref{prop:nick}), we have made the additional robust irreducibility assumption \ref{ass:expander} in order to handle the other small singular values via Wegner estimates. Would it be possible to lighten this assumption? %

\paragraph*{\em Almost sure convergence} One may want to upgrade the convergence $\mu_n^Y\sim \mu_n$ in probability in Theorem \ref{thm:main} to  
almost sure convergence, as discussed in Remark \ref{rmk:almostsure}.

\correction{
\textcolor{red}{
\subsubsection*{Almost sure convergence} 
As we noted in Remark \ref{rmk:almostsure}, the convergence $\mu_n^Y\sim \mu_n$ in probability in Theorem \ref{thm:main} could be upgraded to almost sure convergence if we had improved lower tail estimates on the smallest singular value for nonzero scalar shifts of the matrices $Y_n$ (specifically, an improvement of the bound in Proposition~\ref{prop:nick} to be summable in $n$).
Such an improvement may be possible by combining tools of Inverse Littlewood--Offord theory with the approach in \cite{Cook:ssv}.
}
}

\correction{
\textcolor{red}{
\subsubsection*{Local law}
In \cite{Bourgade2014} it was shown that the circular law holds on the optimal scale of $n^{-1/2+\epsilon}$. These results were extended in \cite{Xiproductlocallaw} to $TX_n$ where $T$ is a deterministic matrix and $X_n$ is an i.i.d.\ random matrix. Both results rely heavily on proving an optimal local law for the empirical distribution of the singular values of scalar shifts of an i.i.d.\ matrix. \\
After the initial release of this paper, a local law was proved in \cite{alt2016local} under \ref{ass:sigmin} and a stronger assumption on distribution of the matrix entries. Additionally they show that \ref{ass:sigmin} implies the deterministic density is bounded from above and below. It remains an open problem to prove a local law on the optimal scale when the density vanishes or is unbounded.
}
}

\paragraph*{\em Extension to sparse models}
While our assumptions allow any fixed proportion of the entries $\sigma_{ij}$ to be zero, Assumption \ref{ass:sigmax} requires the number of non-zero entries to be a constant proportion of the total number of entries. Indeed, otherwise by the Weyl comparison inequality (cf.\ e.g.\ \cite[Theorem~3.3.13]{hor-joh-topics}), the empirical spectral distributions $\mu^Y_n$ converge weakly in probability to $\delta_0$, the point mass at the origin. 
To obtain a nontrivial limit would require a rescaling of the matrices $Y_n$, which amounts to rescaling $A_n$ to have entries of growing size.

\ref{ass:sigmax} is required both to bound the smallest singular value of the shifted random matrices and to prove effective bounds on the Stieltjes transform. 
We expect that our results should extend to certain matrices with density $\sim n^{\eps-1}$ for arbitrary fixed $\eps\in (0,1)$, suitably rescaled.
An interesting first case to consider is random band matrices with shrinking bandwidth. The limit of the empirical distribution of the singular values was recently computed in \cite{Janaband}, but bounds on the smallest singular value were not considered.

\paragraph*{\em Extension to allow heavy-tailed entries}
In a similar direction, it would be interesting to prove an analogue of Theorem \ref{thm:main} for the case that the entries $X_{ij}$ lie in the basin of attraction of an $\alpha$-stable law for some $\alpha\in (0,2)$. 
In this case we expect the deterministic equivalents $\mu_n$ will not have compact support.
The limiting empirical distribution of singular values for such matrices (allowed to be rectangular with bounded eccentricity) was studied by Belinschi, Dembo and Guionnet in \cite{BDG:heavytail}. 
For the case that the entries are i.i.d.\ the limiting empirical spectral distribution was established by Bordenave, Caputo and Chafa\"i in \cite{BCC:heavytail}.

\section{Asymptotics of singular values distributions} 
\label{sec:svd}  

Recall that in Section \ref{outline}
we introduced the Hermitian matrix $\Y_n^z $ in \eqref{Y2n} whose spectral measure is $\check{L}_{n,z}$ in \eqref{eq:symsingularvalues}. We also introduced the resolvent of $\Y_n^z $, $\Res(z,\eta)$ in \eqref{def:res0} and labeled its blocks in \eqref{def:FGdef}. 
%
By the well-known formula for the inverse of a partitioned matrix
\cite[\S 0.7.3]{book-horn-johnson}, 
\begin{equation}\label{def:FG} 
\begin{array}{ccllccl}
G(z,\eta) &=& \eta \left( (Y-z)(Y-z)^* -\eta^2\right)^{-1}&,&F(z,\eta) &=& (Y-z)  \left( (Y-z)^*(Y-z) -\eta^2\right)^{-1}\ ,\\
\widetilde G(z,\eta)& = &\eta \left( (Y-z)^*(Y-z) -\eta^2\right)^{-1}
& , &F'(z,\eta) &=& \left( (Y-z)^*(Y-z) -\eta^2\right)^{-1}(Y-z)^*\ .
\end{array}
\end{equation}
The main objective of this section is to provide deterministic counterparts 
of the normalized traces of these matrix functions. Given the matrices $G$ and $\widetilde G$, it will be convenient to introduce the vectors 
\begin{equation}\label{def:gbold}
\bs{g}=(G_{11}\ \cdots G_{nn})^\tran\, , \qquad  \bs{\tilde g}=(\widetilde G_{11}\ \cdots \widetilde G_{nn})^\tran \qquad \textrm{and} \qquad
\vec{\bs g}^\tran = ({\bs g}^\tran \, \bs{\tilde g}^\tran)\ . 
\end{equation}

We begin by deriving the \emph{Schwinger--Dyson equations}, a system of equations approximately satisfied by the diagonal entries of the matrices in \eqref{def:FG}.  We then show the Schwinger--Dyson equations have a unique solution corresponding to Stieltjes transforms of probability measures and analyze the properties of the solution. 
Finally we estimate the difference between \eqref{def:FG} and the true solution of the Schwinger--Dyson equations, which in turn is used to estimate the difference between the empirical spectral measure of $\Y_n^z$ and its deterministic counterpart.  

\begin{notation}		\label{not:asymp}
Let $\alpha_n = \alpha_n(z,\eta)$ and $\beta_n=\beta_n(z,\eta)$ be complex
sequences such that there exist some constant $C>0$ and some integers $c_0,c_1$ all
independent from $\eta$ and $n$ but which may depend on $z$ such that  
$$
|\alpha_n| \le \frac{C|\eta|^{c_1}}{\im^{c_0}(\eta)\wedge 1} |\beta_n|\ .
$$
We denote this by $\alpha_n = \Oeta{\beta_n}$. If $\alpha_n=\alpha_n^i$ and
$\beta_n=\beta_n^i$ depend on some extra parameter $i\in {\mathcal I}$, then
the notation $\Oeta{}$ in $\alpha_n^i=\Oeta{\beta_n^i}$ must be understood
uniform in $i$. If $\alpha_n$ and $\beta_n$ are vectors or matrices, the
notation $\alpha_n = \vOeta{\beta_n}$ corresponds to a uniform entrywise
relation. 
\end{notation}

\subsection{Derivation of the Schwinger--Dyson equations}

In this subsection we specialize to the case that the entries of $X$ are i.i.d.\ standard complex Gaussian variables. Later we will compare a general matrix with a Gaussian matrix, at which point we will label the Gaussian matrix and associated quantities with a superscript ${\mN}$; however, we omit the superscript in the present subsection.

For a resolvent $\Res$ as defined in \eqref{def:res0} with complex entries, the following differentiation formulas hold true and will be needed in the sequel:
\begin{equation}\label{eq:diff}
\frac{\partial \Res_{ij}}{\partial Y_{k\ell}} = - \Res_{ik}\Res_{n+\ell, j}\ ,\quad \frac{\partial \Res_{ij}}{\partial \overline{Y}_{\ell k}} = - \Res_{i,n+k}\Res_{\ell j}\ ,\quad 1\le k,\ell\le n\ ,\quad 1\le i,j\le 2n\ .
\end{equation}

We will heavily rely on the variance estimates provided in Proposition \ref{prop:var-estimates} and Corollary \ref{coro:var-estimates}.
Denote by $\Y=\left[ \begin{array}{cc} 0 &Y\\ Y^* & 0\end{array}\right]$. The equation $\Res^{-1} \Res=I_{2n}$ yields
\begin{equation}\label{eq:RR=I}
-\eta \Res + \Y \Res + \left[ \begin{array}{cc} -z F'  &-z\widetilde{G} \\ -z^* G & -z^* F\end{array}\right] = I_{2n}\ .
\end{equation} 
Taking $i\in [n]$ yields
\begin{equation}\label{eq:1}
-\eta \E G_{ii} +\E (\Y \Res)_{ii} -z\E F'_{ii} = 1\ .
\end{equation}

Applying the integration by part formula for complex Gaussian random variables (see for instance \cite[(2.1.40)]{pas-livre}) together with \eqref{eq:diff} yields
$$
\E(\Y\Res)_{ii} = \sum_{\ell=1}^n \E Y_{i\ell} \Res_{n+\ell,i} =\sum_{\ell=1}^n \frac{\sigma^2_{i\ell}}n \mathbb{E} \left[
\frac{\partial \Res_{n+\ell, i}}{\partial \overline{Y_{i\ell}}}  
\right] =- \sum_{\ell=1}^n \frac{\sigma^2_{i\ell}}n \E ( \Res_{n+\ell, n+\ell } \Res_{ii} )\ .
$$
Plugging this into \eqref{eq:1} yields
\begin{equation}\label{eq:1bis}
-\E \left(\eta + [V_n\bs{\tilde g}]_i \right) G_{ii} - z\E F'_{ii}=1 \ .
\end{equation}
Specializing again Equation \eqref{eq:RR=I} for $i\in [n]$ yields, with similar arguments, 
\begin{equation}\label{eq:2}
-\eta \E F_{ii} -  \E [V_n\bs{\tilde g}]_iF_{ii} - z \E \widetilde{G}_{ii} =0\ .
\end{equation}
Arguing similarly with the help of the following integration by parts formula, valid for $i>n$,
$$
\E(\Y\Res)_{ij} = \sum_{\ell=1}^n \E (\overline{Y}_{\ell, i-n} \Res_{\ell j}) = \sum_{\ell=1}^n \frac{\sigma^2_{\ell, i-n}}n \E\left[ \frac{\partial \Res_{\ell j}}
{\partial Y_{\ell, i-n}}\right] = - \sum_{\ell=1}^n \frac{\sigma_{\ell, i-n}^2}n \E(\Res_{\ell \ell} \Res_{ij})
$$
yields the following equations
\begin{eqnarray}
-\E (\eta + [V_n^\tran  \bs{g}]_i )\widetilde G_{ii} -z^* \E F_{ii} &=& 1\ ,\label{eq:3}\\
-\E(\eta+ [V_n^\tran  \bs{g}]_i)F'_{ii} -z^*\E G_{ii} &=&0\ .\label{eq:4}
\end{eqnarray}
Notice that equations \eqref{eq:1bis}-\eqref{eq:4} can be compactly written
\begin{equation}
\label{eq:approx-master-equation}
\E \left[\begin{array}{cc}
-\diag (V_n\bs{\tilde g}) -\eta & -z\\
-z^* &-\diag (V_n^\tran  \bs{g})-\eta
\end{array}\right]\Res = I_{2n}\ .
\end{equation}
Using Cauchy-Schwarz
inequality and the estimates in Proposition \ref{prop:var-estimates}, we get
$$
\E [V_n\bs{\tilde g}]_i F_{ii} - \E [V_n\bs{\tilde g}]_i\E F_{ii} = \Oeta{\frac 1{n^{3/2}}}\quad \textrm{thus}\quad 
-\left( \eta +\E [V_n\bs{\tilde g}]_i\right) \E F_{ii}= z\E \widetilde G_{ii} + \Oeta{\frac 1{n^{3/2}}}
$$
by \eqref{eq:2}. In particular,
$$
- \E F_{ii} = z\frac{\E \widetilde G_{ii}}{\eta +\E [V_n\bs{\tilde g}]_i} + \Oeta{\frac 1{n^{3/2}}}
$$
since $\left|\eta +\E  [V_n\bs{\tilde g}]_i\right|^{-1}\le \im^{-1}(\eta)$. On the other hand, using the same decorrelation argument in equation \eqref{eq:3}, we obtain
$$
-\E (\eta + [V_n^\tran  \bs{g}]_i)\E \widetilde G_{ii} -z^* \E F_{ii}= 1 + \Oeta{\frac 1{n^{3/2}}}\ .
$$
Combining these two equations, we finally get
$$
\E \widetilde G_{ii} \left\{ - (\eta + \E [V_n^\tran  \bs{g}]_i) + \frac{|z|^2}{\eta +  \E [V_n\bs{\tilde g}]_i} \right\} = 1+ \Oeta{\frac 1{n^{3/2}}}\ .
$$
Using the property \eqref{eq:ST-estimate} twice, one has
$$
\left| - (\eta + \E [V_n^\tran  \bs{g}]_i ) + \frac{|z|^2}{\eta + \E [V_n\bs{\tilde g}]_i} \right|^{-1}\ \le\ \frac 1{\im(\eta)}\ .
$$
Hence
\begin{equation}\label{eq:Gtilde-ii}
\E \widetilde G_{ii} \ =\  \frac 1{- (\eta + \E [V_n^\tran  \bs{g}]_i) + \frac{|z|^2}{\eta + \E[V_n\bs{\tilde g}]_i}}+ \Oeta{\frac 1{n^{3/2}}}\ .
\end{equation}
Combining similarly equations \eqref{eq:1bis} and \eqref{eq:4} and decorrelating when needed with the help of Proposition \ref{prop:var-estimates}, we obtain the companion equation:
\begin{equation}\label{eq:G-ii}
\E G_{ii} \ =\  \frac 1{- (\eta + \E [V_n\bs{\tilde g}]_i) + \frac{|z|^2}{\eta +\E [V_n^\tran  \bs{g}]_i}}+ \Oeta{\frac 1{n^{3/2}}}\ .
\end{equation}

We now introduce an unperturbed version of equations \eqref{eq:Gtilde-ii} and \eqref{eq:G-ii}.

\subsection{Schwinger--Dyson equations}		\label{sec:self}

In this section we introduce the Schwinger--Dyson equations \eqref{cve0}. 
Notice that these equations already appear in \cite{AEK16:mde}, from which we will deduce properties for their solutions.
To write this more compactly, we introduce the notation $\boldsymbol{\vec{b}}=\begin{pmatrix} \bs{b} \\ \bs{\tilde b}\end{pmatrix}$  for any two $n\times 1$ vectors $\boldsymbol{b}$ and $\boldsymbol{\tilde b}$ with complex components 
and the following definitions: 
\begin{equation}
\label{def-upsilon} 
\bs{\Upsilon}(\bs{\vec{b}},\eta) \ :=\ 
\diag \left( \frac 1{|z|^2-([V_n\bs{\tilde b}]_i+\eta) ([V_n^\tran \bs{b}]_i +\eta) }  
;\, i\in [n]\right) \ :=\ \diag( \Upsilon_i(\bs{\vec{b}},\eta)\, ;\  i\in[n] )\ ,
\end{equation} 

and
\begin{eqnarray}
\mathcal J(\bs{\vec b},\eta) &:=&
\begin{pmatrix} 
\bs{\Upsilon}(\bs{\vec b},\eta) V_n^\tran  & 0\\
0&\bs{\Upsilon}(\bs{\vec b},\eta) V_n
\end{pmatrix} \bs{\vec b}
+ \eta \begin{pmatrix} \Upsilon(\bs{\vec b},\eta) \boldsymbol{1}\\
                    \Upsilon(\bs{\vec b},\eta) \boldsymbol{1} \end{pmatrix} \ .\\
                    \nonumber
\end{eqnarray} 
Then \eqref{cve0} can be compactly written as 
\begin{equation}
\bs{\vec p} = {\mathcal J}(\bs{\vec p}, \eta).
\end{equation}
Both $\bs{\Upsilon}$ and ${\mathcal J}$ depend on $z$ as well (and to be even more precise, on $|z|$). We will not indicate this dependence in the sequel.

We now collect properties of solutions to \eqref{cve0}. 

\begin{prop}[Schwinger--Dyson equations]
\label{prop:deteq} 
For all fixed $\eta \in \C_+$ and $z\in \C$, let $\bs p=(p_i)$ and $\bs{\tilde p}=(\tilde p_i)$ be two $n\times 1$ vectors which solve \eqref{cve0}.
\begin{enumerate}
\item \label{SD-prop1} The system \eqref{cve0} admits a unique solution $\pvec$ 
satisfying $\Imm \pvec \succ 0$.
\item \label{SD-prop2} For any initial vector $\pvec_0$ with $\Imm\pvec_0 \succcurlyeq 0$, 
the iterations 
$
\pvec_{k+1} = \mathcal J(\pvec_k)
$ converge to this solution $\pvec$
as $k\to\infty$. 
\item \label{SD-prop3} For all $z\in\C$ and $i\in [n]$, the functions
$
\eta\mapsto p_i(\eta)$ and $\eta\mapsto \tilde p_i(\eta)
$ 
are Stieltjes transforms of symmetric probability measures on $\R$ respectively denoted by $\mu_i$ and $\tilde \mu_i$. In particular, let $r_i(t):= \im\, p_i(\ii t)$ and $\tilde r_i(t) := \im\, \tilde p_i(\ii t)$ for $t>0$ then
$$
p_i(\ii t) = \ii r_i(t) \quad \textrm{and}\quad \tilde p_i(\ii t) = \ii \tilde r_i(\ii t)\, .
$$ 
Otherwise stated, $p_i$ and $\tilde p_i$ are purely imaginary complex numbers along the imaginary axis.
\item \label{SD-prop4} Moreover, $\sum_{i=1}^n p_i \ =\ \sum_{i=1}^n \tilde p_i$ and the common value
\[
\eta\ \longmapsto\ \frac 1n \sum_{i=1}^n p_i(\eta) \ =\ 
\frac 1n \sum_{i=1}^n \tilde p_i(\eta) 
\]
is the Stieltjes transform if a symmetric probability measure $\check\nu_{n,z}$. We denote this Stieltjes transform by $\eta \mapsto g_{\check\nu_{n,z}}(\eta)$.
\item \label{SD-prop5} The sequences of probability measures $(\mu_i;\, i\le n; \, n\ge 1)$, $(\tilde \mu_i;\, i\le n;\, n\ge 1)$ and $(\check \nu_{n,z};\, n\ge 1)$ are tight. In particular, 
$$
\sup_{n\ge 1} \int |x|^2 \check \nu_{n,z}(\, dx) \ <\ \infty\ .
$$

\end{enumerate}
\end{prop} 
We henceforth refer to the solution $\bs{\vec p}=\bs{\vec p}(\eta)$, $\Imm \bs{\vec p}\succ 0$ as the solution to the \emph{Schwinger--Dyson equations}. 

\begin{proof} 
Proofs of parts \eqref{SD-prop1} and \eqref{SD-prop2} are a direct application of Earle-Hamilton's theorem \cite{har-01}, which was first used (to the authors' knowledge) in random matrix theory by \cite{speicher-et-al-2007}.

\correction{
\textcolor{red}{[TO BE REMOVED] \{ 
The proof of part (1) relies on Earle-Hamilton's theorem \cite{har-01}, which was first used (to the authors' knowledge) in random matrix theory by \cite{speicher-et-al-2007}. 
Let $\mathcal D$ be the domain of $\C^{2n}$ defined as $\mathcal D := \{ \pvec
\in \C^{2n} \, : \, \Imm \pvec \succ 0, \| \pvec \|_\infty < b \}$, where
$\|\cdot\|_\infty$ is the supremum norm.  For a convenient choice of the
constant $b > 0$, we will show that the holomorphic function $\mathcal J$ on
the domain $\mathcal D$ satisfies the property that there is an $\varepsilon > 0$
such that the $\varepsilon$-neighborhood of $\mathcal J(\mathcal D)$ lies in
$\mathcal D$. The Earle-Hamilton theorem then states that $\mathcal J$ is a
strict contraction with respect to the so-called Carath\'eodory-Riffen-Finsler
metric, and the results of the proposition follow at once. 
Write $(\mathcal J(\pvec))_i = 1/d_i$ for $i\in[n]$. For 
$\pvec \in \mathcal D$, we have 
\begin{gather*} 
\Imm d_i = - \Imm ([V_n \bs{\tilde p}]_i + \eta) - |z|^2 
\frac{\Imm ([V_n^\tran \bs{p}]_i + \eta)}{|[V_n^\tran \bs{p}]_i + \eta|^2} \leq 
- \Imm \eta , \quad \text{and}  \quad  
| d_i | \leq |\eta| + b + \frac{|z|^2}{\Imm \eta} . 
\end{gather*}
Therefore, 
\begin{equation}
\label{eq:iterbound}
\Imm (\mathcal J(\pvec))_i = \frac{- \Imm d_i}{|d_i|^2} \geq 
\frac{\Imm \eta}{(|\eta| + b + |z|^2 / \Imm\eta)^2} > 0, 
\quad \text{and} \quad 
| (\mathcal J(\pvec))_i | = \frac{1}{|d_i|} \leq \frac{1}{\Imm\eta} , 
\end{equation}
and the case $n+1\leq i \leq 2n$ is handled similarly. 
By choosing $b > 1/\Imm \eta$, we get the desired result on 
$\mathcal J(\mathcal D)$. 
Part (2) is a by-product of Earle-Hamilton's theorem.\}}}

We now address part \eqref{SD-prop3} and \eqref{SD-prop5}. 
The fact that the $p_i$'s and $\tilde p_i$'s are Stieltjes transforms of probability measures immediately follows from \cite[Proposition 2.1]{AEK16:mde}. This result also yields the desired tightness properties.
In order to complete the proof of part \eqref{SD-prop3}, we now prove that the probability measures associated to the $p_i$'s are symmetric. 
To this end, simply observe that, given $\eta \in \C_+$,
if $\pvec = (\bp, \bpt)$ is the solution with $\Imm \pvec \succ 0$ of the Schwinger--Dyson equations \eqref{cve0}, 
then $-\pvec$ is the unique solution with
$\Imm(-\pvec) \prec 0$ of the analogous system obtained by replacing $\eta$
with $-\eta$. The result follows from the application of Lemma~\ref{nusym}.

It remains to prove that along the imaginary axis $(\ii t,t>0)$, $p_i$ and $\tilde p_i$ are purely imaginary complex numbers.
\begin{eqnarray}
p_i(\ii t) &=& 
 \int_{(-\infty,0)} \frac{\mu_i(d\lambda) }{\lambda - \ii t}  
 + \int_{(0,\infty)} \frac{\mu_i(d\lambda) }{\lambda - \ii t} 
 - \frac{\mu_i(\{0\})}{\ii t}  
\ =\   \int_{(0,\infty)} \Bigl( \frac{1}{\lambda - \ii t} - 
\frac{1}{\lambda + \ii t} \Bigr) \mu_i(d\lambda) 
 + \frac{\ii\,  \mu_i(\{0\}) }{t} \nonumber \\
&=& 2\ii \int_{(0,\infty)} \frac{t}{\lambda^2+t^2} \mu_i(d\lambda) 
+ \frac{\ii\, \mu_i(\{0\})}{t} \quad =:\quad \ii \rr_i(t)\ .\label{eq:p->q}
\end{eqnarray}  

Similarly, $\bs{\tilde p}(\ii t)=\ii \brt(t)$ with $\brt(t)=(\rt_i(t))$ and $\rt_i(t)>0$. Notice for future use that 
\begin{equation}\label{eq:q-lim}
\lim_{t\to\infty} t \rr_i(t)= 1\quad \textrm{and}\quad  \lim_{t\to\infty} t\rt_i(t)=1\ .
\end{equation}

Parts \eqref{SD-prop3} and \eqref{SD-prop5} of the theorem are established.

\correction{
\textcolor{red}{and rely on the implicit function theorem for holomorphic functions (see for instance \cite[Theorem 7.6, Chapter 1]{fritzsche-holomorphic}).  
In order to prove that $\eta\mapsto \vec{\bs p}(\eta)$ is holomorphic, first notice that 
$$
\vec{\bs p}- {\mathcal J}(\vec{\bs p}, \eta) \ =\ 0\ .
$$
If, for all $\eta\in \C_+$,  the Jacobian $ J(\vec{\bs p},\eta)$ of $\vec{\bs p}\mapsto \vec{\bs p}- {\mathcal J}(\vec{\bs p}, \eta)$ differs from zero, then the function $\eta\mapsto \vec{\bs p}(\eta)$ will be holomorphic on $\C_+$.\\
Recall the definition of $\bs{\Upsilon}$ in \eqref{def-upsilon}. In order to express $ J(\vec{\bs p},\eta)$, we introduce a few more notations. \\
For given $n\times 1$ vectors $\bs{b}$ and $\bs{\tilde b}$, define
\begin{equation}
\left\{ 
\begin{array}{ccc}\Delta(\bs{b}) &=& \diag( \eta + (V_n^\tran  \bs{b})_i) := \diag (\Delta_i(\bs{b})\, )\\
\widetilde \Delta(\bs{\tilde b}) &=& \diag( \eta + (V_n \bs{\tilde b})_i):=\diag(\widetilde \Delta_i(\bs{\tilde b})\, )
\end{array}\right.\ .
\label{def:Delta}
\end{equation}
Recall that matrix $A$ is the standard deviation profile and satisfies $V=\frac A{\sqrt{n}}\odot \frac A{\sqrt{n}}$. 
For a given $2n\times 1$ vector $\bs{\vec b}=\begin{pmatrix} \bs{b}\\ \bs{\tilde b}\end{pmatrix}$, introduce the $2n\times 2n$ matrix
\begin{equation}\label{eq:def-A}
{\mathcal A}(\bs{\vec b}) =\begin{pmatrix}
\dfrac{|z| \bs{\Upsilon}\big(\bs{\vec b}\,\big) A^\tran }{\sqrt{n}}& \dfrac{\bs{\Upsilon}\big(\bs{\vec b}\,\big) \Delta(\bs{b})A}{\sqrt{n}}\\
\\
\dfrac{ \bs{\Upsilon}\big(\bs{\vec b}\,\big) \widetilde \Delta(\bs{\tilde b})A^\tran }{\sqrt{n}}& \dfrac{|z| \bs{\Upsilon}\big(\bs{\vec b}\,\big) A}{\sqrt{n}}
\end{pmatrix}\ .
\end{equation}
Then straightforward computations yield
\begin{equation}
 J(\vec{\bs p},\eta)\ =\ 
\mathrm{det}\left(
I_{2n} - {\mathcal A}(\vec{\bs p})\odot {\mathcal A}(\vec{\bs p})
\right)\ .
\end{equation}
On the other hand, straightforward but lengthy computations yield
$$
\im(\vec{\bs{p}}) = {\mathcal A}(\vec{\bs{p}}) \odot \overline{{\mathcal A}(\vec{\bs{p}})} \, \im(\vec{\bs{p}}) + \bs{v}\ ,
$$
where $\bs{v}\succ 0$ (see Section \ref{app:linear-algebra} for details). By Proposition \ref{prop:folklore}, 
$$
\rho \left(  {\mathcal A}(\vec{\bs{p}}) \odot \overline{{\mathcal A}(\vec{\bs{p}})}\right) \ <\ 1\ .
$$
Applying \cite[Theorem 8.1.18]{book-horn-johnson}, we have
$$
\rho \left(  {\mathcal A}(\vec{\bs{p}}) \odot {\mathcal A}(\vec{\bs{p}})\right)\ \le \ \rho \left(  {\mathcal A}(\vec{\bs{p}}) \odot \overline{{\mathcal A}(\vec{\bs{p}})}\right)\ <\ 1\ .
$$
Hence $I_{2n} - {\mathcal A}(\vec{\bs{p}}) \odot {\mathcal A}(\vec{\bs{p}})$ is invertible and its determinant is nonzero. We focus on $p_i$, the same arguments will work for $\tilde p_i$.  The implicit function theorem yields that $\eta\mapsto p_i(\eta)$ is an analytic map from $\C_+$ onto itself. It remains to prove that $\lim_{t\to \infty} \ii t p_i(\ii t) =-1$. \\
Since $\vec{\bs p}= {\mathcal J}(\vec{\bs p}, \eta)$, \eqref{eq:iterbound} implies that $\left| t \,p_i(\ii t) \right| \ \le \ 1$. Combining this estimate with \eqref{cve0}, implies that $\lim_{t\to \infty} \ii t p_i(\ii t) =-1$. }}

We now prove part \eqref{SD-prop4}, that is $\sum_i p_i = \sum_i \tilde p_i$. Getting back to the system \eqref{cve0}, we have 
\[
\sum_{i=1}^n p_i([V\tilde{\bs p}]_i+\eta) = 
\sum_{i=1}^n \frac{([V\tilde{\bs p}]_i +\eta)([V_n^\tran \bs{p}]_i +\eta)}
{-([V\tilde{\bs p}]_i+\eta)([V_n^\tran \bs{p}]_i +\eta) + |z|^2} = 
\sum_{i=1}^n \tilde p_i([V_n^\tran \bs{p}]_i +\eta) .
\]
But 
\[
\sum_{i=1}^n p_i[V\tilde{\bs p}]_i = \sum_{i,\ell=1}^n p_i \sigma_{i,\ell}^2 \tilde p_i
= \sum_{i=1}^n \tilde p_i[V_n^\tran \bs{p}]_i .
\]
Since $\eta\neq 0$, we get the desired result. 
\correction{
\textcolor{red}{We now prove part \eqref{SD-prop5}. Notice first that 
$$
\mathrm{Re} \left\{ t^2\left( \ii tp_i(\ii t) +1\right)\right\} =\int \frac{t^2\lambda^2}{t^2+\lambda^2} \mu_i(\, d\lambda)\quad \Rightarrow \quad 
\mathrm{Re} \left\{ t^2\left( \ii tp_i(\ii t) +1\right)\right\} \xrightarrow[t\to\infty]{}\int \lambda^2 \mu_i(\, d\lambda)\ .
$$ 
We now rely on the equation \eqref{cve0} satisfied by $p_i$. We have
$$
\ii t p_i(\ii t) + 1 = \ii t \frac{ [V^\tran  \bs{p}]_i +\ii t}{|z|^2 - ( [V^\tran  \bs{p}]_i +\ii t)( [ V \bs{\tilde p}]_i +\ii t)} +1
= \frac{|z|^2- \ii t[V\bs{\tilde p}(t)]_i -  [V^\tran \bs{p}(t)]_i  [ V\bs{\tilde p}(t)]_i }{|z|^2 - ( [V^\tran  \bs{p}]_i +\ii t)( [ V \bs{\tilde p}]_i +\ii t)}\, .
$$
Multiplying by $t^2$, taking the limit as $t\to \infty$ and taking into account $\lim_{t\to \infty} \ii t p_i(\ii t) =-1$ we have
$$
\int \lambda^2\, \mu_i(\, d\lambda) \ =\ |z|^2 + [V\bs{1}_n]_i \ \le \ |z|^2 + \smax^2\ .
$$
The same estimate holds for $\int \lambda^2 \tilde \mu_i(d\lambda)$ and $\int \lambda^2 \check \nu_{n,z}(d\lambda)$, hence the required tightness.}}
\end{proof}

\subsection{Asymptotics of the spectral measure $\check L_{n,z}$ and the Hermitian resolvent}

\begin{theo} 
\label{th:L->nu} 
Assume \ref{ass:moments} and \ref{ass:sigmax} hold, and let $\check{\nu}_{n,z}$ be defined as in Proposition \ref{prop:deteq}-(4).
Then for all $z\in \C$, $(\check \nu_{n,z})_n$ is tight, and 
$$
\check L_{n,z} \sim \check\nu_{n,z}
$$ 
almost surely. Moreover, for any $\varepsilon > 0$, $x\mapsto \log |x|$ is $\check \nu_{n,z}$-integrable on the set $\{ |x|\ge \varepsilon\}$ and
\[
\int_{\{ |x|\ge \varepsilon\}} \log |x| \, \check L_{n,z}(dx) - 
\int_{\{ |x|\ge \varepsilon\}} \log |x| \, \check\nu_{n,z}(dx)  
\ \xrightarrow[n\to\infty]{\text{a.s.}} \ 0 \ .
\]
\end{theo} 
We will sometimes refer to $\check \nu_{n,z}$ as the deterministic equivalent of $\check L_{n,z}$.

The proof of Theorem \ref{th:L->nu} is postponed to Section~\ref{prf:L->nu}. 
Notice that the first part ($\check L_{n,z}\sim \check \nu_{n,z}$) is a variation of classical results, see for example \cite{hachem-et-al-2007}. It will be a direct consequence of the forthcoming theorem on the asymptotics of Hermitian resolvent.

In order to get some insight on the asymptotics of the spectral measure $\mu_n^Y$, we need more than the asymptotics of $\check L_{n,z}$. 
We rewrite hereafter the Schwinger--Dyson equations of Proposition \ref{prop:deteq} in a more suitable way for the forthcoming analysis. 
In what follows, the dependence in $|z|$ is implicit and will be recalled if necessary.

We now introduce the deterministic equivalents to $F$ and $G$, defined in \eqref{def:FG}.
Let $\pvec=(\bs{p} , \bs{\tilde p})$ be the solution of the Schwinger--Dyson equations \eqref{cve0}. Define the $n\times n$ diagonal matrices $P$, $\widetilde P$, $\Theta$ and $\widetilde \Theta$ by
$$
P \ :=\  \diag(\p) \ ,\quad  \widetilde P \ := \ \diag(\ptilde) \ ,\quad \Theta  \ :=\  \diag\left ( (V_n\,\bs{\tilde p})_i \,   ,\ i\in [n] \right)
\quad \textrm{and}\quad 
\widetilde \Theta \ :=\
\diag\left( (V_n^\tran\bs{p})_i \, , \ i\in [n]\right) \ . 
$$
After easy massaging, the Schwinger--Dyson equations $\pvec ={\mathcal J}(\pvec,\eta)$ are equivalent to:
$$
P \ =\  \left( - (\Theta +\eta) + 
|z|^2 (\widetilde\Theta +\eta)^{-1} \right)^{-1}\quad \textrm{and}\qquad 
\widetilde P  \ =\  \left( - (\widetilde\Theta +\eta) + 
|z|^2 (\Theta +\eta)^{-1} \right)^{-1}\ .
$$
Consider $2n\times 2n$ matrix $\bS$ defined as
\begin{equation}
\label{eq:master-equation}
\bS :=
-\begin{pmatrix} 
 \Theta(|z|,\eta) +\eta & z\\
z^*&  \widetilde \Theta(|z|,\eta)  +\eta
\end{pmatrix}^{-1} .  
\end{equation}
This definition is similar to equation \eqref{eq:approx-master-equation} satisfied by the entries of the resolvent $\Res$.
By the formula for the inverse of a partitioned matrix
\cite[\S 0.7.3]{book-horn-johnson}, it holds that 
\[
\bS = \begin{pmatrix} P(|z|,\eta) & B(z,\eta)\\
B'(z,\eta)& \widetilde P(|z|,\eta) \end{pmatrix} , 
\]
where
\begin{equation}
B(z,\eta) \ =\  -z\Big(\Theta(|z|,\eta) +\eta \Big)^{-1} \widetilde P(|z|,\eta) \ =\ 
 -zP(|z|,\eta) \left(\widetilde \Theta(|z|,\eta)  +\eta\right)^{-1} \ , 
 \label{def:B}
\end{equation}
and $B'(z,\eta)$ can be made explicit in a similar fashion, but will not 
be used.

\begin{theo} 
\label{th:convergence-QR} 
Assume \ref{ass:moments} and \ref{ass:sigmax} hold.
Then almost surely, for every $z\in \C$ and $\eta\in \C_+$, 
\[
\frac 1n \begin{pmatrix} 
\tr G(z,\eta) & \tr F(z,\eta) \\
\tr F'(z,\eta) & \tr \widetilde G(z,\eta) 
\end{pmatrix} 
-
\frac 1n 
\begin{pmatrix} 
 \tr P(|z|,\eta) & \tr B(z,\eta)\\
 \tr B'(z,\eta) &  \tr \widetilde P(|z|,\eta)
\end{pmatrix} 
\xrightarrow[n\to\infty]{}  0 \ .
\]
Moreover, there exist $\alpha,\beta>0$ such that for $t\in (n^{-\alpha},n^{\beta})$ and $n$ large enough, with $\eta=\ii t$ we have 
\[
\frac 1n 
\begin{pmatrix} 
\tr \E G(z,\eta) & \tr \E F(z,\eta) \\
\tr \E F'(z,\eta) & \tr \E \widetilde G(z,\eta) 
\end{pmatrix} 
 - \frac 1n \begin{pmatrix} 
 \tr P(|z|,\eta) & \tr B(z,\eta)\\
 \tr B'(z,\eta) &  \tr \widetilde P(|z|,\eta)
\end{pmatrix} 
= \vOeta{n^{-1/2}} \, .
\] 
\end{theo}

The rate provided along the imaginary axis is not optimal, but sufficient for our purposes. 

The proof of Theorem \ref{th:convergence-QR} immediately follows from Propositions \ref{prop:QR-concentration}, \ref{prop:QR-interpolation}
and \ref{prop:QR} stated hereafter.

\begin{prop}\label{prop:QR-concentration}
Assume \ref{ass:moments} holds and let $z\in \C$ and $\eta\in
\Cplus$.  Then almost surely,
\[
\frac 1n 
\begin{pmatrix} 
\tr  G(z,\eta) & \tr  F(z,\eta) \\
\tr  F'(z,\eta) & \tr  \widetilde{G}(z,\eta) 
\end{pmatrix} 
 - \frac 1n \begin{pmatrix} 
\tr \E G(z,\eta) & \tr \E F(z,\eta) \\
\tr \E F'(z,\eta) & \tr \E \widetilde{G}(z,\eta) 
\end{pmatrix} \xrightarrow[n\to\infty]{} 0\ . 
\] 
\end{prop}
\begin{proof} This is a direct application of \cite[Lemma 4.21]{2012-bordenave-chafai-circular}.
\end{proof}

To manage the expectation terms $n^{-1} \tr\E(\cdot)$, we introduce the Gaussian counterparts of the quantities of interest.
Consider a family of i.i.d. standard complex random variables 
$(\Xn_{ij}; 1\le i,j\le n)$, where $\Xn_{ij}=(U+\ii U')/\sqrt{2}$, with 
$U,U'$ being independent real ${\mathcal N}(0,1)$ random variables. 
Notice in particular that
$$
\E \Xn_{ij} =0\ ,\quad \E\left( \Xn_{ij}\right)^2=0 \quad \text{and}\quad \E|\Xn_{ij}|^2=1\ .
$$
Similarly, let $\Yn_{ij}=\frac{\sigma_{ij}}{\sqrt{n}} \Xn_{ij}$, and let 
$\Res^{\mathcal N},\ \Gn,\ \tGn,\ \Fn$, and $\cFn$ the matrix functions associated with the matrix 
$\Yn = (\Yn_{ij})$ as in~\eqref{def:res0},\eqref{def:FG}. Then we have the following 
proposition.

\begin{prop}
\label{prop:QR-interpolation} 
Assume \ref{ass:moments} and \ref{ass:sigmax} hold. Let $z\in \C$ and $\eta\in \Cplus$. Then 
\[
\frac 1n 
\begin{pmatrix} 
\tr \E G(z,\eta) & \tr \E F(z,\eta) \\
\tr \E F'(z,\eta) & \tr \E \widetilde{G}(z,\eta) 
\end{pmatrix} 
- \frac 1n \begin{pmatrix} 
\tr \E \Gn(z,\eta) & \tr \E \Fn(z,\eta) \\
\tr \E \cFn(z,\eta) & \tr \E \tGn(z,\eta) 
\end{pmatrix} 
=\vOeta{n^{-1/2}}\ .
\] 
\end{prop}

The proof of Proposition \ref{prop:QR-interpolation} relies on fairly standard
arguments and is thus postponed to Appendix~\ref{app:proof-interpolation}.

\begin{prop}
\label{prop:QR} 
Assume \ref{ass:sigmax} holds, and let $z\in \C$ and $\eta\in \C_+$. Then 
\[
\frac 1n 
\begin{pmatrix} 
\tr \E \Gn(z,\eta) & \tr \E \Fn(z,\eta) \\
\tr \E \cFn(z,\eta) & \tr \E \tGn(z,\eta) 
\end{pmatrix} 
- \frac 1n \begin{pmatrix} 
\tr P(|z|,\eta) &\tr B(z,\eta)\\
\tr B'(z,\eta) &  \tr \widetilde P(|z|,\eta)
\end{pmatrix} 
\xrightarrow[n\to\infty]{} 0\, .
\] 
Moreover, there exist $\alpha,\beta>0$ such that for $t\in (n^{-\alpha},n^{\beta})$ and $n$ large enough, with $\eta=\ii t$ we have 
\[
\frac 1n 
\begin{pmatrix} 
\tr \E \Gn(z,\eta) & \tr \E \Fn(z,\eta) \\
\tr \E \cFn(z,\eta) & \tr \E \tGn(z,\eta) 
\end{pmatrix} 
- \frac 1n \begin{pmatrix} 
\tr P(|z|,\eta) &\tr B(z,\eta)\\
\tr B'(z,\eta) &  \tr \widetilde P(|z|,\eta)
\end{pmatrix} 
= \vOeta{n^{-3/2}}\, .
\] 
\end{prop} 

The proof of Proposition \ref{prop:QR} follows hereafter in Section
\ref{proof:QR-quantitative}. 
\correction{
\textcolor{red}{It relies on an inequality on quadratic forms of
independent interest (cf.\ Lemma \ref{lemma:key-lemma}).}}

From Theorem \ref{th:convergence-QR}, we will deduce the asymptotic behavior
of the empirical distribution $L_{n,z}$ of the singular values of $Y_n-z$ 
by analyzing the convergence of $n^{-1} \tr G(z,\eta)$. Moreover, for any 
$\varepsilon > 0$, the asymptotic behavior of $\int_{\{|x|\ge \varepsilon\} } 
\log |x| \, \check L_{n,z}(dx)$ will be identified.

\subsection{Proof of Proposition \ref{prop:QR}}
\label{proof:QR-quantitative}
We first prove the convergence to zero. Recall that $V=\frac A{\sqrt{n}}\odot \frac A{\sqrt{n}}$, that $\vec{\bs g}$ is introduced in \eqref{def:gbold}, that $\vec{\bs p}$ is the solution of the Schwinger--Dyson equations and define 
$$
\bs{\vec \varepsilon} = \E \bs{\vec g} - \pvec\, .
$$
Recall the definition of $\bs{\Upsilon}$ in \eqref{def-upsilon}. For given $n\times 1$ vectors $\bs{b}$ and $\bs{\tilde b}$, define
\begin{equation}
\left\{ 
\begin{array}{ccc}\Delta(\bs{b}) &=& \diag( \eta + (V_n^\tran  \bs{b})_i) := \diag (\Delta_i(\bs{b})\, )\\
\widetilde \Delta(\bs{\tilde b}) &=& \diag( \eta + (V_n \bs{\tilde b})_i):=\diag(\widetilde \Delta_i(\bs{\tilde b})\, )
\end{array}\right.\ .
\label{def:Delta}
\end{equation}
Then
\begin{align}
\E G_{ii} - p_i
 &=  \frac 1{- (\eta +  [V\E \bs{\tilde g}]_i) + \frac{|z|^2}{\eta + [V^\tran \E\bs{g}]_i}}
- \frac 1{- (\eta + [V\bs{\tilde p}]_i) + \frac{|z|^2}{\eta +[V^\tran \bs{p}]_i}}
+ \Oeta{n^{-3/2}}\ \nonumber\\
&= \Upsilon_i (\E \bs{\vec g}) \Upsilon_i (\pvec)
\left\{ |z|^2 [V^\tran(\E {\bs g} - \bs{p})]_i +  \Delta_i(\E \bs g) \Delta_i({\bs p}) [V(\E \bs{\tilde g} - \bs{\tilde p})]_i
\right\} + \Oeta{n^{-3/2}}\ ,
\label{eq:diff-Gii}
\end{align}
and a similar expression holds for $\E \widetilde G_{ii} - \tilde p_i$. Taking into account the straightforward estimates
$$
| \Delta_i(\E \bs g) | ,\ | \Delta_i({\bs p})| \le |\eta| + \frac{\sigma_{\max}^2}{\im(\eta)}\quad \textrm{and}
\quad | \Upsilon_i (\E \bs{\vec g})|,\  | \Upsilon_i (\pvec)| \le \frac 1{\im^2(\eta)}\, ,
$$
we end up with 
$$
\| \bs{\vec \varepsilon}\,\|_{\infty} \le K\left( \frac{|\eta|^2 +|z|^2}{\im^4(\eta)} + \frac 1{\im^6(\eta)}\right) \| \bs{\vec \varepsilon}\,\|_{\infty} + 
\vOeta{n^{-3/2}}\ ,
$$
where $K$ is an absolute constant (depending on $\sigma_{\max}$).
Letting $\eta\in \Cplus$ be chosen in such a way that $K\left( \frac{|\eta|^2 +|z|^2}{\im^4(\eta)} + \frac 1{\im^6(\eta)}\right)<1$, one has 
$\| \bs{\vec \varepsilon}\,\|_{\infty}(\eta)\xrightarrow[n\to\infty]{} 0$. 
On the other hand, $\E G_{ii} - p_i$ is analytic and uniformly bounded over the compact subsets of $\C_+$. 
Therefore, every converging subsequence converges toward an analytic function which coincides with the zero function on a sufficiently large 
open subset  
of $\Cplus$, and hence is equal to the zero function over $\Cplus$. The same applies to $\E \tilde G_{ii} - \tilde p_i$. This proves that $\| \bs{\vec \varepsilon}\,\|_{\infty} \xrightarrow[n\to\infty]{} 0$ for every $\eta\in \Cplus$.

We now focus on $\eta=\ii t$. . 
For a given $2n\times 1$ vector $\bs{\vec b}=\begin{pmatrix} \bs{b}\\ \bs{\tilde b}\end{pmatrix}$, introduce the $2n\times 2n$ matrix
\begin{equation}\label{eq:def-A}
{\mathcal A}(\bs{\vec b}) =\frac 1{\sqrt{n}}\begin{pmatrix}
|z| \bs{\Upsilon}\big(\bs{\vec b}\,\big) A^\tran & \bs{\Upsilon}\big(\bs{\vec b}\,\big) \Delta(\bs{b})A\\
\\
\bs{\Upsilon}\big(\bs{\vec b}\,\big) \widetilde \Delta(\bs{\tilde b})A^\tran &|z| \bs{\Upsilon}\big(\bs{\vec b}\,\big) A
\end{pmatrix}\ .
\end{equation}
We can compactly express this equation as
$
\bs{\vec \varepsilon} ={\mathcal A}(\E  \bs{\vec g}\, )\odot {\mathcal A}(\pvec )\  \bs{\vec \varepsilon}\, +\,  \vec{\mathcal O}_{\eta} \left( n^{-3/2}\right)
$
and easily prove that 
\begin{equation}\label{eq:approx-epsilon}
\bs{\vec \varepsilon} \quad =\quad {\mathcal A}(\pvec)\odot {\mathcal A}(\pvec )\  \bs{\vec \varepsilon}\, +\,  \vec{\mathcal O}_{\eta} \left( 
\| \bs{\vec \varepsilon}\,\|^2_{\infty} +
n^{-3/2}\right)\, .
\end{equation}
In Appendix \ref{app:linear-algebra} we prove that $I-{\mathcal A}(\pvec )\odot {\mathcal A}(\pvec )$ is invertible and that 
\begin{equation}\label{eq:invertibility-I-AA}
\lrn (I-{\mathcal A}(\pvec )\odot {\mathcal A}(\pvec ))^{-1}\rrn_{\infty} \ =\  \Oeta{1}\ ,\quad \eta=\ii t\, .
\end{equation}
Combining \eqref{eq:approx-epsilon} and \eqref{eq:invertibility-I-AA}, we finally end up with
$
\| \bs{\vec \varepsilon}\,\|_{\infty}   ={\mathcal O}_{\eta} \left( 
\| \bs{\vec \varepsilon}\,\|^2_{\infty} +
n^{-3/2}\right)\, ,
$ for $\eta=\ii t$. Following Notation \ref{not:asymp}, we rewrite this estimate as 
\begin{equation}\label{cond:epsilon}
\| \bs{\vec \varepsilon}\,\|_{\infty}   \le \kappa(\ii t)  \left( 
\| \bs{\vec \varepsilon}\,\|^2_{\infty} +
n^{-3/2}\right)\quad\textrm{for}\quad \kappa(\ii t)= \frac{C t^{c_1}}{t^{c_0} \wedge 1}\, .
\end{equation} 
The quadratic polynomial $P(X)=\kappa(\ii t)X^2-X+\kappa(\ii t) n^{-3/2}$ admits two distinct real roots as long as $1-4\kappa^2(\ii t)n^{-3/2}>0$. 
Denote by $\Gamma_n(\alpha,\beta)$ the set
$$
\Gamma_n(\alpha,\beta)=\left\{ \eta=\ii t,\, t>0,\,  t\in (n^{-\alpha},n^\beta)\right\}
$$
If $\eta=\ii t\in \Gamma_n(\alpha,\beta)$, then 
$
4\kappa^2(\ii t) n^{-3/2} \le C^2 n^{2c_0\alpha + 2c_1 \beta - 2/3}. 
$ The r.h.s. goes to zero as long as $c_0 \alpha + c_1 \beta <1/3$.  This can be fulfilled for $\alpha, \beta>0$ small enough, which is supposed to hold henceforth.

Then for $n\ge n_0(\alpha,\beta)$ large enough, $1-4\kappa^2(\ii t)n^{-3/2}>0$, thus $P(X)$ admits two distinct real roots
and \eqref{cond:epsilon} yields 
\begin{equation}\label{cond:trinome}
\| \bs{\vec \varepsilon}\,\|_{\infty} \ \le\ X_1(\ii t) := \frac{1-\sqrt{ 1 - 4 \kappa^2(\ii t) n^{-3/2}}}{2\kappa(\ii t)} \qquad \textrm{or} \qquad \|\bs{\vec \varepsilon}\,\|_{\infty}\  
\ge\ X_2(\ii t) := \frac{1+\sqrt{ 1 - 4 \kappa^2(\ii t) n^{-3/2}}}{2\kappa(\ii t)}\, .
\end{equation}
Since for $n\ge n_0(\alpha,\beta)$ and $\eta\in \Gamma_n(\alpha,\beta)$ the functions $\eta\mapsto \| \bs{\vec \varepsilon}\|_{\infty}(\eta), X_1(\eta), X_2(\eta)$ are continuous, only one of the two conditions \eqref{cond:trinome} can hold uniformly in $\Gamma_n(\alpha,\beta)$. For $\eta=\ii$ we have that $\kappa(\ii)=C$ and $X_2(\ii)= \frac{1+\sqrt{1-4C^2n^{2/3}}}{2C}={\mathcal O}(1)$, but $\|\bs{\vec \varepsilon}\, \|_{\infty} \xrightarrow[n\to\infty]{} 0$ by the first part of the proposition. Hence the condition 
$\|\bs{\vec \varepsilon}\,\|_{\infty}(\ii)\  
\ge\ X_2(\ii)$ cannot hold and necessarily
$$
\| \bs{\vec \varepsilon}\,\|_{\infty} \quad \le\quad X_1(\eta) = \frac{2\kappa(\eta)n^{-3/2}}{1+\sqrt{ 1 - 4 \kappa^2(\eta) n^{-3/2}}} \quad \le
\quad  2\kappa(\eta)n^{-3/2}\qquad \forall \eta\in \Gamma_n(\alpha,\beta)\, .
$$
We have proved so far that for $n\ge n_0(\alpha,\beta)$ large enough and $\eta\in \Gamma_n(\alpha,\beta)$, $\| \bs{\vec \varepsilon}\,\|_{\infty}=\Oeta{n^{-3/2}}$.

\correction{
\textcolor{red}
{The crux of the proof lies in the following proposition
\begin{prop}\label{prop:inversion} Let the matrix ${\mathcal A}$ be as in \eqref{eq:def-A}. Under the assumptions of Proposition~\ref{prop:QR}, we have that the matrix 
$$
I-{\mathcal A}(\E  \bs{\vec g}\, )\odot {\mathcal A}(\pvec )
$$ is invertible and $$
\lrn (I-{\mathcal A}(\E  \bs{\vec g}\, )\odot {\mathcal A}(\pvec ))^{-1}\rrn_{\infty} \ =\  \Oeta{1}\ .
$$
\end{prop}
The proof of Proposition \ref{prop:inversion} is postponed to Appendix \ref{app:linear-algebra}.}}

We are now in position to conclude: 
\correction{
\textcolor{red}{
Notice that 
$$
|\E G_{ii} - p_i| = | \bs{e}_i^* \vecepsilon| = \left| \bs{e}_i^* \left(I-{\mathcal A}(\E  \bs{\vec g}\, )\odot {\mathcal A}(\pvec )\right)^{-1} \vec{\mathcal O}_{\eta}\left(\frac1{n^{3/2}}\right)\right| = \Oeta{\frac 1{n^{3/2}}}\ ,
$$ 
where $\bs{e}_i$ is the canonical $2n\times 1$ vector $(\delta_{ij},\ j\in [2n])$.  A similar estimate holds for $|\E \tilde G_{ii} - \tilde p_i|$. 
Since these estimates are uniform in $i\in [n]$, we obtain:
$$
\bs{\varepsilon}_{\max} :=\max_{i\in [n]}\left( |\E\, G_{ii} - p_i|, |\E\, \widetilde G_{ii} - \tilde p_i|\right) =  \Oeta{\frac 1{n^{3/2}}}\ .
$$
Finally,
}}
\begin{eqnarray*}
\left|  \frac 1n \sum_{i\in [n]}  \E\, G_{ii} - \frac 1n \sum_{i\in [n]} p_i\right| &\le & \| \vec{\bs \varepsilon}\|_{\infty} \ \le\ \Oeta{n^{-3/2}}\ .\\
\end{eqnarray*} 
The same arguments apply verbatim for the term $\frac 1n \sum_i (\E\, \widetilde G_{ii} - \tilde p_i)$. Consider now the term
\begin{align*}
\frac 1n \tr \E F - \frac 1n \tr B 
&\stackrel{(a)}= \frac 1n \sum_{i=1}^n \left\{ 
\frac z{-([V\E \bs{\tilde g}]_i+\eta)} \E \widetilde G_{ii} - \frac z{-([V\bs{\tilde p}]_i+\eta)} \tilde p_i
\right\} +\Oeta{n^{-3/2}} \ ,
\end{align*}
where $(a)$ follows from \eqref{eq:2} and \eqref{def:B}. One can now apply the same arguments as previously and handle similarly the term 
$\frac 1n \tr \E F'(z,\eta) - \frac 1n \tr B'(z,\eta)$. 
This completes the proof of Proposition~\ref{prop:QR}.

\subsection{Proof of Theorem \ref{th:L->nu}}\label{prf:L->nu} 

The convergence $\check L_{n,z} \sim \check\nu_{n,z}$ is a direct consequence of Theorem \ref{th:convergence-QR}. Now, it is easy to prove with the help of the law of large numbers that a.s.
$$
\limsup_n \int |x|^2 \check{L}_{n,z}(\, dx)\ <\ \infty\ .
$$
This, together with Proposition \ref{prop:deteq}-\eqref{SD-prop5}, yields 
\[
\int_{\{ |x|\ge \varepsilon\}} \log |x| \, \check L_{n,z}(dx) - 
\int_{\{ |x|\ge \varepsilon\}} \log |x| \, \check\nu_{n,z}(dx)  
\ \xrightarrow[n\to\infty]{\text{a.s.}} \ 0 \ .
\]
The proof is complete.

\section{Proof of Theorem \ref{thm:master}: Analysis of the Master Equations} 
\label{sec:master}

\subsection{Regularized Master Equations}	\label{sec:RME}

Our first step is to introduce the so-called \emph{Regularized Master Equations}, which are obtained from the Schwinger--Dyson equations \eqref{cve0} by taking $\eta=\ii t$ and substituting $\bp(\ii t) = \ii \br(t)$.

Given two $n\times 1$ vectors $\boldsymbol{a}$ and $\boldsymbol{\tilde a}$ 
with nonnegative components and fixed numbers $s>0$ and $t\ge 0$, let $\boldsymbol{\vec{a}}=\begin{pmatrix} \boldsymbol{a} \\ \boldsymbol{\tilde a}\end{pmatrix}$ and define the following quantities
\begin{align}
\label{def-psi} 
\Psi(\boldsymbol{\vec{a}},t) &\ :=\ 
\diag \left( \frac 1{s^2+((V_n\bs{\tilde a})_i+t) ((V_n^\tran \bs{a})_i +t) }  
;\, i\in [n]\right) \ ,\\
&\ :=\ \diag( \psi_i(\boldsymbol{\vec{a}},t)\, ;\  i\in[n] )\ , \nonumber
\end{align} 
and
\begin{align}
\mathcal I(\bs{\vec{a}},t) &:=
\begin{pmatrix} 
\Psi(\bs{\vec{a}},t) V_n^\tran  & 0\\
0&\Psi(\bs{\vec{a}},t) V_n
\end{pmatrix} \bs{\vec{a}} 
+ t \begin{pmatrix} \Psi(\bs{\vec{a}},t) \boldsymbol{1}_n\\
                    \Psi(\bs{\vec{a}},t) \boldsymbol{1}_n \end{pmatrix} \label{eq:def-I}\ .\\
                    \nonumber
\end{align} 
We also define $\Psi(\boldsymbol{\vec{a}}):=\Psi(\boldsymbol{\vec{a}},0)$.\\
The proof of Proposition \ref{prop:MEt} amounts to showing that the vector equation 
$
\rvec=  \mathcal I(\rvec,t)$
admits a unique solution $\rvec=\rvec(s,t) \succ 0$. 
\begin{rem}
We shall also prove that for any initial vector $\rvec_0 \succcurlyeq 0$, the iterations 
$\rvec_{k+1} = \mathcal I(\rvec_k,t)$ converge to $\rvec$ as $k\to\infty$. 
\end{rem}

%

\begin{proof}[Proof of Proposition~\ref{prop:MEt}]
This follows immediately from Proposition \ref{prop:deteq} by setting ${\bs p}(\ii t)=\ii \br(t)$ and \\
$\tilde{\bs p}(\ii t)=\ii \brt(t)$.
\end{proof}

\subsection{Proof of Theorem \ref{thm:master}}

In the following propositions, we recall some known properties of nonnegative and
irreducible matrices.

\begin{prop}[{\cite[Theorems 1.1 and 5.5]{book-seneta}}]\label{prop:seneta-1}
Let $A$ and $B$ be two square matrices such that 
$0\preccurlyeq A\preccurlyeq B$. Then $\rho(A)\le \rho(B)$.
Moreover, if $B$ is irreducible, then $\rho(A)=\rho(B)$ implies that 
$A=B$. 
\end{prop}

\begin{prop}[{\cite[Theorem 1.6]{book-seneta}}]\label{prop:seneta-2} 
Let $A\succcurlyeq 0$ be a square and irreducible matrix, and let
$\bs{x}\posneq 0$ be a vector satisfying 
$A\bs{x} \preccurlyeq \bs{x}$. Then $\bs{x}\succ 0$ and $\rho(A)\le1$. 
Moreover, $\rho(A)=1$ if and only if $A\bs{x}=\bs{x}$. 
\end{prop}

The proof of the following lemma is deferred to Section \ref{sec:bddness} -- see Proposition \ref{prop:bdd-weak}.

\begin{lemma}		\label{lem:qbound} 
Let $V$ be a nonnegative and irreducible $n\times n$ matrix, and let 
$\rvec(s,t)$ be the solution of the regularized master equations \eqref{def:MEt}. 
Let $[a,b]\subset (0,\infty)$ and $\varepsilon>0$, then
$$
\sup_{(s,t)\in [a,b]\times [0,\varepsilon]} \| \rvec(s,t)\| <\infty\ .
$$
In particular, $\rvec(s,t)$ admits an accumulation point for fixed $s>0$ as 
$t\downarrow 0$.
\end{lemma}

Next we show that any accumulation point provided by the above lemma constitutes a solution to the Master Equations \eqref{def:ME2}:

\begin{lemma}[Existence of solutions to the Master Equations] 		\label{lem:qexist}		
Let $V$ and $\rvec(s,t)$ be as in Lemma \ref{lem:qbound}.
\begin{enumerate}
\item Let $s>0$. If $\rvecstar=( \brstar , \brtstar )  \succcurlyeq 0$ is an accumulation point for $\rvec(s,t)$ as $t\downarrow 0$, then 
\[
\begin{pmatrix}
\rvecstar\\
0
\end{pmatrix} \ =\ 
\begin{pmatrix}
\Psi(\rvecstar\,) V^\tran  & 0\\
0& \Psi(\rvecstar\,) V\\
\boldsymbol{1}^\tran & -\boldsymbol{1}^\tran 
\end{pmatrix}\rvecstar \ , 
\] 
where we recall that 
$\Psi(\rvecstar) = \diag (\psi_i(\rvecstar))_{i=1}^n$ with 
$
\psi_i(\rvecstar) =  (s^2+(V\brtstar)_i (V^\tran  \brstar)_i)^{-1}$. 
\item If moreover $s^2\in (0,\rhoV)$, then $\rvecstar \posneq 0$. 
\end{enumerate}
 \end{lemma}
\begin{proof}

The proof of part (1) is straightforward. 

We now prove part (2) of the lemma. Let $(t_k)$ be a positive sequence 
converging to zero in such a way that 
$\lim_{k\to\infty} \rvec(s,t_k)=\rvecstar$. 
Since $\rvec(s,t_k)$ satisfies
\eqref{def:MEt}, we have in particular that $\Psi(\rvec)(s,t_k)V \brt \prec \brt$ and $\Psi(\rvec)(s,t_k)V^\tran\br  \prec \br$.
From Proposition~\ref{prop:seneta-2} it follows that 
that $\rho(\Psi(\rvec)(s,t_k)V) =  \rho(\Psi(\rvec)(s,t_k)V^\tran) < 1$, and by the continuity of the spectral radius that $\rho(\Psi(\rvecstar)V)\le 1$. If $\rvecstar=0$, then 
$\Psi(\rvecstar)=s^{-2} I$ and $\rho(\Psi(\rvecstar)V)=s^{-2}\rhoV>1$ since
$s^2\in (0,\rhoV)$, which yields a contradiction. 
Necessarily, $\rvecstar\posneq 0$. 
\end{proof}

Theorem~\ref{thm:master}--\eqref{q:sys1} and \eqref{q:sys2} are now consequences of the following lemma. 

\begin{lemma}[Uniqueness of solutions to the Master Equations] 		\label{lem:qunique}	
Let $V$ be a nonnegative and irreducible $n\times n$ matrix, and let 
$\qvec \succcurlyeq 0$ be a solution of the system \eqref{def:ME2}, which 
exists by the previous lemma.
\begin{enumerate}
\item\label{s2>rho} If $s^2 \geq \rhoV$, then $\qvec=0$.
\item\label{s2<rho} 
If $s^2 \in (0,\rhoV)$, then $\qvec$ is unique as a solution 
of~\eqref{def:ME2} satisfying $\qvec \posneq 0$. This solution 
satisfies $\qvec\succ 0$.
\end{enumerate}
\end{lemma}

\begin{proof} 
We first prove Item \eqref{s2>rho}. Observe that $\Psi(\qvec)V^\tran $ is a
nonnegative irreducible matrix for all $\qvec \succcurlyeq 0$. Assume that
$\bs q \posneq 0$. Then $\rho(\Psi(\qvec)V^\tran ) = 1$ and $\bs q \succ 0$ 
by Proposition~\ref{prop:seneta-2}. From the equation 
$\boldsymbol{1}^\tran  \q = \boldsymbol{1}^\tran  \qtilde$ obtained 
from~\eqref{def:ME2}, we have $\qtilde \posneq 0$. Therefore, 
$\qtilde \succ 0$ by an argument similar to the one used for $\bs q$. 
Consequently, $(V^\tran  \bs q)_i (V \qtilde)_i > 0$ for all $i\in[n]$, 
leading to the contradiction 
\[
1 = \rho(\Psi(\qvec)V^\tran ) < \rho(s^{-2} V^\tran ) = s^{-2} \rhoV \leq 1
\]
where the strict inequality is due to Proposition~\ref{prop:seneta-1}. 
Hence $\qvec=0$.

We now turn to Item \eqref{s2<rho}. The argument $\qvec \posneq 0 
\Rightarrow \qvec \succ 0$ is identical to Item~\eqref{s2>rho}. 

The first step towards establishing uniqueness of the solution is showing that if 
$\qvec = 
(\q^\tran , \qtilde^\tran )^\tran $ and $\vec{\bs q'} = ((\bs q')^\tran , (\bs{\tilde q}')^\tran )^\tran $ 
are two positive solutions such that $\qvec \neq \vec{\bs q'}$, then 
$\Psi(\qvec) \neq \Psi(\vec{\bs q'})$. Assume the contrary. The equation 
$\q = \Psi(\qvec)V^\tran  \q$ shows that $1$ is the Perron--Frobenius eigenvalue 
of the irreducible matrix $\Psi(\qvec)V^\tran $ (Proposition~\ref{prop:seneta-2}). Since 
its eigenspace has the
dimension one, we get that $\q = \alpha \bs q'$ for some $\alpha > 0$.
A similar argument shows that $\qtilde = \tilde\alpha \bs{\tilde q}'$ for
some $\tilde\alpha > 0$. Using the assumption 
$\Psi(\qvec) = \Psi(\vec{\bs q'})$ again and inspecting the expressions of 
these terms, we get that $\alpha = \tilde\alpha^{-1}$. Moreover, the equations 
$\boldsymbol{1}^\tran  \q = \boldsymbol{1}^\tran  \qtilde$ and 
$\boldsymbol{1}^\tran  \bs q' = \boldsymbol{1}^\tran  \bs{\tilde q}'$ show that 
$\alpha = \tilde\alpha$. This implies that $\qvec = \vec{\bs q'}$, a 
contradiction. 

To establish the uniqueness, let us still consider the two positive solutions 
$\qvec \neq \vec{\bs q'}$. Write $\Psi = \diag(\psi_i) = \Psi(\qvec)$ and 
$\Psi' = \diag(\psi_i') = \Psi(\vec{\bs q}')$, and define the vectors  
\[
\bs\varphi = \begin{pmatrix} \varphi_1 \\ \vdots \\ \varphi_n \end{pmatrix} 
= V \bs{\tilde q} , \quad 
\bs{\tilde\varphi} = 
\begin{pmatrix} \tilde\varphi_1 \\ \vdots \\ \tilde\varphi_n \end{pmatrix} 
= V^\tran  \bs q , \quad 
\vec{\bs\varphi} = \begin{pmatrix} \bs\varphi \\ \bs{\tilde\varphi} 
\end{pmatrix}, 
\]
and their similarly defined analogues $\bs\varphi'$, $\bs{\tilde\varphi}'$,
and $\vec{\bs\varphi}'$. It holds by the irreducibility of $V$ that 
$\vec{\bs\varphi}, \vec{\bs\varphi}' \succ 0$. We now write 
\[
\varphi_i 
= \frac 1n \sum_{\ell=1}^n \sigma_{i,\ell}^2 \psi_\ell \varphi_\ell 
= \frac 1n \sum_{\ell=1}^n \sigma_{i,\ell}^2 \psi_\ell^2 
( s^2 \varphi_\ell + \varphi_\ell^2 \tilde\varphi_\ell ) 
\]
and a similar equation for $\tilde\varphi_i$, giving rise to the identity
\[
\vec{\bs\varphi} = \begin{pmatrix} 
s^2 V \Psi^2 & V \Psi^2 \bs\Phi^2 \\
V^\tran  \Psi^2 \widetilde{\bs\Phi}^2 & s^2 V^\tran  \Psi^2
\end{pmatrix} 
\vec{\bs\varphi} 
\]
where $\bs\Phi = \diag(\varphi_i)$ and 
$\widetilde{\bs\Phi} = \diag(\tilde\varphi_i)$. Equivalently, the nonnegative 
matrix
\[
K_{\vec{\bs q}} := 
\begin{pmatrix} 
s^2 \bs\Phi^{-1} V \Psi^2 \bs\Phi & 
\bs\Phi^{-1} V \Psi^2 \bs\Phi^2 \widetilde{\bs\Phi} \\
\widetilde{\bs\Phi}^{-1} V^\tran  \Psi^2 \widetilde{\bs\Phi}^2 \bs\Phi 
& s^2 \widetilde{\bs\Phi}^{-1} V^\tran  \Psi^2 \widetilde{\bs\Phi} 
\end{pmatrix} 
\]
satisfies $K_{\vec{\bs q}} \bs 1 = \bs 1$. Considering now the two
solutions $\vec{\bs\varphi}$ and $\vec{\bs\varphi}'$, we can write 
\begin{align*}
\varepsilon_i &:= 
\Bigl|\frac{\varphi_i - \varphi'_i}{\sqrt{\varphi_i\varphi'_i}}\Bigr| \\
&= 
\frac{1}{\sqrt{\varphi_i\varphi'_i}} 
\frac 1n \Bigl| 
\sum_{\ell=1}^n \sigma_{i,\ell}^2 ( \psi_\ell\varphi_\ell - 
\psi_\ell' \varphi_\ell' ) \Bigr| 
= 
\frac{1}{\sqrt{\varphi_i\varphi'_i}} 
\frac 1n \Bigl| 
\sum_{\ell=1}^n \sigma_{i,\ell}^2 \psi_\ell\psi'_\ell 
  ( (\psi_\ell')^{-1} \varphi_\ell - \psi_\ell^{-1} \varphi_\ell' ) \Bigr| 
\\
&= \frac 1n \Bigl| \sum_{\ell=1}^n 
\Bigl( \frac{\sigma_{i,\ell}^2 s^2 \psi_\ell \psi_\ell' 
             \sqrt{\varphi_\ell \varphi'_\ell}}{\sqrt{\varphi_i\varphi'_i}} 
     \frac{\varphi_\ell - \varphi'_\ell}{\sqrt{\varphi_\ell \varphi'_\ell}} 
+  \frac{\sigma_{i,\ell}^2 \psi_\ell \psi_\ell'
  \varphi_\ell \varphi'_\ell 
         \sqrt{\tilde\varphi_\ell \tilde\varphi'_\ell}} 
       {\sqrt{\varphi_i\varphi'_i}} 
     \frac{\tilde\varphi'_\ell - \tilde\varphi_\ell}
          {\sqrt{\tilde\varphi_\ell \tilde\varphi'_\ell}}  \Bigr)\Bigr|  \\
&\leq 
\frac 1n \sum_{\ell=1}^n 
\Bigl( \frac{\sigma_{i,\ell}^2 s^2 \psi_\ell \psi_\ell' 
             \sqrt{\varphi_\ell \varphi'_\ell}}{\sqrt{\varphi_i\varphi'_i}} 
\Bigl| \frac{\varphi_\ell - \varphi'_\ell}{\sqrt{\varphi_\ell \varphi'_\ell}} 
 \Bigr|
+  \frac{\sigma_{i,\ell}^2 \psi_\ell \psi_\ell'
  \varphi_\ell \varphi'_\ell 
         \sqrt{\tilde\varphi_\ell \tilde\varphi'_\ell}} 
       {\sqrt{\varphi_i\varphi'_i}} 
     \Bigl| \frac{\tilde\varphi_\ell - \tilde\varphi'_\ell}
          {\sqrt{\tilde\varphi_\ell \tilde\varphi'_\ell}} \Bigr| \Bigr) 
\end{align*}
for every $i \in [n]$, and we also have a similar inequality for 
$\tilde\varepsilon_i := | (\tilde\varphi_i - \tilde\varphi'_i) / 
\sqrt{\tilde\varphi_i\tilde\varphi'_i} |$. It results that the vector 
$\bs\varepsilon = ( \varepsilon_1,\ldots,\varepsilon_n, \tilde\varepsilon_1,
\ldots,\tilde\varepsilon_n)^\tran $ satisfies the inequality 
$\bs\varepsilon \preccurlyeq K_{\vec{\bs q},\vec{\bs q}'} \bs\varepsilon$, 
where
\[
K_{\vec{\bs q},\vec{\bs q}'} := 
\begin{pmatrix} 
s^2 (\bs\Phi \bs\Phi')^{-1/2} V \Psi \Psi'
 (\bs\Phi \bs\Phi')^{1/2} & 
(\bs\Phi \bs\Phi')^{-1/2} V \Psi \Psi'
\bs\Phi \bs\Phi' (\widetilde{\bs\Phi} \widetilde{\bs\Phi}')^{1/2} \\
(\widetilde{\bs\Phi}\widetilde{\bs\Phi}')^{-1/2} V^\tran  \Psi \Psi'
\widetilde{\bs\Phi} \widetilde{\bs\Phi}' (\bs\Phi \bs\Phi')^{1/2} 
& s^2 (\widetilde{\bs\Phi}\widetilde{\bs\Phi}')^{-1/2} V^\tran  \Psi \Psi' 
(\widetilde{\bs\Phi} \widetilde{\bs\Phi}')^{1/2} 
\end{pmatrix} , 
\]
and $\bs\Phi' = \diag(\varphi_i')$, 
$\widetilde{\bs\Phi}' = \diag(\tilde\varphi_i')$. 
By applying the Cauchy-Schwarz inequality to the scalar products 
$\bs x_m \bs 1, m = 1,\ldots, n$, where $\bs x_m$ is the row $m$ of 
$K_{\vec{\bs q},\vec{\bs q}'}$, we get that 
$K_{\vec{\bs q},\vec{\bs q}'} \bs 1 \preccurlyeq 
(K_{\vec{\bs q}} \bs 1) \odot (K_{\vec{\bs q}'} \bs 1) = \bs 1$. 

Now, for any $k\in\N$, we have
\[ 
\begin{pmatrix} V & V \\ V^\tran  & V^\tran  \end{pmatrix}^k 
\succcurlyeq 
\begin{pmatrix} V^k & V^k \\ (V^\tran )^k & (V^\tran )^k \end{pmatrix} . 
\]
Since $\vec{\bs\varphi}, \vec{\bs\varphi}' \succ 0$ and $V$ is 
irreducible, it holds that for any $(i,j) \in [2n]^2$, there exists 
$k \in [n]$ such that $[K_{\vec{\bs q},\vec{\bs q}'}^k]_{ij} > 0$, implying 
that $K_{\vec{\bs q},\vec{\bs q}'}$ is irreducible. 
Relying on these results, we will show that there exists $i\in[n]$ such that
$\bs x_i \bs 1 < 1$, i.e.\ Cauchy-Schwarz inequality is strict for this 
row vector. Proposition~\ref{prop:seneta-2} will then show that
$\rho( K_{\vec{\bs q},\vec{\bs q}'} ) < 1$. By consequence, the only 
solution to the inequality
$\bs\varepsilon \preccurlyeq K_{\vec{\bs q},\vec{\bs q}'} \bs\varepsilon$ will
be $\bs\varepsilon = 0$, contradicting the assertion $\qvec \neq \vec{\bs q'}$,
and the uniqueness of the solution of~\eqref{def:ME2} for 
$\qvec \posneq 0$ will follow. 

Recalling that $\Psi(\qvec) \neq \Psi(\vec{\bs q'})$, there exists $\ell\in[n]$
such that 
$\varphi_\ell \tilde\varphi_\ell\neq \varphi'_\ell \tilde\varphi'_\ell$. 
Since $V$ is irreducible, no column of this matrix is zero. Therefore, we can
choose $i\in[n]$ such that $\sigma^2_{i,\ell} > 0$. Consider the vector 
\[
\bs v_i := \Bigl( 
\Bigl( \frac{s \sigma_{i,m} \psi_m(\qvec) \sqrt{\varphi_m}}
{\sqrt{n} \sqrt{\varphi_i}} \Bigr)_{m=1}^n, 
\Bigl( \frac{\sigma_{i,m} \psi_m(\qvec) \varphi_m \sqrt{\tilde\varphi_m}}
{\sqrt{n} \sqrt{\varphi_i}} \Bigr)_{m=1}^n 
\Bigr) = 
\Bigl( (v_{1,m})_{m=1}^n,  (v_{2,m})_{m=1}^n \Bigr) 
\]
and his analogue $\bs v'_i$ (with the obvious notations). Consider also the
$2 \times 2$ matrix 
\[
M = \begin{pmatrix} v_{1,\ell} & v_{2,\ell} \\ v'_{1,\ell} & v'_{2,\ell} 
\end{pmatrix} = 
\begin{pmatrix} \frac{\sigma_{i,\ell} \psi_\ell(\qvec)}
{\sqrt{n}\sqrt{\varphi_i}} & \\
&  \frac{\sigma_{i,\ell} \psi_\ell(\vec{\bs q}')}
{\sqrt{n}\sqrt{\varphi'_i}} \end{pmatrix}
\begin{pmatrix} 
\sqrt{\varphi_\ell} & \sqrt{\tilde\varphi_\ell} \varphi_\ell \\
\sqrt{\varphi_\ell'} & \sqrt{\tilde\varphi_\ell'} \varphi_\ell' 
\end{pmatrix} 
\begin{pmatrix} s & \\ & 1 \end{pmatrix} 
:= M_1 M_2 M_3. 
\] 
Since $\det(M_2) = \sqrt{\varphi_\ell \varphi_\ell'}  
( \sqrt{\varphi'_\ell \tilde{\varphi}_\ell'} -   
  \sqrt{\varphi_\ell \tilde\varphi_\ell} ) \neq 0$, the vectors $\bs v_i$ and
$\bs v_i'$ are not collinear. Therefore,    
\[
\bs x_i \bs 1 = \langle \bs v_i, \bs v'_i \rangle < 
\| \bs v_i \|^2 \| \bs v'_i \|^2 = 
(K_{\vec{\bs q}} \bs 1)_i \, (K_{\vec{\bs q}'} \bs 1)_i = 1 . 
\] 
Lemma~\ref{lem:qunique} is proved. 
\end{proof}

It remains to establish Theorem~\ref{thm:master}--\eqref{q:contdif}. This is a 
consequence of following lemma, which also provides an expression for 
$\nabla \qvec(s)$ on $(0,\rhoV^{1/2})$.

\begin{lemma}
\label{q:dif} 
Assume that the nonnegative $n\times n$ matrix $V$ is irreducible. Then the 
function $s\mapsto \qvec(s)$ is continuous on $(0,\infty)$, and is 
continuously differentiable on $(0,\rhoV^{1/2}) \cup (\rhoV^{1/2}, \infty)$. 
Setting
\begin{gather*} 
M(s) = \begin{pmatrix} 
s^2 \Psi(\qvec(s))^2 V^\tran  & - \Psi(\qvec(s))^2 \widetilde{\bs\Phi}(s)^2 V  \\
- \Psi(\qvec(s))^2 {\bs\Phi}(s)^2 V^\tran  & s^2 \Psi(\qvec(s))^2 V
\end{pmatrix}, \\ 
{\bs\Phi}(s) = \diag(\varphi_i(s))_{i=1}^n, \ 
\widetilde{\bs\Phi}(s) = \diag(\tilde\varphi_i(s))_{i=1}^n, \quad 
\varphi_i(s) = (V \bs{\tilde q}(s))_i , \ 
\tilde\varphi_i(s) = (V^\tran  \bs{q}(s))_i,  \\
A(s) = \begin{pmatrix} I_{2n} - M(s) \\
\begin{pmatrix} \bs 1_n^\tran  & - \bs 1_n^\tran  \end{pmatrix} \end{pmatrix} 
\in \R^{(2n+1)\times 2n} ,  
\quad \text{and} \quad 
b(s) = - \begin{pmatrix} 
\Psi(\qvec(s))^2 V^\tran  \bs q(s) \\
\Psi(\qvec(s))^2 V \bs{\tilde q}(s)  \\ 
0 \end{pmatrix} \in \R^{2n+1} , 
\end{gather*}
the matrix $A(s)$ has a full column rank, and
\[
\nabla \qvec(s) = 2s A(s)^{-\text{L}} b(s),  
\] 
where $A(s)^{-\text{L}}$ is the left inverse of $A(s)$. On 
$(\rhoV^{1/2}, \infty)$, it holds that $\qvec(s) = \nabla \qvec(s) = 0$. 
\end{lemma}

\begin{proof} 
We already know that $\qvec(s) = 0$ on $[\rhoV^{1/2},\infty)$, and we can 
easily check that $\nabla \qvec(s) = 2s A(s)^{-\text{L}} b(s)$ on 
$(\rhoV^{1/2},\infty)$.  

Let us show that $\q(s)$ is continuous on $(0,\rhoV^{1/2})$.  
Fix $s\in (0, \rhoV^{1/2})$. When $u$ belongs to a small neighborhood of $s$ in
$(0, \rhoV^{1/2})$, the function $\qvec(u)$ is bounded by
Lemma~\ref{lem:qbound}, since $\qvec(u)$ is the limit as $t\downarrow 0$
of the bounded function $\qvec(u,t)$. Let $u_k \to_k s$ be such that
$\qvec(u_k) \to_k \qvec_*$. The vector $\qvec_*$ is clearly a solution
to~\eqref{def:ME2}.  Observing that $\rho(\Psi(\qvec(u_k)) V) = 1$, we
get by the continuity of the spectral radius that $\rho(\Psi(\qvec_*) V) = 1$.
If $\qvec_*$ were equal to zero, then we would have 
$\rho(\Psi(\qvec_*) V) = \rho(s^{-2} V) > 1$, a contradiction.  Therefore
$\qvec \posneq 0$, and by Lemma~\ref{lem:qunique}--\eqref{s2<rho}, 
$\qvec_* = \qvec(s)$ since $\qvec(s)$ is the only nonnegative and non zero
solution to~\eqref{def:ME2}.

To obtain the continuity of $\qvec(s)$ on $(0,\infty)$, all what remains to 
prove is that $\qvec(u) \to 0$ as $u \uparrow \rhoV^{1/2}$. Relying on
Lemma~\ref{lem:qbound}, take a sequence $u_k \uparrow_k \rhoV^{1/2}$
such that $\qvec(u_k) \to_k \qvec_*$. Then, the same argument as in the proof 
of Lemma~\ref{lem:qunique}--\eqref{s2>rho} shows that $\qvec_* = 0$. 

To establish the differentiability of $\qvec(s)$ on $(0,\rhoV^{1/2})$, we 
start by writing 
\[
q_i(s) = \psi_i \tilde\varphi_i = 
\psi_i^2 ( s^2 \tilde\varphi_i + \tilde\varphi_i^2 \varphi_i ) 
= \psi_i^2 ( s^2 ( V^\tran \bs q )_i 
   + \tilde\varphi_i^2 ( V \bs{\tilde q} )_i ) .  
\]
Doing a similar derivation for $\tilde q_i(s)$, we get the equation 
$\qvec(s) = N(s) \qvec(s)$, where 
\[
N(s) = \begin{pmatrix} 
s^2 \Psi(\qvec(s))^2 V^\tran & \Psi(\qvec(s))^2 \widetilde{\bs\Phi}(s)^2 V  \\
  \Psi(\qvec(s))^2 \bs\Phi(s)^2 V^\tran & s^2 \Psi(\qvec(s))^2 V
\end{pmatrix} .  
\]
As in the proof of Lemma~\ref{lem:qunique}--\eqref{s2<rho}, we can show that 
$N(s)$ is irreducible.  Thus, the Perron--Frobenius eigenvalue of $N(s)$ is 
equal to one, it is algebraically simple, and its associated eigenspace is 
generated by $\qvec(s)$. 

Now, given two real numbers $s, s' \in (0, \rhoV^{1/2})$ with $s\neq s'$,
we have 
\begin{align*}
q_i - q_i' &= 
   \psi_i \tilde\varphi_i - \psi_i' \tilde\varphi_i' 
= \psi_i \psi_i' \bigl( (\psi_i')^{-1} \tilde\varphi_i - 
\psi_i^{-1} \tilde\varphi_i' \bigr) \\
&=  \psi_i \psi_i' \bigl( 
- (s^2-s'^2) \tilde\varphi_i + s^2 ( \tilde\varphi_i - \tilde\varphi_i' ) 
- \tilde\varphi_i \tilde\varphi_i' ( \varphi_i - \varphi_i') 
\bigr)  
\end{align*}
where we set $q_i = q_i(s)$ and $q'_i = q(s')$, and we used the same notational
shortcut for all the other quantities. We thus have
\begin{align*}
\frac{q_i - q_i'}{s^2 - s'^2} &= 
- \Bigl( \Psi \Psi' V^\tran \bs q \Bigr)_i + 
\Bigl( s^2 \Psi \Psi' V^\tran \frac{\bs q - \bs q'}{s^2 - s'^2} \Bigr)_i
- \Bigl( \Psi \Psi' \widetilde{\bs\Phi} \widetilde{\bs\Phi}' V 
   \frac{\bs{\tilde q} - \bs{\tilde q}'}{s^2-s'^2} \Bigr)_i . 
\end{align*}
Doing a similar derivation for $\tilde q_i - \tilde q_i'$, we obtain
the system
\[
( I - \bs M(s,s') ) \frac{\qvec - \qvec'}{s^2 - s'^2}
= \bs a(s,s')
\]
where 
\begin{gather*} 
\bs M(s,s') = \begin{pmatrix} 
s^2 \Psi \Psi' V^\tran & 
       - \Psi \Psi' \widetilde{\bs\Phi} \widetilde{\bs\Phi}' V  \\
- \Psi \Psi' {\bs\Phi} {\bs\Phi}' V^\tran & s^2 \Psi \Psi' V
\end{pmatrix} \quad \text{and} \quad 
\bs a(s,s') = 
- \begin{pmatrix} \Psi\Psi' V^\tran \bs q \\
\Psi\Psi' V \bs{\tilde q} 
\end{pmatrix} . 
\end{gather*} 
Using in addition the identity $\sum q_i = \sum \tilde q_i$, we get
the system $\bs A(s,s') (\qvec - \qvec') / (s^2 - s'^2) = \bs b(s,s')$,  
where
\[
\bs A(s,s') = \begin{pmatrix} I - \bs M(s,s')  \\
\begin{pmatrix} \bs 1_n^\tran & - \bs 1_n^\tran \end{pmatrix} \end{pmatrix} 
\in \R^{(2n+1)\times 2n} 
\quad \text{and} \quad 
\bs b(s,s') = \begin{pmatrix} \bs a(s,s') \\ 0 \end{pmatrix} .
\]
By the continuity of $\qvec(s)$, $\bs A(s,s') \to A(s)$ and 
$\bs b(s,s') \to b(s)$ as $s' \to s$. 
It is easy to see that $( \bs x^\tran, \bs{\tilde x}^\tran )^\tran$ is an eigenvector
of $M(s)$ if and only if $( \bs x^\tran, - \bs{\tilde x}^\tran )^\tran$ is an 
eigenvector of $N(s)$. Thus, the right null space $I - M(s)$ is spanned by 
$\bs v(s) := ( \bs q(s)^\tran, - \bs{\tilde q}(s)^\tran )^\tran$. But we clearly have
$A(s) \bs v(s) \neq 0$, hence the matrix $A(s)$ has a full column rank. 
Thus, for $s'$ close enough to $s$, 
\[
\frac{\qvec - \qvec'}{s^2 - s'^2} = 
\bs A(s,s')^{-\text{L}} \begin{pmatrix} \bs a(s,s') \\ 0 \end{pmatrix} 
\xrightarrow[s'\to s]{} A(s)^{-\text{L}}  b(s) 
\]
which shows that $\qvec(s)$ is differentiable for any 
$s \in (0, \rhoV^{1/2})$, with the gradient 
$2s A(s)^{-\text{L}}  b(s)$. The continuity of this gradient follows
from the continuity of $A(s)$ and $b(s)$ and the fact that $A(s)$ has a
full column rank for any $s \in (0, \rhoV^{1/2})$. 
\end{proof}

\section{Bounding solutions to the Regularized Master Equations via graphical bootstrapping}
\label{sec:bddness}


In this section we are concerned with establishing bounds on the solution $\rvec(s,t)\succ0$ to the Regularized Master Equations \eqref{def:MEt} that are uniform in the regularization parameter $t>0$.
Here we will view the standard deviation profile $A$ and all parameters as fixed.
Hence, we fix $n\ge1$ and consider an arbitrary nonnegative $n\times n$ matrix $A=(\sigma_{ij})$.
Putting $V=\frac1nA\odot A$ and fixing $s,t>0$, we let $\rvec=\rvec(s,t)$ denote the unique solution to the Regularized Master Equations satisfying $\rvec\succ 0$, as is provided by Proposition \ref{prop:MEt}.

\subsection{Some preparation and proofs of Propositions \ref{prop:sigmin} and \ref{prop:symmetry}}
We begin by recording some key estimates and identities that will be used repeatedly in the sequel.
We may write the Regularized Master Equations \eqref{def:MEt} in component form as
\begin{equation}	\label{qsys}
\frac1{\rr_i} = \vp_i + \frac{s^2}{\tp_i},\quad \quad \frac1{\rt_i} = \tp_i+ \frac{s^2}{\vp_i}
\end{equation}
where
\begin{equation}	\label{def:vtvp}
\tp_i := t+ (V^\tran \br)_i,\quad\quad \vp_i := t+(V\brt)_i
\end{equation}
(from Proposition \ref{prop:MEt} we have $\rr_i,\rt_i>0$ for all $i\in [n]$, so we are free to take reciprocals).
Additionally from Proposition \ref{prop:MEt} we have trace identity:
\begin{equation}	\label{qtrace}
\frac1n\sum_{j=1}^n \rr_j= \frac1n\sum_{j=1}^n \rt_j.
\end{equation}
From \eqref{qsys} it is immediate that
\begin{equation}	\label{qbds}
\frac12 \le  \frac{\rr_i}{\min\big(\tp_i/s^2, 1/\vp_i\big)}, \;
\frac{ \rt_i}{ \min\big(\vp_i/s^2, 1/\tp_i\big)} \le 1
\end{equation}
for all $i\in [n]$.
We can similarly bound the product
\begin{equation}	\label{qproduct}
\rr_i\rt_i = \frac{\vp_i\tp_i}{(s^2+\vp_i\tp_i)^2}  \le  \frac1{s^2} \min\Big( \frac{\vp_i\tp_i}{s^2}, \frac{s^2}{\vp_i\tp_i}\Big) \le 1/s^2.
\end{equation}
Hence, if for some $i\in[n]$ one of $\rr_i,\rt_i$ is large, the other is small.
Finally, we note the trivial upper bounds
\begin{equation}	\label{qbds:trivial}
\rr_i,\, \rt_i \le 1/t \quad \forall i\in [n].
\end{equation}

We now prove Propositions \ref{prop:sigmin} and \ref{prop:symmetry}.

\begin{proof}[Proof of Proposition \ref{prop:sigmin}]
Assume towards a contradiction that $\frac{1}{n} \sum_{j=1}^n \rr_j>1/\smin$.
Then there exists $i\in [n]$ such that $\rr_i> 1/\smin$.
From \eqref{qbds} it follows that
\[
1/\smin < 1/\vp_i = 1/(t+(V \brt)_i)
\]
and so
\[
\smin> (V \brt)_i = \frac{1}{n}\sum_{j=1}^n \sigma_{ij}^2\rt_j \ge \frac{\smin^2}{n} \sum_{j=1}^n \rt_j.
\]
Rearranging, we find $\frac1n\sum_{j=1}^n \rt_j <1/\smin$. From \eqref{qtrace} it follows that $\frac1n\sum_{j=1}^n \rr_j <1/\smin$, a contradiction.
\end{proof}

\begin{proof}[Proof of Proposition \ref{prop:symmetry}] If $V=\frac1nA\odot A$ is symmetric, then the $2n$ regularized master equations merge into $n$ equations since $\br=\brt$. In fact since $V^\tran=V$ then if $\rvec=\begin{pmatrix} \br^\tran\,   \brt^\tran \end{pmatrix}^\tran$ is a solution of the Regularized Master Equations, so is $\boldsymbol{\check r}=\begin{pmatrix} \brt^\tran \,  \br^\tran \end{pmatrix}^\tran$. By uniqueness, $\br=\brt$. Hence the Regularized Master Equations write
$$
r_i=\frac{(V\br)_i +t}{s^2+ (V\br)_i^2}\ ,\quad i\in [n]\ .
$$
An elementary analysis of the function $f(x)=\frac{x}{s^2+x^2}$ yields
$
\sup_{x\in [0,\infty)} f(x) \le (2s)^{-1}
$.
Hence
\[
\frac 1n \sum_{i\in[n]} r_i =\frac 1n \sum_{i\in [n]} f(\, (V\br)_i + t) \ \le \ \frac1{2s} \qquad \textrm{for all $t>0$} .	\qedhere
\]
\end{proof}

Our main objective now is to establish the following, which immediately yields Theorem \ref{th:sufficient}.

\begin{prop} \label{prop:boundq} 	
Assume $\sigma_{ij}\le \smax$ for all $i,j\in[n]$ and some $\smax<\infty$, and that $A(\scut)$ is $(\delta,\kappa)$-robustly irreducible for some $\scut,\delta,\kappa\in (0,1)$ (see Definition \ref{def:robust}).
For every fixed $z\in \C\setminus \{ 0 \}$ there exists a constant $K=K(z,\scut,\smax,\delta,\kappa)<\infty$ such that
\[
\frac 1n \sum_{i=1}^n \rr_i \le K . 
\]
\end{prop}

In the proof of Proposition \ref{prop:sigmin} we were able to pass from a lower bound on a single component $\rr_i$ to an upper bound on the average $\frac1n \sum_{j=1}^n \rt_j$ in one line. 
This will not be possible when we allow some of the variances $\sigma_{ij}^2$ to be zero.
We will employ a bootstrap-type argument which we roughly outline as follows:
\begin{enumerate}
\item Assume towards a contradiction that $\frac1n\sum_{i=1}^n \rr_i$ is large. 
By the pigeonhole principle there exists $i_0\in [n]$ such that $\rr_{i_0}$ is large.
\item Use the estimates \eqref{qbds}--\eqref{qbds:trivial} together with assumptions on the connectivity properties of the associated directed graph to iteratively ``grow" the set of indices $i$ for which we know $\rr_i$ is large. 
\item Once we have shown $\rr_i$ is large for (almost) all $i\in [n]$, by \eqref{qproduct} it follows that $\rt_i$ is small for (almost) all $i\in [n]$. 
We then apply the trace constraint \eqref{qtrace} to derive a contradiction. 
\end{enumerate}
We emphasize that the key idea for our proofs to bound $\frac1n\sum_i \rr_i$ will be to play \eqref{qproduct} against the trace constraint \eqref{qtrace}.

Recall the graph-theoretic notation from Section \ref{sec:notation}.
In this section we abbreviate
\begin{equation}
\mN_+(i):= \mN_{A(\scut)}(i), \quad \mN_-(i):= \mN_{A(\scut)^\tran}(i)
\end{equation}
for the in- and out-neighborhoods of a vertex $i$ in the graph $\Gamma=\Gamma(A(\scut))$, 
and similarly define $\mN_-^{(\delta)}(i)$. 
For parameters $\alpha,\beta>0$ we define the sets
\begin{equation}	\label{def:ST}
S_\alpha = \{ i\in [n]: \rr_i \ge \alpha \tpmax\},\quad\quad T_\beta = \{ i\in [n] : \rt_i< \beta/\tpmax \}
\end{equation}
where here and in the sequel we write $\tpmax = \max_{i\in [n]}\tp_i$ and similarly $\vpmax = \max_{i\in [n]}\vp_i$.
(Note that by \eqref{def:vtvp} we have $\tpmax\ge t>0$.)

\subsection{Qualitative boundedness}	\label{sec:qualitative}

In this subsection we establish Proposition \ref{prop:qualitative} and Lemma \ref{lem:qbound}, which give an $n$-dependent bound on the components of $\br,\brt$ assuming only that the standard deviation profile is irreducible.
This was used in Section \ref{sec:master} with a compactness argument to establish existence of solutions to the Master Equations.
While Lemma \ref{lem:qbound} follows from Proposition \ref{prop:boundq} under the robust irreducibility assumption \ref{ass:expander} (which we assume for our main result), we prove this lemma separately for two reasons: 
\begin{itemize}
\item to show that \ref{ass:expander} is not needed for the conclusion of Theorem \ref{thm:master}, and 
\item to provide a cartoon for the more technical proof of Proposition \ref{prop:boundq}.
\end{itemize}
We will also establish some auxiliary lemmas that will be reused in the proof of Proposition \ref{prop:boundq} (in particular Lemmas \ref{lem:S14} and \ref{lem:Wbounds}).

Proposition \ref{prop:qualitative} and Lemma \ref{lem:qbound} are an immediate consequence of the following:
\begin{prop}\label{prop:bdd-weak} 
For any fixed $0<\snot\le \smax$ and $s_0>0$ there is a constant $K_0(n,s_0,\snot,\smax)$ such that the following holds. 
Let $s\ge s_0$, and suppose $A$ is irreducible with $\sigma_{ij}\in \{0\}\cup [\snot,\smax]$ for all $1\le i,j\le n$. Then
$\frac{1}{n}\sum_{i=1}^n \rr_i \le K_0.$
\end{prop}

We now begin the proof of Proposition \ref{prop:bdd-weak}.
First we dispose of the case that $s>2\smax$:

\begin{claim}	\label{claim:bigs}
Suppose $s>2\smax$. Then 
\begin{equation}
\rr_i, \rt_i \le \min\left( \frac{t}{s^2- \smax^2}, \frac1t\right) \le (s^2-\smax^2)^{-1/2}.
\end{equation}
\end{claim}

\begin{proof}
Suppose $\rr_{i^*} = \max_{i\in [n]} \rr_i$. 
From \eqref{qbds} we have
\[
\rr_{i^*} \le \frac1{s^2} \tp_i = \frac1{s^2} (t+ (V^\tran \br)_i) \le \frac1{s^2} (t+ \smax^2 \rr_{i^*}).
\]
Rearranging we obtain $\rr_{i^*} \le t/(s^2-\smax^2)$, which combines with \eqref{qbds:trivial} to give the desired uniform bound for $\rr_i$, $i\in [n]$. The same bound is obtained for $\rt_i$ by similar lines.
\end{proof}

Without loss of generality we take $\smax=1$. 
By Claim \ref{claim:bigs} we may assume $0<s_0\le s\le 2$.
We may also assume $t\le 1$.
Indeed, otherwise it follows from \eqref{qbds:trivial} that $\frac1n\sum_{i=1}^n \rr_i< 1$ and we are done.
Let $K>0$ to be chosen later depending on $n, \snot$ and $s_0$, but independent of $t$, and assume
\begin{equation}	\label{assume:qK}
\frac1n\sum_{i=1}^n \rr_i \ge K.
\end{equation}
We will derive a contradiction for $K$ sufficiently large. 

In the following lemma we use the irreducibility of $A_n$ to show that if $T_\beta$ is non-empty for some $\beta$ sufficiently small, then $T_{\beta'}=[n]$ for a somewhat larger value of $\beta'$.
This will allow us to assume a uniform lower bound on the components $\rt_i$. 

\begin{lemma}		\label{lem:beta0.qual}
There are positive constants $C_0(\snot, n)$, $\beta_0(\snot,n)$ such that for all $\beta\le\beta_0$, if $T_\beta$ is non-empty then $T_{C_0\beta}=[n]$.
\end{lemma}

\begin{proof}
Let $\beta>0$ to be taken sufficiently small depending on $\snot,n$, and suppose $T_\beta$ is non-empty. 
Then there exists $i\in[n]$ such that $\rt_i< \beta/\tpmax$.
From \eqref{qbds} it follows that
\[
\frac12\min(\vp_i/s^2,1/\tp_i) < \beta/\tpmax.
\]
Assuming $\beta\le1/2$ it follows that 
\[
2s^2\beta/\tpmax > \vp_i \ge (V \brt)_i \ge \snot^2 \frac1n \sum_{j\in \mN_+(i)} \rt_j
\]
and hence 
\begin{equation}
\rt_j < \frac{2s^2n}{\snot^2} \frac{\beta}{\tpmax} \quad\quad \forall j\in \mN_+(i).
\end{equation}
Again from \eqref{qbds}, if we further assume $\beta \le\snot^2/4s^2n$ then it follows that
\begin{equation}
\frac{4s^4n}{\snot^2}\frac{\beta}{\tpmax} > \vp_j\quad\quad \forall j\in \mN_+(i).
\end{equation}

Now let $k\in[n]$ be arbitrary. 
By the irreducibility of $A_n$ there exists a directed path in the associated digraph $\Gamma_n$ from vertex $k$ to vertex $i$ of length at most $n$.
Applying the above lines iteratively along each edge of the path we find
\[
\rt_k \le \left(\frac{2s^2n}{\snot^2}\right)^n \frac{\beta}{\tpmax}
\]
if we take $\beta\le \frac12(\frac{\snot^2}{2s^2n})^{n-1}$.
Since $k$ was arbitrary, the result follows by setting $C_0= (8n/\snot^2)^n$ and $\beta_0=\frac12(\snot^2/(8n))^{n-1}$ (here we have used our assumption $s\le 2$).
\end{proof}

If $T_{C_0\beta}=[n]$ for some $\beta\le\beta_0$ then by the trace identity \eqref{qtrace},
\begin{equation}	\label{ifTC0:start}
\frac1n\sum_{i=1}^n \rr_i =\frac1n\sum_{i=1}^n \rt_i \le C_0\beta/\tpmax.
\end{equation}
On the other hand, from \eqref{qbds} we have
\begin{equation}	\label{qj:ubtp}
\rr_j \le \tp_j/s^2 \le \tpmax / s^2
\end{equation}
for all $j\in [n]$.
In particular,
\begin{equation}	\label{qj:avgub}
\frac1n\sum_{j=1}^n \rr_j \le \tpmax/s^2
\end{equation}
Hence,
\begin{equation}
 \frac1n\sum_{i=1}^n \rr_i  \le \min\left(\frac{C_0\beta}{\tpmax}, \frac{\tpmax}{s^2}\right) \le (C_0\beta/s^2)^{1/2}\le 
  (C_0\beta_0/s_0^2)^{1/2},
\end{equation}
which contradicts \eqref{assume:qK} if $K$ is sufficiently large.
Hence we may assume $T_{\beta_0}$ is empty for $\beta_0(\snot,n)$ as in Lemma \ref{lem:beta0.qual}.
Thus,
\begin{equation}	\label{assume:tqi.lb}
\rt_i \ge \beta_0/\tpmax	\quad\quad \forall i\in [n].
\end{equation}

Now we find a value of $\alpha$ for which $S_\alpha$ is already of linear size:

\begin{lemma}		\label{lem:S14}
Assume $K\ge 2/s^2$. Then
$|S_{1/4}| \ge (s^2/4)n$. 
\end{lemma}

\begin{proof}
From our assumption and \eqref{qj:avgub},
\begin{equation}
2/s^2\le K\le \frac1n\sum_{j=1}^n \rr_j \le \tpmax/s^2,
\end{equation}
so $\tpmax\ge 2$.
Let $i\in [n]$ such that $\tp_i=\tpmax $.
We have
\[
\tpmax  = \tp_i = t+ \frac{1}{n} \sum_{j=1}^n \sigma_{ji}^2 \rr_j \le t+ \frac1{n}\sum_{j=1}^n \rr_j.
\]
Since $\tpmax \ge 2$ and $t\le 1$, 
$
\frac{1}{n}\sum_{j=1}^n \rr_j\ge \frac12\tpmax .
$
Now again by \eqref{qj:ubtp},
\[
\tpmax n/2 \le \sum_{j\in S_{1/4}}\rr_j + \sum_{j\in S_{1/4}^c} \rr_j \le (\tpmax /s^2)|S_{1/4}| + \tpmax n/4,
\]
and the result follows by rearranging. 
\end{proof}

Next we seek to show that we can enlarge $S_\alpha$ by lowering $\alpha$.
By irreducibility we can find a vertex $i^*\in S_\alpha$ that is connected to $S_\alpha^c$. We can use this to show that the average of the components $\rr_k$ over $S_\alpha^c$ is bounded below by $c\rr_{i^*}$ for some small $c>0$ depending on $\alpha,n,s,\snot$.
From the pigeonhole principle we obtain $k\in S_\alpha^c$ with $\rr_k\ge c\rr_{i^*}\ge c\alpha \tpmax$.
Taking $\alpha' = c\alpha$, we will then have shown $|S_{\alpha'}|\ge |S_{\alpha}|+1$.

We begin by relating the values of $\br,\brt$ on a fixed set of vertices $S\subset [n]$ to the values taken on $S^c$.
For an $n\times n$ matrix $M$ and $S,T\subset[n]$ nonempty we write $M_{S\times T}$ for the $|S|\times |T|$ submatrix of $M$ with entries indexed by $S\times T$. 
The following lemma will also be used in the proof of Proposition \ref{prop:boundq}.

\begin{lemma}	\label{lem:Wbounds}
Fix a nonempty set $S\subset[n]$, and recall the diagonal matrix $\Psi$ from \eqref{def-psi}.
The $|S|\times |S|$ matrix $\Psi_{S\times S}^{-1} - V_{S\times S}^\tran$ is invertible, and we denote its inverse
\begin{equation}	\label{def:W2}
W^S = (\Psi_{S\times S}^{-1} - V_{S\times S}^\tran)^{-1}.
\end{equation}
In terms of $W^S$ the restrictions of $\br,\brt$ to $S$ and $S^c$ satisfy
\begin{equation}	\label{def:W}
\br_S = W^S(t+ V_{S^c\times S}^\tran\br_{S^c}),\quad \brt_S = (W^S)^\tran (t+ V_{S\times S^c} \brt_{S^c}).
\end{equation}
Furthermore, the entries of $W^S$ satisfy the following bounds.
For all $i,j\in S$,
\begin{equation}	\label{Wbound0}
W^S_{ij} \ge 0.
\end{equation}
For all $j\in S$, 
\begin{equation}	\label{Wbound1}
\sum_{i\in S} W_{ij}^S \le \rt_j/t
\end{equation}
and if \eqref{assume:tqi.lb} holds for some $\beta_0>0$.
\begin{equation}	\label{Wbound2}
\sum_{i\in S}W_{ij}^S |\mN_+(i)\cap S^c| \le \Big(\frac{n\tpmax }{\beta_0\snot^2}\Big)\rt_j.
\end{equation}
\end{lemma}

\begin{proof}
Arguing as in the proof of Lemma \ref{lem:qexist}(2) (using Proposition \ref{prop:seneta-2}) we find the spectral radius of $\Psi_{S\times S}V_{S\times S}^\tran$ is strictly less than 1.
Hence, $(I - \Psi_{S\times S}V_{S\times S}^\tran)^{-1}$ has a convergent Neumann series, and it follows that 
\[
W^S = \Psi_{S\times S}(I - \Psi_{S\times S}V_{S\times S}^\tran)^{-1}  = \Psi_{S\times S}\sum_{k=0}^\infty (\Psi_{S\times S}V_{S\times S}^\tran)^k.
\]
is well defined. 
Furthermore, since all of the matrices in the above series have non-negative entries, \eqref{Wbound0} follows. 

\eqref{def:W} is quickly obtained by rearranging the equations \eqref{def:MEt}.

Now for \eqref{Wbound1} and \eqref{Wbound2}, let $j\in S$ be arbitrary. 
From the second equation in \eqref{def:W}, 
\[
\rt_j = \sum_{i\in S} W^S_{ij}\Big( t + \frac1n\sum_{k\in S^c} \sigma_{ik}^2 \rt_k\Big).
\]
In particular,
$
\rt_j\ge t\sum_{i\in S} W^S_{ij},
$
giving \eqref{Wbound1}, and
\[
\rt_j \ge \frac1n \sum_{i\in S} W^S_{ij}\sum_{k\in S^c} \sigma_{ik}^2\rt_k
\ge \frac{\beta_0\snot^2}{n\tpmax } \sum_{i\in S} W^S_{ij} |\mN_+(i)\cap S^c|
\]
which rearranges to give \eqref{Wbound2}.
\end{proof}

We can use Lemma \ref{lem:Wbounds} and the irreducibility of $A$ to establish the following:

\begin{lemma}[Incrementing $\alpha$]	\label{lem:alphaprime.qual}
Let $\alpha>0$ such that $1\le |S_\alpha|\le n-1$. 
If 
$K>\sqrt{2n}/(s_0^3\alpha)$
then there exists $\alpha'=\alpha'(\alpha,s_0, \snot,n)\in (0,\alpha)$ such that $|S_{\alpha'}|\ge |S_{\alpha}|+1$.
\end{lemma}

Let us conclude the proof of Proposition \ref{prop:bdd-weak} on the above lemma.
Putting $\alpha_0=1/4$, by Lemma \ref{lem:S14} we have $S_{\alpha_0}\ne \emptyset$. 
Taking $K$ sufficiently large depending on $s_0, \snot$ and $n$ we can iterate Lemma \ref{lem:alphaprime.qual} at most $n$ times to find $\alpha=\alpha(s_0, \snot,n)>0$ such that $S_\alpha=[n]$.
Then by \eqref{qtrace} and \eqref{qbds},
\begin{equation}
\alpha \tpmax  \le \frac1n \sum_{j=1}^n \rr_j =  \frac1n\sum_{j=1}^n \rt_j  \le \frac1{s^2\alpha \tpmax },
\end{equation}
so we have $\tpmax \le 1/(s\alpha)$.
Then again by \eqref{qbds} we have
\begin{equation}
\frac1n\sum_{j=1}^n \rr_j \le \frac{\tpmax }{s^2} \le \frac{1}{s^3\alpha} \le \frac{1}{s_0^3\alpha}
\end{equation}
and we are done.

\begin{proof}[Proof of Lemma \ref{lem:alphaprime.qual}]
We write $W=W^{S_\alpha}$. 
From the first equation in \eqref{def:W}, for any $i\in S_\alpha$ we have
\begin{equation}	\label{apq:1}
\alpha \tpmax  \le \rr_i = \sum_{j\in S_\alpha} W_{ij} \left( t+ \sum_{k\in S_\alpha^c} \sigma_{kj}^2 \rr_k\right).
\end{equation}
Suppose first that 
\begin{equation}	\label{apq:case1}
\sum_{j\in S_\alpha} W_{ij} > \frac{\alpha \tpmax }{2t}
\end{equation}
for some $i\in S_\alpha$.
Then from \eqref{Wbound1},
\[
\frac{\alpha \tpmax }{2t} < \sum_{j\in S_\alpha} W_{ij} \le \frac1t \sum_{j\in S_\alpha} \rt_j\le \frac{1}{s^2t} \sum_{j\in S_\alpha} \frac{1}{\rr_j} \le \frac{|S_\alpha|}{ts^2\alpha \tpmax },
\]
where in the third inequality we applied \eqref{qbds}.
Rearranging we have $\tpmax \le \sqrt{2|S|}/(s\alpha) \le \sqrt{2n}/(s_0\alpha)$ in this case.
On the other hand, from \eqref{qj:avgub} we have $K\le \tpmax/s_0^2$, and we obtain a contradiction if 
$K>\sqrt{2n}/(s_0^3\alpha)$.

Suppose now that \eqref{apq:case1} does not hold for any $i\in S_\alpha$.
Then rearranging \eqref{apq:1} we have
\begin{equation}	\label{apq:2}
\frac{\alpha \tpmax }{2} \le \sum_{j\in S_\alpha}\sum_{k\in S_\alpha^c} \sigma_{kj}^2 \rr_k W_{ij}
\end{equation}
for any $i\in S_\alpha$.
From the assumption that $\Sig_n$ is irreducible there exists $(i^*,j^*)\in S_\alpha\times S_\alpha^c$ such that $\sigma_{i^*j^*}\ge \snot$, i.e.\ $|\mN_+(i^*)\cap S_\alpha^c|\ge 1$.
From \eqref{Wbound2} it follows that
\begin{equation}
W_{i^*j} \le \left( \frac{n\tpmax }{\beta_0\snot^2}\right) \rt_j
\end{equation}
for all $j\in S_\alpha$.
Inserting this bound in \eqref{apq:2} we have
\begin{align*}
\frac{\alpha \tpmax }{2} 
&\le \frac{n\tpmax }{\beta_0 \snot^2} \sum_{j\in S_\alpha} \sum_{k\in S_\alpha^c} \sigma_{kj}^2 \rr_k\rt_j \le \frac{n|S_\alpha|}{\beta_0 \snot^2 s^2\alpha} \sum_{k\in S_\alpha^c}\rr_k,
\end{align*}
where in the second inequality we applied the bounds $\sigma_{ij}\le 1$ for all $i,j\in[n]$ and $\rt_j\le (s^2\alpha \tpmax )^{-1}$ for all $j\in S_\alpha$ (by \eqref{qproduct}).
Rearranging we have
\begin{equation}	\label{avgout.qual}
\sum_{k\in S_\alpha^c} \rr_k \ge \frac{\alpha^2 s^2\beta_0\snot^2}{2n|S_\alpha|}\tpmax .
\end{equation}
By the pigeonhole principle there exists $k\in S_\alpha^c$ such that
\begin{equation}
\rr_k \ge \frac{\alpha^2 s^2\beta_0\snot^2}{2n|S_\alpha||S_\alpha^c|}\tpmax  \ge \frac{\alpha^2 s_0^2\beta_0\snot^2}{2n^3}\tpmax.
\end{equation}
Setting $\alpha'= \alpha^2 s_0^2\beta_0\snot^2/(2n^3)$ we have $k\in S_{\alpha'}$. Also, since $\alpha'<\alpha$ we have $S_\alpha\subset S_{\alpha'}$.
Hence, $|S_{\alpha'}|\ge |S_\alpha|+1$ as desired.
\end{proof}

\subsection{Quantitative boundedness}	\label{sec:quantitative}

Now we prove Proposition \ref{prop:boundq}.
By rescaling the variance profile $V$ we may take $\smax=1$.
By Claim \ref{claim:bigs} we may assume $s\in (0,2]$.
As in the proof of Proposition \ref{prop:bdd-weak} we may assume $t\le 1$.
We may also assume $n$ is sufficiently large depending on $s,\snot, \delta$ and $\kappa$.
In the remainder of the section we make use of asymptotic notation $O(\,),$ $\ls,$ $\gs$, allowing implied constants to depend on the parameters $s,\snot,\delta$ and $\kappa$, but not on $n$ and $t$. 

As in the proof of Proposition \ref{prop:bdd-weak} we assume \eqref{assume:qK} holds
for some $K>0$ and aim to derive a contradiction for $K$ sufficiently large depending on $s,\snot,\delta$ and $\kappa$.
The argument follows the same general outline as the proof in the previous subsection.
We will reuse Lemmas \ref{lem:S14} and \ref{lem:Wbounds} as stated, but we will need versions of Lemmas \ref{lem:beta0.qual} and \ref{lem:alphaprime.qual} with constants independent of $n$.

\subsubsection{Lower bounding $\rt_i$}

The following is an analogue of Lemma \ref{lem:beta0.qual}.

\begin{lemma}		\label{lem:beta0}
There are positive constants $C_0(s, \snot, \delta,\kappa)$, $\beta_0(s, \snot, \delta,\kappa)$ such that for all $\beta\le\beta_0$, if $T_\beta$ is non-empty then $T_{C_0\beta}=[n]$.
\end{lemma}

\begin{proof}
Let $\beta>0$ to be taken sufficiently small and assume $T_\beta$ is non-empty. Fix an element $i_0\in T_\beta$.
We will grow the set $T_\beta$ in stages by enlarging $\beta$ by appropriate constant factors.
We do this by iterative application of the following:

\begin{claim}	\label{claim:iterateT}
Let $\beta,\eps_0\in (0,1/2]$, and assume $0<|T_\beta|\le (1-\eps_0)n$.
There exists $C=C(\snot, \delta,\eps_0)>0$ such that if $n$ is sufficiently large depending on $\kappa$ and $\eps_0$ then $|T_{C\beta}\setminus T_\beta|\ge (\delta\eps_0/2)n$.
\end{claim}

\begin{proof}
By the assumption that $A(\scut)$ is $(\delta,\kappa)$-robustly irreducible we have
\[
|\mN_-^{(\delta)}(T_\beta^c)\cap T_\beta| \ge \min(\kappa |T_\beta^c|,|T_\beta|) \ge \min(\kappa \eps_0n,1)\ge1
\]
if $n$ is sufficiently large.
Fix an element $i\in \mN_-^{(\delta)}(T_\beta^c)\cap T_\beta$. 
By definition we have $$|\mN_+(i) \cap T_\beta^c|\ge \delta|T_\beta^c|\ge \delta\eps_0n.$$
Next, we claim that for any $C>0$ we have
\begin{equation}	\label{bdd:goodout}
|\mN_+(i) \cap T_{C\beta}^c|\le \frac{2s^2}{C\snot^2} n.
\end{equation}
Indeed, since $i\in T_\beta$, by \eqref{def:ST}
and \eqref{qbds} we have
\[
\frac{\beta}{\tpmax} > \rt_{i} \ge \frac12 \min\left(\frac{\vp_{i}}{s^2}, \frac1{\tp_{i}}\right).
\]
Since $\beta\le 1/2$ it follows that the minimum is attained by the first argument. Thus
\begin{align*}
\frac{\beta}{\tpmax} &> \frac{\vp_i}{2s^2} > \frac{1}{2s^2} \frac1n\sum_{j=1}^n \sigma_{ij}^2 \rt_j 
\ge \frac{\snot^2}{2s^2} \frac1n \sum_{j\in \mN_+(i)}\rt_j.
\end{align*}
From Markov's inequality it follows that for any $C>0$, $\rt_j < C\beta/\tpmax$ for all but at most $(2s^2/C\snot^2)n$ values of $j\in \mN_+(i)$, which gives \eqref{bdd:goodout}.
Combining these estimates and taking $C=8s^2/(\snot^2\delta\eps_0)$ we have
\begin{align*}
|T_{C\beta}\setminus T_\beta| &\ge |\mN_+(i)\cap T_\beta^c|-|\mN_+(i)\cap T_{C\beta}^c| \ge \left(\delta\eps_0 - \frac{2s^2}{C\snot^2}\right)n \ge (\delta\eps_0/2)n
\end{align*}
as desired.
\end{proof}

Applying the above claim iteratively with $\eps_0=s^2/8$ we obtain $C'(\snot, \delta,s)<\infty$ such that if $\beta$ is sufficiently small depending on $\snot,\delta,s$ and $n$ is sufficiently large depending on $\kappa,s$, then
\begin{equation}	\label{TCprime}
|T_{C'\beta}|\ge (1-s^2/8)n.
\end{equation}

Now let $C_0>0$ to be chosen later, and towards a contradiction suppose $T_{C_0\beta}\ne [n]$. 
Then there exists $i\in [n]$ such that $\rt_i\ge C_0\beta/\tpmax$. 
From \eqref{qbds} we have the upper bound $\rt_i\le \vpmax/s^2$, so we conclude $\vpmax \ge C_0s^2\beta/\tpmax$.
Now from our assumption $K\le \frac1n\sum_{j=1}^n \rr_j = \sum_{j=1}^n \rt_j$, if $K$ is sufficiently large depending on $s$ then the same argument as in the proof of Lemma \ref{lem:S14} shows that $\rt_j\ge \vpmax/4$ for at least $(s^2/4)n$ values of $j\in [n]$.
Thus, 
$
\rt_j \ge  C_0s^2\beta/(4\tpmax)
$
for at least $(s^2/4)n$ values of $j\in [n]$, i.e.\ $|T_{C_0s^2\beta/4}|<(1-\frac{s^2}{4})n$. 
Taking $C_0= 4C'/s^2$ we contradict \eqref{TCprime}, and we conclude $T_{C_0\beta}=[n]$.
\end{proof}

Now by the same lines as in \eqref{ifTC0:start}--\eqref{assume:tqi.lb} we conclude
\begin{equation}	\label{assume:tqi.lb2}
\rt_i \ge \beta_0/\tpmax \quad\quad\forall i\in [n]
\end{equation}
for some $\beta_0(s,\snot,\delta,\kappa)>0$.
Note we are now free to use the estimates in Lemma \ref{lem:Wbounds} with this value of $\beta_0$.

\subsubsection{Upper bounding $\rr_i$}

Here our task is essentially to modify the proof of Lemma \ref{lem:alphaprime.qual} to show we can take $\alpha'$ sufficiently small and \emph{independent of $n$} such that $|S_{\alpha'}\setminus S_\alpha|\gs n$, rather than merely nonempty.
We can then conclude the proof by iterating this fact a bounded number of times.

Let us summarize the key new ideas. 
In the proof of Lemma \ref{lem:alphaprime.qual} we used the irreducibility assumption to find an element $i^*\in S_\alpha$ such that the average of the components $\rr_k$ over $S_\alpha^c$ was of order $\gs_{\alpha, n} \rr_{i^*} \ge \alpha\tpmax$ (see \eqref{avgout.qual}).
In a similar spirit, Lemma \ref{lem:avgLB} below controls the average of $\rr_k$ over $k\in S_{\alpha}^c$ from below by the average of $\rr_i$ over $i\in U_0$, for a set $U_0\subset S_{\alpha}$ that is \emph{densely connected} to $S_\alpha^c$.
By averaging over a large set $U_0$ we are able to use the full strength of the bounds in Lemma \ref{lem:Wbounds} and avoid any dependence of the constants on $n$. 

Proceeding na\"ively, one can then use Lemma \ref{lem:avgLB} to deduce 
$
|S_{c_0\alpha^2}\setminus S_\alpha| \ge c_0 \alpha |S_\alpha^c|
$
for a sufficiently small constant $c_0=c_0(s,\snot,\delta,\kappa)>0$.
However, when iterating this bound over a sequence of values $\alpha_{k+1} = c_0\alpha_{k}^2 $, the sets $S_{\alpha_k}$ grow by an exponentially decreasing proportion of $n$, so this is not enough to find a value of $\alpha$ for which $|S_{\alpha}|$ is close to $n$.

Instead, in Lemma \ref{lem:alphaprime} we are able to grow $S_\alpha$ by a constant factor using a nested iteration argument, which we now describe.
We would like to find some value of $\alpha'\in (0,\alpha)$ for which
\begin{equation}	\label{spose.alpha}
|S_{\alpha'}\setminus S_{\alpha}|\ge c |S_{\alpha'}|
\end{equation}
where $c>0$ is small constant.
Suppose that \eqref{spose.alpha} fails.
By the expansion assumption, we know that $S_{\alpha'}$ contains a fairly large set $U=\mN_-^{(\delta)}(S_{\alpha'}^c)\cap S_{\alpha'}$ (of size at least $\min(|S_{\alpha'}|, \kappa|S_{\alpha'}^c|)$) that is densely connected to $S_{\alpha'}^c$.
In particular, if $c$ is sufficiently small depending on $\kappa$, then $U$ must have large overlap with $S_{\alpha}$. 
Denoting the overlap by $U_0$, Lemma \ref{lem:avgLB} can now be applied to deduce
\[ 
|S_{c_0\alpha\alpha'}\setminus S_{\alpha'}| \ge c_0 \alpha |S_{\alpha'}^c|
\]
for some $c_0=c_0(s,\snot,\delta,\kappa)>0$ sufficiently small.
The key is that the constant of proportionality on the right hand side is independent of $\alpha'$.
Hence, for fixed $\alpha$, as long as \eqref{spose.alpha} fails we can iteratively lower $\alpha'$ to increase $|S_{\alpha'}\setminus S_{\alpha}|$ by an amount $\gs \alpha |S_{\alpha'}^c|$, until eventually \eqref{spose.alpha} holds.
This whole procedure can then be iterated a bounded number of times to obtain $\alpha''$ such $|S_{\alpha''}|$ is close to $n$.

Having motivated the key ideas, we turn now to the proofs.

\begin{lemma}	\label{lem:avgLB}
Let $\alpha\in (0,1)$ and suppose that $0<|S_\alpha|\le (1-\delta/2)n$. 
If $K$ is sufficiently large depending on $\alpha, s,\snot,\delta,\kappa$, 
then for any $U_0\subset \mN_-^{(\delta)}(S_\alpha^c)\cap S_\alpha$ with $|U_0|\ge \frac1{10} |\mN_-^{(\delta)}(S_\alpha^c)\cap S_\alpha|$ we have
\begin{equation}
\frac1{|S_\alpha^c|}\sum_{k\in S_\alpha^c} \rr_k \gs \frac{\alpha}{|S_\alpha|} \sum_{i\in U_0} \rr_i.
\end{equation}
\end{lemma}

\begin{proof}
First we prove the comparison
\begin{equation}	\label{SaS0:compare}
\sum_{i\in U_0} \rr_i \ge 2 \sum_{j\in S_\alpha}\rt_j
\end{equation}
assuming $K$ is sufficiently large depending on $\alpha, s,\snot,\delta,\kappa$.
Indeed, if \eqref{SaS0:compare} does not hold, then by the fact that $U_0\subset S_\alpha$ and  \eqref{qproduct},
\[
\alpha \tpmax|U_0| \le \sum_{i\in U_0} \rr_i < 2\sum_{j\in S_\alpha} \rt_j \le \frac{2|S_\alpha|}{s^2\alpha\tpmax}.
\]
Rearranging we have
\[
\tpmax \le \frac1\alpha\left( \frac{2|S_\alpha|}{s^2 |U_0|}\right)^{1/2} \ls \frac1{\alpha} \left(\frac{|S_\alpha|}{\min(|S_\alpha|,|S_\alpha^c|)}\right)^{1/2} \ls 1/\alpha
\]
where in the second bound we applied the robust irreducibility assumption and our assumed bounds on $U_0$ and $S_\alpha$, and in the bound we used that both $S_\alpha$ and its complement are of linear size in $n$.
From our assumption \eqref{assume:qK}, \eqref{qbds} and the above it follows that
\begin{equation}
K\le \frac1n\sum_{i=1}^n \rr_i \le \tpmax/s^2 \ls 1/\alpha.
\end{equation}
Taking $K$ sufficiently large depending on $\alpha, s,\snot,\delta$ and $\kappa$, we may assume \eqref{SaS0:compare} holds.

From \eqref{def:W} and Lemma \ref{lem:Wbounds} we have
\begin{align*}
\sum_{i\in U_0} \rr_i &= t\left( \sum_{j\in {S_\alpha}} \sum_{i\in U_0} W_{ij}^{S_\alpha}\right) + \frac1n \sum_{j\in {S_\alpha}}\sum_{k\in {S_\alpha}^c} \left( \sum_{i\in U_0} W_{ij}^{S_\alpha} \right) \sigma_{kj}^2 \rr_k
\le \sum_{j\in {S_\alpha}} \rt_j + \frac1n \sum_{j\in {S_\alpha}}\sum_{k\in {S_\alpha}^c} \left( \sum_{i\in U_0} W_{ij}^{S_\alpha} \right) \sigma_{kj}^2 \rr_k,
\end{align*}
where in the bound we have applied \eqref{Wbound0} and \eqref{Wbound1}.
Applying \eqref{SaS0:compare} and rearranging yields
\begin{equation}	\label{avgLB:1}
\sum_{i\in U_0} \rr_i \le \frac2n\sum_{j\in S_\alpha} \sum_{k\in S_\alpha^c} \left(\sum_{i\in U_0} W_{ij}^{S_\alpha}\right) \sigma_{kj}^2 \rr_k.
\end{equation}
Now since $U_0\subset \mN_-^{(\delta)}(S_\alpha^c)$, for any $i\in U_0$ we have $|\mN_+(i) \cap S_\alpha^c| \ge \delta|S_\alpha^c|$. 
Together with \eqref{Wbound0} and \eqref{Wbound2} this implies	

\[
\delta|S_\alpha^c| \sum_{i\in U_0}W_{ij}^{S_\alpha} \le \sum_{i\in S_\alpha}W_{ij}^{S_\alpha}|\mN_+(i) \cap S_\alpha^c| \le \left( \frac{n\tpmax}{\beta_0\snot^2}\right) \rt_j.
\]
Rearranging we obtain a bound on $\sum_{i\in U_0}W_{ij}^{S_\alpha}$, which we substitute in \eqref{avgLB:1} to obtain
\begin{align*}
\sum_{i\in U_0}\rr_i 
&\;\le\; \frac{2\tpmax}{\beta_0\snot^2 \delta|S_\alpha^c|} \sum_{j\in S_\alpha}\sum_{k\in S_\alpha^c} \sigma_{kj}^2 \rt_j\rr_k \;\le\; \frac{2|S_\alpha|}{s^2\alpha\beta_0\snot^2 \delta|S_\alpha^c|}\sum_{k\in S_\alpha^c} \rr_k,
\end{align*}
where in the second inequality we applied \eqref{qproduct} to bound $\rt_j\le 1/s^2\alpha\tpmax$ for all $j\in S_\alpha$.
The result now follows by rearranging.
\end{proof}

\begin{lemma}	\label{lem:alphaprime}
For any $\alpha\in (0,1)$ there exists $\alpha'=\alpha'(\alpha,s,\scut,\delta,\kappa)>0$ such that either
\begin{equation}	\label{ap:alt1}
|S_{\alpha'}|\ge (1-\delta/2)n
\end{equation}
or
\begin{equation}	\label{ap:alt2}
|S_{\alpha'}\setminus S_{\alpha}| \ge \frac12\min (|S_{\alpha'}|, \kappa |S_{\alpha'}^c|)
\end{equation}
(or both).
\end{lemma}

\begin{proof}
For $\alpha'\in (0,\alpha)$ denote by $P(\alpha')$ the statement that at least one of \eqref{ap:alt1} and \eqref{ap:alt2} holds.
We will show that while $P(\alpha')$ fails, we can lower $\alpha'$ by a controlled amount to increase the size of $S_{\alpha'}\setminus S_\alpha$ by a little bit. We can then iterate this until $P(\alpha')$ holds. 

Let $\alpha'\in (0,\alpha)$ be arbitrary and assume $P(\alpha')$ fails. 
We claim there exists $c_0(s,\scut,\delta,\kappa)>0$ such that
\begin{equation}
\frac{1}{|S_{\alpha'}^c|} \sum_{k\in S_{\alpha'}^c} \rr_k  \ge c_0 \alpha\alpha' \tpmax. 		\label{lb:qtout}
\end{equation}
Put
$
U_0 = \mN_-^{(\delta)}(S_{\alpha'}^c) \cap S_\alpha.
$
By the robust irreducibility assumption and the fact that \eqref{ap:alt2} fails,
\begin{align}
|U_0| 
&\ge |\mN_-^{(\delta)}(S_{\alpha'}^c)\cap S_{\alpha'}| - |S_{\alpha'}\setminus S_\alpha| \ge \frac12 |\mN_-^{(\delta)}(S_{\alpha'}^c)\cap S_{\alpha'}|	\label{inca:1}\\
&\ge \frac12\min (|S_{\alpha'}|, \kappa |S_{\alpha'}^c|).	\label{inca:2}
\end{align}
By \eqref{inca:1} and Lemma \ref{lem:avgLB},
\begin{align*}
\frac{1}{|S_{\alpha'}^c|} \sum_{k\in S_{\alpha'}^c} \rr_k 
&\,\gs\, \frac{\alpha'}{|S_{\alpha'}|}\sum_{i\in U_0}\rr_i	
\,\ge\,   \alpha\alpha' \tpmax  \frac{|U_0|}{|S_{\alpha'}|}	
\,\gs \, \alpha\alpha' \tpmax ,
\end{align*}
where in the last inequality we applied \eqref{inca:2} and the fact that \eqref{ap:alt1} fails.
This gives \eqref{lb:qtout} as desired.

Now denoting
\[
U' = \Big\{k\in S_{\alpha'}^c: \rr_k \ge \frac12c_0 \alpha\alpha' \tpmax \Big\}
\]
we have
\begin{align*}
\sum_{k\in S_{\alpha'}^c}\rr_k 
&\le \sum_{k\in U'}\rr_k + \sum_{k\in S_{\alpha'}^c \setminus U'} \rr_k \le \alpha ' \tpmax |U'| + \frac12c_0 \alpha\alpha' \tpmax |S_{\alpha'}^c|,
\end{align*}
where we used that by definition, $\rr_k \le \alpha'\vpmax $ for all $k\in S_{\alpha'}^c$.
Combining with \eqref{lb:qtout} and rearranging gives
\begin{equation}	\label{lb:continuity}
|S_{c_0\alpha\alpha'/2}\setminus S_{\alpha'}| \ge |U'|\ge \frac12c_0 \alpha |S_{\alpha'}^c|.
\end{equation}
Since \eqref{lb:continuity} holds as long as $P(\alpha')$ fails,
we can repeatedly lower $\alpha'$ by a factor $c_0\alpha/2$ to obtain $\alpha'=\alpha'(\alpha,s,\scut,\delta,\eps)$ such that $P(\alpha')$ holds.
More explicitly, for each $k\ge 0$ put $\alpha_k= (c_0\alpha/2)^k\alpha$ and abbreviate $S_k:= S_{\alpha_k}$.
Then for all $k\ge 1$ such that $P(\alpha_k)$ fails we have
$
|S_{k+1}\setminus S_k| \ge \frac12c_0\alpha |S_k^c|,
$
so
\begin{align*}
|S_{k+1}\setminus U_0| &= |S_{k+1}\setminus S_k| + \cdots + |S_{1}\setminus U_0| 
\ge \frac12c_0\alpha( |S_k^c| + \cdots + |S_{0}^c|)
\ge (k+1)\frac12c_0\alpha |S_{k+1}^c|.
\end{align*}
Thus, we must have that $P(\alpha_{k})$ holds for some $k\le 2\kappa/c_0\alpha$.
(This gives $\alpha'$ of size $O(1/\alpha)^{-O(1/\alpha)}$.)
\end{proof}

Now we conclude the proof of Proposition \ref{prop:boundq}.
From Lemma \ref{lem:S14} we have $|S_{1/4}|\ge (s^2/4)n$. 
Applying Lemma \ref{lem:alphaprime} $O(1)$ times we obtain $\alpha''\gs 1$ such that
\begin{equation}	\label{adp}
|S_{\alpha''}|\ge (1-\delta/2)n.
\end{equation}
Now from \eqref{qproduct} we have
\begin{equation}	\label{bd:inSaa}
\rt_j \le \frac1{s^2\alpha'' \tpmax}
\end{equation}
for all $j\in S_{\alpha''}$.
On the other hand, for any $j\in S_{\alpha''}^c$,
\[
1/\rt_j \ge \tp_j \ge (V^\tran \rr)_j \ge \frac1n \sum_{i\in S_{\alpha''}} \sigma_{ij}^2\rr_i \ge \frac1n \snot^2 \alpha'' \tpmax|\mN_-(j)\cap S_{\alpha''}|.
\]
From \eqref{adp} and the robust irreducibility assumption (specifically the condition \eqref{mindegree}), 
\[
|\mN_-(j)\cap S_{\alpha''}| \ge \delta n - |S_{\alpha''}^c| \ge \delta n/2.
\]
Combining the previous two displays we obtain
\[
\rt_j \le \frac{2}{\delta\snot^2 \alpha'' \tpmax}
\]
for all $j\in S_{\alpha''}^c$.
Together with \eqref{bd:inSaa} we have
\[
\rt_j \ls \frac{1}{\alpha'' \tpmax}
\]
for all $j\in [n]$.
Applying \eqref{qtrace},
\[
\alpha'' \tpmax n/2 \le \alpha'' \tpmax |S_{\alpha''}| \le \sum_{j=1}^n \rr_j = \sum_{j=1}^n \rt_j \ls \frac{n}{\alpha'' \tpmax}
\]
and rearranging gives $\tpmax \ls 1/\alpha''\ls 1$. Finally, since
\[
K\le \frac1n\sum_{j=1}^n \rr_j \le \tpmax/s^2
\]
by \eqref{qbds}, we obtain a contradiction if $K$ is sufficiently large depending on $s,\snot, \delta$ and $\kappa$.
It follows that \eqref{assume:qK} fails for sufficiently large $K$, which concludes the proof of Proposition \ref{prop:boundq}.

\begin{rem} 
We note that in the above proof we only applied the expansion bound \eqref{lb:expand} to sets of size at least $\delta n/10$. 
\end{rem}

\section{Proof of Theorem \ref{thm:main}-(1): Tail estimates and asymptotics of the logarithmic potential}
\label{sec:logpotential} 

The main purpose of this section is to show that the logarithmic potential
$U_{\mu^Y_n}$ is close to 
$h_n(z) = - \int_\R \log|x| \, \check\nu_{n,z}(dx)$ for large $n$, and moreover that $h_n$ is the logarithmic potential of a probability measure $\mu_n$. 
Recall that in Theorem~\ref{th:L->nu} we have already established the almost sure convergence of the truncated potentials:
\[
\int_{\{ |x|\ge \varepsilon\}} \log |x| \, \check L_{n,z}(dx) - 
\int_{\{ |x|\ge \varepsilon\}} \log |x| \, \check\nu_{n,z}(dx)  
\ \xrightarrow[n\to\infty]{\text{a.s.}} \ 0 \ .
\]
Thus, we need to show that these measures uniformly integrate the singularity of $x\mapsto \log|x|$ at 0.
The proof has two main ingredients. The first is a result from \cite{Cook:ssv} by the first author (stated in Proposition~\ref{prop:nick} below) that provides control on the smallest singular value of $Y_n-z$.

The second is the control of the remaining small singular values of $Y_n-z$ 
via the quantity $\E \Imm g_{\check L_{n,z}}(\ii t)$ when $t$ is close to zero. We refer to step 2 in Section \ref{outline} for an outline of this argument.

Finally, to obtain the deterministic equivalents $\mu_n$ for the ESDs $\mu_n^Y$ we rely on a \emph{meta-model argument}, which has been used before in \cite{dumont-et-al-10,najim-yao-2016}.
The idea is that for fixed $n$ we can define a sequence $\{\Ym_n\}_{m\ge 1}$ of $nm\times nm$ random matrices as in Definition \ref{def:model}, where the standard deviation profiles $\Am_n$ are obtained by replacing each entry $\sigma_{ij}$ of $A_n$ by an $m\times m$ block with entries all equal to $\sigma_{ij}$.
We can then show that the logarithmic potentials of the associated ESDs converge to $h_n$ as $m\to \infty$, which will allow us to deduce that $h_n$ is itself the logarithmic potential of a probability measure.
This argument is described in more detail in Section \ref{sec:meta} below.

\subsection{Control on small singular values}

The following result, obtained by one of the authors in \cite{Cook:ssv}, gives an estimate on the lower tail of the smallest singular value $s_{n,z}$ of $Y_n-z$.

\begin{prop}[\cite{Cook:ssv}, Theorem~1.19 and Corollary~1.22]
\label{prop:nick} 
Assume \ref{ass:moments} and \ref{ass:sigmax} hold, and fix
$z\in \C \setminus\{ 0\}$. 
There exist constants $C(|z|,\Mo,\smax),\alpha(\eps),\beta(|z|,\eps,\Mo\smax)>0$ such that for all $n\ge 1$,
\begin{equation}	\label{bound:nick}
\PP \left( s_{n,z} \le  n^{-\beta}\right) \le 
Cn^{-\alpha}.
\end{equation}
\end{prop}

\begin{rem}
Similar bounds have been obtained under stronger assumptions on the standard deviation profile.
For instance, \eqref{bound:nick} follows from \cite[Lemma A.1]{2012-bordenave-chafai-circular} if we additionally assume \ref{ass:sigmin} (and in fact this result does not require \ref{ass:moments}).
Further assuming that $A_n$ is composed of a bounded number of blocks of equal size with constant entries, \cite[Corollary 5.2]{ARS} gives \eqref{bound:nick} with $\alpha>0$ as large as we please (and $\beta=\beta(\alpha)$).
An easy argument also gives \eqref{bound:nick} for arbitrary fixed $\alpha>0$ and $\beta(\alpha)$ under \ref{ass:sigmin} and replacing \ref{ass:moments} with a bounded density assumption -- see \cite[Section 4.4]{2012-bordenave-chafai-circular}.
For the case that the entries $X_{ij}$ are real Gaussian variables and $A_n(\scut)$ is $(\delta,\kappa)$-broadly connected for some fixed $\scut,\delta,\kappa\in (0,1)$ (see Definition \ref{def:broad}), \eqref{bound:nick} holds with arbitrary $\alpha>0$ and $\beta=\alpha+1$ by \cite[Theorem 2.3]{rudelson-zeitouni-2016}.
\end{rem}

We now consider the other small singular values of $Y_n - z$. The key is the uniform control on solutions to the Regularized Master Equations \eqref{def:MEt} provided by Assumption \ref{ass:admissible} 
combined with Theorem \ref{th:convergence-QR}.

\begin{coro}[Wegner estimates] \label{cor:wegner}	
Let \ref{ass:moments}, \ref{ass:sigmax} and \ref{ass:admissible} hold. 
Then, for all $z\in\C \setminus \{0\}$  there exist constants $C,\gamma_0>0$ such that for all $x>0$,
\begin{equation}
\check \nu_{n,z}( (-x,x) )\le  C x \label{wegner-nu}
\end{equation}
and 
\begin{equation}
\E \check L_{n,z}((-x,x))   \le C   (x \vee n^{-\gamma_0}) .
\label{wegner-L}
\end{equation}
\end{coro}

\begin{proof}
We rely on the following elementary estimate for the Stieltjes transform of a probability measure $\mu$ (see for instance {\cite[Lemma 15]{MR2831116}}): 
\begin{equation}
\label{bnd:ts} 
\Imm g_\mu(\ii t) \ = \ t\int \frac{\mu(d\lambda)}{\lambda^2 + t^2}  \ \geq \
t\int_{-t}^t \frac{\mu(d\lambda)}{\lambda^2 + t^2} \ \geq \
\frac{1}{2t} \mu((-t,t))\, . 
\end{equation} 
Recall that 
$
\im\,  g_{\check \nu_{n,z}}(\ii t) = n^{-1} \sum_{i\in [n]} \rr_i(|z|,t) .
$
The first Wegner estimate \eqref{wegner-nu} is a straightforward consequence of Assumption \ref{ass:admissible} and the estimate \eqref{bnd:ts}.

We now establish the second Wegner estimate \eqref{wegner-L} and first prove that there exists $\gamma_0>0$ such that 
\begin{equation}\label{eq:g-L-estimate}
\sup_{t\geq n^{-\gamma_0}} \E \Imm g_{\check L_{n,z}}(\ii t) \leq C
\end{equation}
for all $n\ge 1$. For $t\geq 1$, $\E \Imm g_{\check L_{n,z}}(\ii t) \leq 1$ by the mere definition of a Stieltjes transform. Assume $t < 1$ and recall that $\E \Imm g_{\check L_{n,z}}(\ii t) = n^{-1} \tr \Imm \E
G(z, \ii t)$. By Theorem \ref{th:convergence-QR}, there exist constants $c_0, C > 0$ such that 
\[
\Bigl | \frac 1n  \tr \Imm \E G(z, \ii t) - 
\frac 1n \sum_{i=1}^n \rr_i(|z|, t) \Bigr | \ \leq\ \frac{C}{\sqrt{n} t^{c_0}} \, .
\]
By \ref{ass:admissible}, we therefore get that 
\[
\E \Imm g_{\check L_{n,z}}(\ii t) \leq C( t^{-c_0} n^{-1/2} + 1).
\] 
By letting now $t\ge n^{-\gamma_0}$ with $\gamma_0 = 1 / (2c_0)$, we obtain \eqref{eq:g-L-estimate}.
Combining this result with \eqref{bnd:ts}, we get
\[
\E \check L_{n,z}((-x,x)) \, \le \, \E \check L_{n,z}(( -(x\vee n^{-\gamma_0}) ,x \vee n^{-\gamma_0} ))\,  \le \,  
2  C   (x \vee n^{-\gamma_0})
\]
which is the desired result.
\end{proof}

\subsection{Comparison of logarithmic potentials via a meta-model}	\label{sec:meta}

We now turn to the task of finding the measures $\mu_n$ from Theorem \ref{thm:main} which serve as a sequence of deterministic equivalents for the ESDs $\mu_n^Y$.
A first idea is to try to show that for every $\psi \in \C_c^\infty(\C)$, 
\[
\int \psi(z) \, \mu^Y_n(dz) = - \frac{1}{2\pi} \int \Delta\psi(z) \, 
U_{\mu^Y_n}(z) \, \ell(dz) 
= \frac{1}{2\pi} \int \Delta\psi(z) \Bigl( \int_\R \log |x| \, 
\check L_{n,z}(dx) \Bigr) \ell(dz) 
\]
is ``close'' to $- (2\pi)^{-1} \int \Delta\psi(z) h_n(z) dz$.  However, there
is a difficulty in directly applying this approach, related to the
fact that $\check L_{n,z}$ does not converge in general with no further assumption
on the variance profile matrices $V_n$. 

To circumvent this difficulty, we rely on a \emph{meta-model argument}, which has been used in \cite{dumont-et-al-10,najim-yao-2016}, and which we now describe.
Let $n$ be fixed, consider the standard deviation profile $A_n=(\sigma_{ij})$ and the normalized variance profile $V_n=\left(\frac 1n \sigma_{ij}^2\right)$. Recall the associated Schwinger--Dyson equations as provided in Proposition \ref{prop:deteq} and the solution ${\bs{\vec p}}=\left( \begin{array}{c} \bs{p}\\ \bs{\tilde p}\end{array}\right)$, of dimension $2n\times 1$. Define the meta-model in the following way: for an integer $m\ge 1$, consider the $nm\times nm$ standard deviation profile matrix defined as
\[
\Am_n = \left( 
\begin{array}{ccc}
A_n &\cdots & A_n\\
\vdots& & \vdots\\
A_n &\cdots & A_n\\
\end{array}
\right) = (\bs 1_m \bs 1_m^\tran)\otimes A_n \ ,
\]
associated to the normalized variance profile $\Vm_n =(\bs 1_m \bs 1_m^\tran)\otimes m^{-1} V_n $, and the random matrix 
\begin{equation}\label{eq:meta-model}
\Ym_n =  \left(
\frac{[\Am_n]_{ij}}{\sqrt{mn}} X^{(nm)}_{ij} \right) _{i,j\in [nm]} .
\end{equation}
Denote by $\Lm_{n,z}$ the symmetrized empirical distribution of the singular values of $\Ym_{n}-zI_{mn}$. Due to the specific form of $\Vm_n$, it is straightforward to check that the solutions of the Schwinger--Dyson equations associated to this model are provided by 
\[
{\bs{\vec p}}_m=\left( \begin{array}{c} \bs{p}_m\\ \bs{\tilde p}_m\end{array}\right)\qquad \textrm{where}\qquad {\bs{p}^\tran_m}=\left(  \bs{p}^\tran,\cdots, \bs{ p}^\tran\right)\quad \textrm{and}\quad \bs{\tilde p}^\tran_m=\left( \bs{\tilde p}^\tran,\cdots, \bs{\tilde p}^\tran\right)\ ,
\]
where ${\bs{p}}_m$ and ${\bs{\tilde p}}_m$ are $nm\times 1$ vectors.
As an important consequence, we have:
\[
g_{\check{\nu}^{(m)}_{n,z}}(\eta) \ =\ \frac 1{mn} \sum_{i=1}^{mn} [\,{\bs{\vec p}}_m\,]_i\ =\ \frac 1{n} \sum_{i=1}^{n} [\,{\bs{\vec p}}\,]_i\ =\ g_{\check \nu_{n,z}}(\eta) .
\]
Hence the Stieltjes transform $g_{\check{\nu}^{(m)}_{n,z}}$ of $\check{\nu}^{(m)}_{n,z}$ does not depend on $m$ and is equal to $\check \nu_{n,z}$.
Finally, if \ref{ass:moments} and \ref{ass:sigmax} are satisfied for $Y_n$, they are also satisfied for $\Ym_n$.
In particular, $\Lm_{n,z}\sim \check{\nu}^{(m)}_{n,z}$ admits a genuine limit as $m\to\infty$:
\begin{equation}\label{eq:conv-meta-model}
\Lm_{n,z} \ \xrightarrow[m\to\infty]{\bf w}\  \check \nu_{n,z} \quad \text{a.s.} 
\end{equation}
since $\check{\nu}^{(m)}_{n,z}=\check \nu_{n,z}$.

We now state our proposition giving the existence of the measures $\mu_n$, the proof of which will occupy the main part of the remainder of this section.

\begin{prop}
\label{Umu}
Let \ref{ass:moments}, \ref{ass:sigmax} and \ref{ass:admissible} hold. Then the following hold:
\begin{enumerate}
\item  For all $n\ge 1$ and $z \in C \setminus \{0 \}$, the function 
\[
h_n(z) = - \int_\R \log |x| \, \check \nu_{n,z}(dx) 
\]
is well defined and for every compact set $\mathcal K \subset \C$,
\[
\sup_n \int_{\mathcal K} | h_n(z) |^2 \, \ell(dz)  < \infty\, .
\]
Moreover, $h_n(z)$ coincides with the logarithmic potential $U_{\mu_n}(z)$
of a probability measure $\mu_n$ on $\C$. 
\item \label{point:tightness} For $\mu_n$ as defined in part \textup{(1)}, there exists a constant
$C > 0$, independent of $n$, such that for all $M > 0$, 
\[
\mu_n(\{ z\in \C; \ |z|>M\} ) \leq \frac{C}{M^2} .
\]
\end{enumerate}
\end{prop}

We will rely on the following lemmas, whose proofs are respectively deferred to Appendices \ref{proof:cvP} and \ref{proof:h:pot}.

\begin{lemma}
\label{cvP}
Let $(\Omega,{\mathcal F},\mathbb{P})$ be a given probability space, $\zeta$ a finite positive measure on $\C$ and $f_n:\Omega\times \C \to \R$ measurable functions satisfying 
\begin{equation}\label{eq:unif-integrability}
\sup_n\,  \int_\C | f_n(\omega,z) |^{1+\alpha} \, \mathbb{P}\otimes\zeta(d\omega\times dz) \ \leq\  C 
\end{equation}
for some constants $\alpha, C > 0$. 

Let $g : \C \to \R$ 
be a measurable function such that for $\zeta$-almost all 
$z\in\C$,  
\[
f_n(\omega,z) \xrightarrow[n\to\infty]{{\mathcal P}} g(z) \ . 
\]
Then $\int_\C |g(z)|^{1+\alpha} \zeta(dz) \leq C$, and 
\begin{equation}\label{eq:conv-probab}
\int_\C f_n(\omega,z) \, \zeta(dz) \xrightarrow[n\to\infty]{{\mathcal P}} 
\int g(z) \, \zeta(dz) .
\end{equation}
\end{lemma}

\begin{lemma}
\label{h:pot} 
Let $(\zeta_n)$ be a sequence of random probability measures on 
$\C$. Assume that a.s. $(\zeta_n)$ is tight and that there exists a locally integrable function $h:\C\to\R$ such that
for all $\psi \in C_c^\infty(\C)$, 
\begin{equation}
\label{cvg:h}
\int \psi(z) \, \zeta_n(dz) 
\xrightarrow[n\to\infty]{{\mathcal P}} 
- \frac{1}{2\pi} \int \Delta\psi(z) h(z) \, \ell(dz)\ .
\end{equation} 
Then there exists a non-random probability measure $\zeta$ on $\C$ with logarithmic potential $h$, i.e.
\[
h(z) = -\int_{\C} \log|z-u|\, \zeta(\, du)
\]
for almost all $z\in \C$ such that 
$
\zeta_n \xrightarrow[n\to\infty]{\bf w} \zeta
$
in probability. 
\end{lemma}

\begin{proof}[Proof of Proposition \ref{Umu}]
We prove the first point of the proposition.

With $\Ym_n$ as in \eqref{eq:meta-model}, 
denote by $\mu^{\Y}_{n,m}$ the spectral measure of $\Ym_n$, and recall that 
\[
U_{\mu^{\Y}_{n,m}}(z) = - \int_\R \log |x| \, \Lm_{n,z}(dx) . 
\]
The proof consists of the following three steps:
\begin{enumerate}
\item\label{potmeta}
To show that for every $z \in \C \setminus \{0\}$, $x\mapsto \log|x|$ is $\check \nu_{n,z}$-integrable and 
\begin{equation}
\label{cvgUmeta} 
U_{\mu^{\Y}_{n,m}}(z) 
\xrightarrow[m\to\infty]{{\mathcal P}} h_n(z) \, . 
\end{equation} 
\item\label{hmes} 
To show that the function $h_n(z)$ is measurable. \\
\item\label{unif-int}
To show that for any compact set $\mathcal K \subset \C$, 
\begin{equation} 
\label{uimeta}
\sup_m \E \int_{\mathcal K} | U_{\mu^{\Y}_{n,m}}(z) |^2 \, \ell(dz) \leq C
\end{equation} 
for some constant $C > 0$ independent of $n$. 
\end{enumerate}
The three previous steps being proved, the assumptions of Lemma~\ref{cvP} are fulfilled and the lemma yields
$$
 \int_{\mathcal K} | h_n(z) |^2 \, \ell(dz) \ \leq\  C 
$$ 
where $C$ does not depend upon $n$. Moreover,
\[
\int \psi(z) \, \mu_{n,m}^{\Y}(dz) \ =\  
- \frac 1{2\pi} \int \Delta\psi(z) \, 
U_{\mu^{\Y}_{n,m}}(z) \, \ell(dz) \
 \xrightarrow[m\to\infty]{{\mathcal P}}\  
- \frac 1{2\pi} \int \Delta\psi(z) \, h_n(z) \, \ell(dz) 
\]
for every $\psi \in \C_c^\infty(\C)$.

It remains to apply Lemma~\ref{h:pot} to conclude that $h_n$ is the logarithmic potential of a probability distribution $\mu_n$ on $\C$ 
and point (1) of Proposition \ref{Umu} will be proved.

Let us address Step~\ref{potmeta}, to prove the convergence in \eqref{cvgUmeta}, we separately consider the integrals defining the logarithmic potentials in the regions $|x|$ greater than and less than $\varepsilon$.

Taking into account the convergence \eqref{eq:conv-meta-model} and applying Theorem~\ref{th:L->nu} to the sequence $(\Lm_{n,z})_m$, we deduce that $x\mapsto \log|x|$ is 
$\check \nu_{n,z}$-integrable near infinity, and that 
\begin{equation}\label{eq:conv-P-1}
 \int_{\{|x|\ge \varepsilon\}}\log |x| \,\Lm_{n,z}(dx) 
\xrightarrow[m\to\infty]{\text{a.s.}} 
 \int_{\{|x|\ge \varepsilon\}}\log |x| \, \check \nu_{n,z}(dx) \  . 
\end{equation}
We now handle the remaining regions.
\begin{eqnarray} 
\lefteqn{\left| \int_{\{|x|<\varepsilon\}}  \log |x| \, \check \nu_{n,z}(dx) \right| 
\quad\le  \quad \int_{\{|x|<\varepsilon\}} |\log |x|\, | \, \check \nu_{n,z}(dx) }\nonumber \\
&= &\int_\R \check \nu_{n,z}\, \{|x|<\varepsilon  \, , \, \log |x| \leq - y\} \, dy\ \stackrel{(a)}\leq \ C \int_0^\infty ( \exp(-y) \wedge \varepsilon ) \, dy 
\ =\ C \varepsilon( 1 - \log\varepsilon) \, , 
\label{eq:conv-P-2} 
\end{eqnarray}
where $(a)$ follows from Wegner's estimate \eqref{wegner-nu} in Corollary \ref{cor:wegner}.

Let $(\bs s_{i,z})_{i\in [mn]}$ be the singular 
values of $\Ym_n - z$ ordered as $\bs s_{1,z}\ge \cdots \ge \bs s_{mn,z}$. 
For $z\in \C \setminus \{0\}$ and $\beta>0$ the exponent as 
in Proposition \ref{prop:nick}, we introduce the event
$$
\mathcal G_m := \left\{  \bs s_{i,z} \geq (mn)^{-\beta}\, , \ i\in[mn] \right\} \ .
$$  
For all $\tau>0$, 
\begin{align*}
\lefteqn{\PP\left\{  \Bigl| \int_{\{|x|<\varepsilon\}} \log |x| \,\Lm_{n,z}(dx) \Bigr| 
  > \tau \right\}}\\
  &\le  \PP\left\{  \Bigl| {\mathbbm 1}_{\mathcal G_m}  \int_{\{|x|<\varepsilon\}} \log |x| \,\Lm_{n,z}(dx)
\Bigr| > \frac{\tau}2 \right\}+\PP\left\{  \Bigl| {\mathbbm 1}_{\mathcal G^c_m}  
\int_{\{|x|<\varepsilon\}} \log |x| \,\Lm_{n,z}(dx)
\Bigr| > \frac{\tau}2 \right\}\ .
\end{align*}
Noticing that 
\[
\left\{  \Bigl| {\mathbbm 1}_{\mathcal G^c_m}  
\int_{\{|x|<\varepsilon\}} \log |x| \,\Lm_{n,z}(dx)
\Bigr| > \frac{\tau}2 \right\}
\ \subset\  \left\{ \bs s_{mn,z} \le (mn)^{-\beta}\right\}
\]
for $m$ large enough, Proposition \ref{prop:nick} yields that 
\[
\PP\left\{  \Bigl| {\mathbbm 1}_{\mathcal G^c_m}  
\int_{\{|x|<\varepsilon\}} \log |x| \,\Lm_{n,z}(dx)
\Bigr| > \frac{\tau}2 \right\}
\ \le\ \frac C{(mn)^{\alpha}}\ .
\]
Recall the constant $\gamma_0$ in Corollary \ref{cor:wegner}. Choose $m$ large enough and $\gamma\le \gamma_0$ small enough so that $(nm)^{-\beta}\le (mn)^{-\gamma} \le \varepsilon\le 1$. We now estimate 
\begin{align*} 
\lefteqn{\E {\mathbbm 1}_{\mathcal G_m}  
\int_{\{|x|<\varepsilon\}}
|\log |x|\,|  \, \Lm_{n,z}(dx)}\\ 
&= \int_{\{(mn)^{-\beta}\le |x|\le (mn)^{-\gamma}\}} |\log |x|\,| \, \E \Lm_{n,z}(dx) 
 + \int_{\{ (mn)^{-\gamma}<|x|<\varepsilon\}} |\log |x|\,| \, \E \Lm_{n,z}(dx) \\
 & :=  I_1 + I_2 \ . 
\end{align*} 
By the Wegner estimate \eqref{wegner-L}, we obtain
$$
I_1 \ \le\  \beta \log(mn) \E \Lm_{n,z}([-(mn)^{-\gamma}, (mn)^{-\gamma}])
 \ \le\  C\beta (mn)^{-\gamma}\log(mn) .
$$
On the other hand, another application of the same Wegner estimate yields
\begin{eqnarray*}
I_2 &= &
\int_0^\infty \, \E \Lm_{n,z}\left(\left\{  x \, : \, | \log |x|\, | 
{\mathbbm 1}_{[ (mn)^{-\gamma}, \varepsilon ]}(|x|) \geq y \right\} \right) \, dy \\
&= &
\int_0^\infty \E\Lm_{n,z}\left([ -e^{-y} \wedge \varepsilon,-(mn)^{-\gamma} ]\cup [(mn)^{-\gamma}, e^{-y} \wedge \varepsilon ]
\right) \, dy \\
&\leq& C \int_0^\infty ( \exp(-y) \wedge \varepsilon ) \, dy \qquad =\qquad C \varepsilon ( 1 - \log\varepsilon ) \, . 
\end{eqnarray*}
Therefore, by Markov's inequality, we finally obtain
\[
\PP\biggl\{ \biggl| \int_{\{|x|<\varepsilon\}} \log |x| \, \Lm_{n,z}(dx) \bigg| 
  > \tau \biggr\}
   \le   \frac1\tau2C \big[
\beta (mn)^{-\gamma} \log (mn) + \varepsilon (1-\log\varepsilon)\big]
+ C (mn)^{-\alpha}  . 
\]
Thus, for all $\tau,\tau' > 0$, we can choose 
$\varepsilon > 0$ small enough so that 
\[
\PP\biggl\{ \bigg| \int_{\{|x|<\varepsilon\}} \log |x| \, \Lm_{n,z}(dx) \bigg| 
  > \tau \biggr\}
 < \tau' 
\]
for $m$ large enough. Gathering this result with \eqref{eq:conv-P-1} and \eqref{eq:conv-P-2} yields \eqref{cvgUmeta}, and Step~\ref{potmeta} is proved. 

We now address Step~\ref{hmes} and study the measurability of $h_n(z)$. Recall the regularized logarithmic potentials defined in \eqref{cU} and consider also the following function:
$$
\cU_{n,m}^{\Y}(z,t) \, :=\, 
- \frac{1}{2nm} \log\det ( (\Ym_n-z)^* (\Ym_n-z) + t^2 ) \ = \ - \frac 12 \int_\R \log(x^2 + t^2) \, \Lm_{n,z}(dx) \ .
$$
Given $z$ and $z' \in \C$, Hoffman-Wielandt's theorem applied 
to $\Ym_n-z$ and $\Ym_n - z'$ yields
\[
\max_{i\in [mn]} | \bs s_{i,z} - \bs s_{i,z'} | \leq | z - z'|\ .
\]
Thus
\begin{multline*}
\left| \cU_{n,m}^{\Y}(z,t) - \cU_{n,m}^{\Y}(z',t) \right| \\
= \ 
\frac{1}{2nm} \left| \sum_{i\in [mn]}
 \log\left(1 + \frac{\bs s_{i,z}^2}{t^2}\right) -  \log\left(1 + \frac{\bs s_{i,z'}^2}{t^2}\right)\right|\ \le\  \frac1{2t^2}\max_{i\in [mn]} | \bs s_{i,z} - \bs s_{i,z'} |
 \ \le\  \frac{|z-z'|}{2t^2}
\end{multline*}
and it follows that for any fixed $t>0$ the family $\{z\mapsto \cU_{n,m}^{\Y}(z,t)\}_{m\ge1}$ is uniformly equicontinuous. 
Since from Theorem~\ref{th:L->nu} we have 
$
\cU_{n,m}^{\Y}(z,t) \xrightarrow[m\to\infty]{} \cU_n(z,t)
$
almost surely, it follows that $z\mapsto \cU_n(z,t)$ is continuous for any fixed $t > 0$. 
Finally, since $x\mapsto \log|x|$ is $\check \nu_{n,z}$-integrable near zero for any $z\neq 0$ by 
\eqref{eq:conv-P-2},
\[
\cU_n(z,t) \xrightarrow[t\to 0]{} h_n(z) .
\]
The measurability of $h_n$ follows and Step \ref{hmes} is proved. 

We now address Step \ref{unif-int} and prove \eqref{uimeta}. Observe that on any compact set 
$\mathcal K \in \C$, there exists a constant $C_{\mathcal K}$ such that 
\[
\int_{\mathcal K} (\log |\lambda - z |)^2 \, \ell(dz) \ \leq\  
C_{\mathcal K} ( 1 + |\lambda|^2 ) 
\]
for all $\lambda \in \C$. Denote by $(\bs \lambda_i;\, i\in [mn])$ the eigenvalues of $\Ym_n$. We have
$$
\E \int_{\mathcal K} | U_{\mu^{\Y}_{n,m}}(z) |^2 \, \ell(dz) 
\ \leq\  \E \left( \frac{1}{mn}\sum_{i\in [mn]}
\int_{\mathcal K} (\log | \bs \lambda_{i}^{(m)} - z|)^2 \, \ell(dz) \right)\ \leq\ 
C_{\mathcal K} \Bigl( 1 + \E \int |\lambda|^2 \mu^{\Y}_{m,n}(d\lambda) \Bigr)\,  .
$$
By the Weyl comparison inequality for eigenvalues and singular values 
(cf.\ e.g.\ \cite[Theorem~3.3.13]{hor-joh-topics}), 
\[
\int |\lambda|^2 \mu^{\Y}_{n,m}(d\lambda) = \frac 1{nm}\sum_{i=1}^{nm} |\bs\lambda_i|^2 
\leq \frac 1{nm}\sum_{i=1}^{nm} \bs s_{i,0}^2 = \frac 1{nm} \tr \big(\Ym_n(\Ym_n)^*\big) \\ 
\leq \frac{\smax^2}{(nm)^2} \sum_{i,j=1}^{nm} |X_{ij}^{(nm)}|^2 . 
\]
Taking the expectation of the previous inequality finally yields 
\[
\E \int_{\mathcal K} | U_{\mu^{\Y}_{n,m}}(z) |^2 \, \ell(dz)  \le C .
\]
Step \ref{unif-int} is proved.

We now prove point (2) of Proposition \ref{Umu}. Given $M > 0$, we get 
from Lemma~\ref{cvP} that 
\[
\limsup_m \mu^{\Y}_{n,m}(\{ z\in \C;\ |z| >M\}) \le 
\limsup_m \frac{1}{M^2} \int_{\C} |\lambda|^2 \, \mu^{\Y}_{n,m}(d\lambda) 
\leq \frac{C}{M^2} \quad \text{a.s.} 
\]
where $C > 0$ is independent of $n$. 
Let $\psi$ be a nonnegative $C_c^\infty(\C)$ function equal to one for $|z|<M$ 
and to zero if $|z|>M+1$. As a byproduct of Lemma \ref{h:pot}, 
\[
\mu^{\Y}_{n,m}\xrightarrow[m\to\infty]{w} \mu_n
\]
almost surely. 
Consequently, on a set of probability one, 
\[
\mu_n(\{ z\in \C; \ |z|\le M+1\} ) \ \geq\  
\int\psi(z) \, \mu_n(dz) \ =\  \lim_m \int\psi(z) \,\mu^{\Y}_{n,m}(dz) 
\ \geq\  1 - \frac{C}{M^2}\, .
\]
Proposition~\ref{Umu} is proved. 
\end{proof}

\subsection{Conclusion of the proof of Theorem~\ref{thm:main}-\textup{(i)}}

We can now complete the proof of Theorem~\ref{thm:main}-\textup{(i)} and prove that $\mu_n^Y \sim \mu_n$ in probability, with $\mu_n$ defined in Proposition \ref{Umu}.

By Proposition~\ref{Umu}, the sequence $(\mu_n)$ is tight. It remains to prove that
for all $\varphi \in C_c(\C)$, 
$\int \varphi d\mu_n^Y - \int\varphi d\mu_n \to 0$ in probability. 
By the density of $C_c^\infty(\C)$ in $C_c(\C)$, it is enough to show that 
\[
\int \psi(z) \mu_n^Y(dz) - \int \psi(z) \mu_n(dz) = - \frac{1}{2\pi} 
\int \Delta \psi(z) ( U_{\mu_n^Y}(z) - U_{\mu_n}(z) ) \, \ell(dz) 
\xrightarrow[n\to\infty]{\mathcal P} 0 
\]
for all $\psi\in C_c^\infty(\C)$. By mimicking the proof of Proposition~\ref{Umu},  
where $\Ym_n$ and $m$ are replaced with $Y_n$ and $n$ respectively, 
we straightforwardly obtain that $U_{\mu_n^Y}(z) - U_{\mu_n}(z) \to 0$ in
probability for every $z\in\C\setminus \{ 0 \}$. This proof also shows that 
$\sup_n \E \int_{\mathcal K} |U_{\mu_n^Y}(z)|^2 \, \ell(dz) < \infty$ for all 
compact sets $\mathcal K \subset \C$. We also know by Proposition~\ref{Umu} that
$\sup_n \int_{\mathcal K} |U_{\mu_n}(z)|^2 \, \ell(dz) < \infty$. 
The result now follows from Lemma~\ref{cvP}.

\section{Conclusion of proofs of Theorems~\ref{thm:main} and \ref{thm:bistochastic}}
\label{sec:end}

\subsection{Proof of Theorem~\ref{thm:main}: Identification of $\mu_n$} \label{sec:density} 

We established in the previous section that $\mu^Y_n \sim \mu_n$ in 
probability. To conclude the proof of Theorem~\ref{thm:main}, it remains to show that
$\mu_n$ is rotationally invariant, and that its radial cumulative distribution function
$$
\mu_n\{ z\in \C\, ; \ |z|\le r\}  
$$
coincides with the function $F_n$ specified in the statement of the 
theorem. These facts, along with the properties of $F_n$, are established in 
Lemma~\ref{propF} below. 

For the remainder of this section, we set 
\begin{equation}\label{def:b}
b_n(z,t) := - \frac z{2n} \tr \Psi(\rvec(|z|,t)\,, t )\qquad  \textrm{and} \qquad
b_n(z) = - \frac z{2n} \tr \Psi(\qvec(|z|) )\ ,
\end{equation}
where $\Psi(\cdot,t)$ and $\Psi(\cdot)$ are defined in \eqref{def-psi},  $\rvec(\cdot,t)$ is defined in Proposition \ref{prop:MEt}, and $\qvec(\cdot)$ is defined in \ref{thm:master}. 

\begin{lemma} \label{lemma:identification}
Under the same assumptions as in Theorem \ref{thm:main}, the function $z\mapsto b_n(z)$ is locally integrable on $\C$, and 
$$
\partial_{\bar{z}} U_{\mu_n}(z) = b_n(z)$$ 
in ${\mathcal D}'(\C)$. 
\end{lemma}
\begin{proof} 

Recall the definition of $\cU_n(z,t)$ in 
\eqref{cU}. Recall also that by Proposition \ref{Umu}-(i), the probability measure $\mu_n$ is such that:
\[
U_{\mu_n}(z) = - \int_\R \log|x|\check \nu_{n,z}(\, dx)\ .
\]
We first prove that 
\begin{equation}\label{conv:U->Umu} 
\cU_n(z,t) \ \xrightarrow[t\downarrow 0]{{\mathcal D}'(\C)} \ U_{\mu_n}(z)\ .
\end{equation}

Recall from the proof of Proposition~\ref{Umu} that $z\mapsto\cU_n(z,t)$ is continuous for any fixed $t>0$.
It is moreover clear from the expressions of
$\cU_n(z,t)$ and $U_{\mu_n}(z)$ that $\cU_n(z,t) \uparrow U_{\mu_n}(z)$ as
$t\downarrow 0$. Recall that
$U_{\mu_n}(z)$, being a logarithmic potential, is locally integrable (as can be
seen by Fubini's theorem). Thus, given a fixed $t_0 > 0$, 
\[ 
0\leq \cU_n(t,z) - \cU_n(t_0,z) 
\leq U_{\mu_n}(z) - \cU_n(t_0,z) , 
\] 
for $0 < t \leq t_0$, 
and \eqref{conv:U->Umu} immediately follows from the monotone convergence theorem.
By a property of the convergence in ${\mathcal D}'(\C)$, this implies the convergence of the distributional derivative
\begin{equation}\label{conv:dU->dUmu}
\partial_{\bar{z}}\,\cU_n(z,t)\ \xrightarrow[t\downarrow 0]{{\mathcal D}'(\C)} \ \partial_{\bar{z}}\, U_{\mu_n}(z) .
\end{equation}

We now prove that for all $t>0$, 
\begin{equation}\label{eq:identification-dU}
\partial_{\bar{z}} \cU_n(z,t) \ =\ b_n(z,t) 
\end{equation}
in ${\mathcal D}'(\C)$. We shall rely on a meta-model argument. Recall the meta-model $\Ym_n$ introduced in \eqref{eq:meta-model}, its limiting property \eqref{eq:conv-meta-model}, and the definition of $\cU_{n,m}^{\Y}(z,t)$ in \eqref{cU}. 

Fix $t > 0$. By Theorem \ref{th:L->nu}, $\cU_{n,m}^{\Y}(z,t) \to \cU_n(z,t)$ almost surely as $m\to \infty$ for 
all $z\in\C$. Furthermore, recalling the notation $(\bs s_{i,z})_{i\in [mn]}$ for the singular 
values of $\Ym_n - z$, we have
\begin{align*}
\left| \cU_{n,m}^{\Y}(z,t)\right|^2 &= \left| \log(t) + \frac 1{2mn} \sum \log\left( 1+{\bs s^2_{i,z}}/{t^2}\right)\right|^2 \\
\ & \stackrel{(a)}\le \  2\left| \log(t)\right|^2 + \frac 1{t^2 mn} \sum_{i\in [mn]} \bs s^2_{i,z}
\end{align*}
where $(a)$ follows from the elementary inequality $2^{-1}\log^2(1+x)\le x$, valid for $x\ge 0$. In particular, this implies that 
\[
\E \int_{\mathcal K} | \cU_{n,m}^{\Y}(z,t) |^2 \, \ell(dz) 
\ \leq\ (\log t)^2 + 
\frac{1}{t^2} \E \int_{\mathcal K} 
\frac{\tr (\Ym_n-z)^* (\Ym_n-z)}{mn} \, \ell(dz) 
\ \leq\  C 
\]
on every compact set $\mathcal K \subset \C$. By Lemma~\ref{cvP}, we get that 
$\cU_n(\cdot,t)$ is locally integrable on $\C$, and that 
\begin{equation} 
\label{dz*} 
\int \partial_{\bar{z}} \psi(z) \, \cU_{n,m}^{\Y}(z,t) \, \ell(dz) 
\ \xrightarrow[m\to\infty]{\mathcal P} \ 
\int \partial_{\bar{z}} \psi(z) \, \cU_{n}(z,t) \, \ell(dz) 
\end{equation} 
for all $\psi \in C_c^\infty(\C)$. 
An integration by parts along with Jacobi's formula shows that for all 
$\omega\in\Omega$, the distributional derivative 
$\partial_{\bar{z}}\, \cU^{\Y}_{n,m}(t,z)$ coincides with the pointwise 
derivative, which is given by 
\[
\partial_{\bar{z}} \cU^{\Y}_{n,m}(z,t)\ =\ \frac{1}{2nm} \tr (\Ym_n-z) 
         ( (\Ym_n-z)^* (\Ym_n-z) + t^2 )^{-1}   \, . 
\]
On the other hand, we know from Theorem~\ref{th:convergence-QR} that 
$\partial_{\bar{z}}\, \cU^{\Y}_{n,m}(z,t) \to b_n(z,t)$ almost surely as $m\to \infty$, for all
$z\in\C$. Moreover, from a singular value decomposition of $\Ym_n-z$ we easily
see that $| \partial_{z^*} \cU^{\Y}_{n,m}(z,t) | \leq (4t)^{-1}$. 
Consequently, we get by Lemma~\ref{cvP} again that 
\[
\int \partial_{\bar{z}} \psi(z) \ \cU_{n,m}^{\bs Y}(t,z) \ \ell(dz) = 
- \int \psi(z) \, \partial_{\bar{z}}\cU_{n,m}^{\bs Y}(t,z) \ \ell(dz) 
\xrightarrow[m\to\infty]{\mathcal P} 
- \int \psi(z) \, b_n(z,t) \, \ell(dz) . 
\] 
Comparing with \eqref{dz*}, we obtain that 
$\partial_{\bar{z}} \cU_n(z,t) = b_n(z,t)$ in ${\mathcal D}'(\C)$. 

We now consider the limit in $t\downarrow 0$ in \eqref{eq:identification-dU}. Since $|b_n(z,t)| \le |2z|^{-1}$, the dominated convergence theorem yields
\[
b_n(z,t) \xrightarrow[t\downarrow 0]{ {\mathcal D}'(\C)} b_n(z) .
\]
Combining this convergence together with \eqref{conv:dU->dUmu} and \eqref{eq:identification-dU}, we obtain the desired result. 
\end{proof} 

In order to characterize the probability measure $\mu_n$, we use the equation 
$\mu_n = - (2\pi)^{-1} \Delta U_{\mu_n}$ and rely on the smoothness 
properties of $\Delta U_{\mu_n}$ that can be deduced from Lemma~\ref{q:dif}. 
We recall that $\qvec(s)$ is defined in the statement of Theorem \ref{thm:master}.

\begin{lemma}
\label{propF}
The probability measure $\mu_n$ is rotationally invariant. On $(0,\infty)$, 
the distribution function $F_n(s) := \mu_n(\{ z\, :\, |z|\le s\})$ satisfies
\[
F_n(s) = 1 - \frac{1}{n} \langle \q(s),  V \qtilde(s) \rangle . 
\] 
The support of $\mu_n$ is contained in $\{z\,:\, |z|\le \sqrt{\rhoV}\}$. 
Finally, $F_n$ is absolutely continuous on $(0,\infty)$, and has a continuous
density on $(0,\sqrt{\rhoV})$. 
\end{lemma} 

Before entering the proof, we note that the rotational invariance of $\mu_n$
can be ``guessed'' from the form of the Schwinger--Dyson equations of Proposition~\ref{prop:deteq}. Indeed, from this one sees that the Stieltjes transform
$g_{\check\nu_{n,z}}(\eta) = n^{-1} \tr P(|z|,\eta)$ of $\check\nu_{n,z}$ 
depends on $z$ only through its absolute value, and this is therefore also the case 
for $U_{\mu_n}(z)$. It is easy to check that this yields the rotational 
invariance of $\mu_n$. 

\begin{proof}
First, we show that $\mu_n(C) = 0$, where $C$ is the circle with center
zero and radius $\sqrt{\rhoV}$. 

Consider a smooth function $\phi:\R\to [0,1]$ with support in $[-1,1]$ and value $\phi(0)=1$, and the function 
\[
g_{\varepsilon}(z)\ =\ \phi\left( \varepsilon^{-1}\left( |z|-\sqrt{\rhoV}\right)\right)\ ,\quad z\in \C
\]
with support in the annulus $\{ z\,:\, \sqrt{\rhoV}-\varepsilon\le |z|\le \sqrt{\rhoV}+\varepsilon\}$. We have 
\[
\int g_{\varepsilon}(z) \, \mu_n(dz) \ =\  
- \frac{1}{2\pi} \int g_{\varepsilon}(z) \, \Delta U_{\mu_n}(z) \, \ell(dz) \ =\   
 \frac{2}{\pi} \int \partial_z g_{\varepsilon}(z) \, b_n(z) \, \ell(dz)  
\]
where $b_n$ is defined in \eqref{def:b}. Notice that 
\begin{equation}\label{eq:eps-to-zero}
\lim_{\varepsilon\downarrow 0} \int g_{\varepsilon}(z) \, \mu_n(dz) = \mu_n(C) .
\end{equation}
By replacing $|z|=\sqrt{x^2+y^2}$ and computing $\partial_z=\frac 12(\partial_x -\ii \partial_y)$, we get
\[
\partial_z g_{\varepsilon} (z) \ =\ \frac {\bar{z}}{2 \varepsilon |z|} \phi'\left( \varepsilon^{-1}\left( |z|-\sqrt{\rhoV}\right)\right) .
\]
Hence, replacing $b_n$ by its expression in \eqref{def:b}, we obtain
\begin{align*}
 \frac{2}{\pi} \int \partial_z g_{\varepsilon}(z) \, b_n(z) \, \ell(dz) &= -\frac 1{2 \varepsilon n\pi} \int |z| \phi'\left( \varepsilon^{-1}\left( |z|-\sqrt{\rhoV}\right)\right)  \tr \Psi(\qvec(|z|)\, \ell(dz)\\
 &\stackrel{(a)}=  -\frac 1{2 \varepsilon n\pi} \int_{\theta=0}^{2\pi} \int_{\rho=\sqrt{\rhoV}-\varepsilon}^{\sqrt{\rhoV}+\varepsilon}  \phi'\left( \varepsilon^{-1}\left( |z|-\sqrt{\rhoV}\right)\right)  \tr \Psi(\qvec(\rho))\, \rho^2\, d\rho\,d\theta \\
 &\stackrel{(b)}= -\frac 1{\varepsilon n} \int_{-1}^{1}  \left( \sqrt{\rhoV} +\varepsilon u\right)^2 \phi'(u)  \tr \Psi(\qvec(\sqrt{\rhoV} +\varepsilon u))\, \varepsilon\, du
\end{align*}
where $(a)$ follows from a change of variables in polar coordinates and $(b)$, from the change of variable $u=\frac{\rho -\sqrt{\rhoV}}{\varepsilon}$.
Since $n^{-1} \tr \Psi(\qvec(|z|))\le |z|^{-2}$, the dominated convergence theorem yields
\begin{multline*}
-\frac {1}{ n} \int_{-1}^{1}  \left( \sqrt{\rhoV} +\varepsilon u\right)^2 \phi'(u)  \tr \Psi(\qvec(\sqrt{\rhoV} +\varepsilon u))\, du\\
\xrightarrow[\varepsilon \downarrow 0]{}\  
-\frac { \rhoV}{n} \tr \Psi (\qvec(\sqrt{\rhoV}) )\int_{-1}^1 \phi'(u)\, du \\
= -\frac { \rhoV }n \tr \Psi (\qvec(\sqrt{\rhoV})) \left[ \phi(1)-\phi(-1)\right] \ =\  0\ .
\end{multline*}
Equating with \eqref{eq:eps-to-zero}, we finally conclude $\mu_n(C)=0$.

By Theorem~\ref{thm:master}--\eqref{q:contdif}, the mapping $z\mapsto \qvec(|z|)$ is 
continuously differentiable on the open set 
${\mathcal D} := \{ z \in \C \, : \, |z| \neq 0, \ |z| \neq \rho(V)^{1/2} \}$. 
Therefore, $b_n(z)$ is continuously differentiable on this set, and 
for any $g \in C^\infty_c({\mathcal D})$, we get 
\[
\int_\C g(z) \, \mu_n(dz) = 
- \frac{1}{2\pi} \int_\C g(z) \, \Delta U_{\mu_n}(z) \, \ell(dz) =  
- \frac{2}{\pi} \int_\C g(z) \,  \partial_z b_n(z) \, \ell(dz) 
= \int_\C g(z) \, f_n(z) \, \ell(dz) 
\]
where the density $f_n(z)$ is given by 
\begin{align*}
f_n(z) &:= -\frac{2}{\pi} \partial_z b_n(z) 
= \frac{1}{n\pi} \partial_z \left( z \tr \Psi(\qvec(|z|)) \right) 
= \frac{1}{n\pi} \left\{ \tr \Psi(\qvec(|z|))  + |z|^2 
\partial_{|z|^2} \tr \Psi(\qvec(|z|)) \right\} \\
&= \frac{1}{n\pi} \partial_{|z|^2} \left\{ |z|^2 
 \tr \Psi(\qvec(|z|)) \right\} = 
  \frac{-1}{\pi n} \sum_{i=1}^n 
  \partial_{|z|^2} \Bigl( \frac{\varphi_i \tilde\varphi_i} 
                   {|z|^2 + \varphi_i \tilde\varphi_i} \Bigr) \\
&= \frac{-1}{\pi n} \partial_{|z|^2} \langle \q(|z|), V \qtilde(|z|)\rangle . 
\end{align*} 
Since $f_n(z)$ depends only on $|z|$, this density is rotationally invariant. 
From Theorem~\ref{thm:master}--\eqref{q:sys1}, $f_n(z) = 0$ for $|z| > \sqrt{\rhoV}$. 
Thus, the support of $\mu_n$ is contained in $\overline{\mathcal B(0,\sqrt{\rhoV})}$. 
Moreover, $F_n(\sqrt{\rhoV}) = 1 = \lim_{v\uparrow\sqrt{\rhoV}} F_n(v)$, 
since $\mu_n(C) = 0$. Given $0 < s \leq v < \sqrt{\rhoV}$, we have 
\begin{align*}
F_n(v) - F_n(s) &= \int_{\mathcal B(0,v)\setminus \mathcal B(0,s)} f_n(z) \, \ell(dz)  
= \int_0^{2\pi} d\theta \int_s^v  
\frac{-1}{2\pi r n} \partial_r \langle \q(r), V \qtilde(r)\rangle \ r \, dr \\
&= 
\frac{1}{n} \langle \q(s), V \qtilde(s)\rangle - 
\frac{1}{n} \langle \q(v), V \qtilde(v)\rangle . 
\end{align*}
By taking $v\uparrow \sqrt{\rhoV}$, $\langle \q(v), V \qtilde(v)\rangle 
\to 0$, and we get the expression~\eqref{expF}. 
Finally, the continuity of the density of $F_n$ on $(0,\sqrt{\rhoV})$ 
follows from Theorem~\ref{thm:master}--\eqref{q:contdif}.
\end{proof}

\subsection{Proof of Theorem~\ref{thm:bistochastic}}	\label{sec:bistochastic}

In Example \ref{ex:bistochastic}, it has been proved that Theorem \ref{thm:bistochastic} holds under the additional assumption that the matrices $A_n$ are irreducible.
Now for the general case, by conjugating $Y_n$ by a permutation matrix we may assume $A_n$ takes the form


\begin{equation}
A_n = \begin{pmatrix} 
A_n^{(1)} &  \cdots & 0\\ 
 \vdots &  \ddots  & \vdots\\
0 & \cdots  &A_n^{(m)} \end{pmatrix}
\end{equation}

where $A_n^{(1)},\dots, A_n^{(m)}$ are square irreducible matrices of respective dimension $n_1\ge \cdots \ge n_m$. 
Indeed, for $A_n$ a general nonnegative matrix we can achieve this with the upper triangular blocks not necessarily zero, but these are forced to be zero by the stochasticity condition.
Also, by \ref{ass:sigmax} and the row-sum constraint applied to the last row of $A_n=(\sigma_{ij})$,
$
1 = \frac1n\sum_{j=1}^n \sigma_{ij}^2 \le \frac{n_m}{n} \smax^2 
$
so in fact we have 
\begin{equation}	\label{ns:lb}
n_1,\dots, n_m\ge n/\smax^2.
\end{equation}

Denote the corresponding submatrices of $X_n$ by $X_n^{(k)}$ and set 
\begin{equation}
Y_n^{(k)} = \frac{1}{\sqrt{n}}A_n^{(k)}\odot X_n^{(k)} = \frac{1}{\sqrt{n_k}}B_n^{(k)}\odot X_n^{(k)}
\end{equation}
where we set $B_n^{(k)}= (n_k/n)^{1/2}A_n^{(k)}$. 
For each $k$ we have:
\begin{enumerate}
\item $A_n^{(k)}$ is irreducible,
\item $\frac1{n_k}B_n^{(k)}\odot B_n^{(k)}$ is doubly stochastic, and
\item $n_k\to \infty$ as $n\to \infty$ (by \eqref{ns:lb}).
\end{enumerate}
Thus, for each $k$ the ESD $\mu_n^{(k)}$ of $Y_n^{(k)}$ converges weakly in probability to $\mu_{\textrm{circ}}$.
Since the $\mu_n^Y$ is the weighted sum:
\[
\mu_n^Y = (n_1/n)\mu_n^{(1)} + \cdots +(n_m/n) \mu_n^{(m)}
\] 
we get that $\mu_n^Y$ converges weakly in probability to $\mu_{\textrm{circ}}$. 
This concludes the proof of Theorem \ref{thm:bistochastic}.

\begin{appendix}

\section{Remaining proofs}

\subsection{Stieltjes transform of a symmetric probability measure} 
We note that a symmetric probability distribution $\check\nu$ on $\R$ satisfies 
$\check\nu(A) = \check\nu(-A)$ for each Borel set $A\subset \R$. 
\begin{lemma}
\label{nusym} 
A probability measure $\check\nu$ is symmetric if and only if its Stieltjes 
transform $g_{\check\nu}$, seen as an analytic function on $\C\setminus\R$,
satisfies $g_{\check\nu}(-\eta) = - g_{\check\nu}(\eta)$. 
\end{lemma}
\begin{proof}
The necessity is obvious from the definition of the Stieltjes 
transform and from the fact that $\check\nu(d\lambda) = \check\nu(-d\lambda)$. 
To prove the sufficiency, we use the Perron inversion formula, that says that
for any function $\varphi \in C_c(\R)$, 
\[
\int_\R \varphi(x) \, \check\nu(dx) = 
\lim_{\varepsilon\downarrow 0} \frac{1}{\pi} \int_\R \varphi(x) \, 
\Imm g_{\check\nu}(x + \ii \varepsilon) \, dx .
\]
By a simple variable change at the right hand side, and by using the 
equalities $g_{\check\nu}(-\eta) = - g_{\check\nu}(\eta)$ and 
$g_{\check\nu}(\bar\eta) = \bar g_{\check\nu}(\eta)$, we obtain that 
$\int \varphi(x) \, \check\nu(dx) = \int \varphi(-x) \, \check\nu(dx)$, as desired.  
\end{proof}

\subsection{Variance estimates}\label{app:variance-estimates} 
In this section we collect without proofs a number of standard variance estimates. Let $(Y_n)$ be a sequence of matrices as in Definition \ref{def:model}. In the sequel we drop the subscript $n$. Denote by $(\vec{e}_i)$ the standard vector basis. 
We introduce the following notations:
$$
Y=\left( \vec{y}_1, \cdots, \vec{y}_n \right)\quad \text{and}\quad Q(\eta^2)= \left[ \sum_{i=2}^n (\vec{y}_i -z\vec{e}_i)(\vec{y}_i -z\vec{e}_i)^* -\eta^2\right]^{-1}\ .
$$ 
Recall the definition of matrices $\Res$ and $G$ in \eqref{def:res0}.
\begin{prop}\label{prop:var-estimates} Let \ref{ass:moments} and \ref{ass:sigmax} hold. Let $\Delta$ be a $n\times n$ deterministic diagonal matrix, then the following estimates hold:
\begin{eqnarray}
\var(\Res_{ij}) &=& \Oeta{n^{-1}}\quad \text{for}\quad 1\le i,j\le 2n\ ,\label{eq:var-H}\\
\var\left( \frac 1n \tr \Delta G \right)&=& \Oeta{\|\Delta\|^2\, n^{-2}}\ ,\label{eq:var-trace}\\
\var\left[ (\vec{y}_1 -z\vec{e}_1)^* [\eta Q]^\alpha (\vec{y}_1 -z\vec{e}_1)\right] &=& \Oeta{n^{-1}}\quad\text{for}  \quad 
\alpha=1,2\ .\label{eq:var-quadra}
\end{eqnarray}
Similar estimates hold true if $G$ is replaced by $\widetilde{G}$, if one considers the columns of $Y^*$ instead of those of $Y$, etc.
\end{prop}

These estimates can be obtained as in the proof of \cite[Proposition 6.3]{najim-yao-2016}, see also the references therein.

As a direct corollary of the previous proposition, we have:
\begin{coro}
\label{coro:var-estimates} Let \ref{ass:moments} and \ref{ass:sigmax} hold.
\begin{eqnarray}
\var\Bigl[\Bigl(\eta +\frac 1n \tr \Delta G\Bigr)^{-1}\Bigr] 
&=& \Oeta{\|\Delta\|^2\, n^{-2}}\ ,\label{eq:var-frac}\\
\var\left[\left(\eta + (\vec{y}_1 
-z\vec{e}_1)^* [\eta Q](\vec{y}_1 -z\vec{e}_1)\right)^{-1}\right] 
&=& \Oeta{n^{-1}}\ .\label{eq:var-frac-quadra}
\end{eqnarray} 
\end{coro}
\begin{proof}
Let us establish \eqref{eq:var-frac}. Notice first that $\left|\eta +\frac 1n \tr \Delta G\right|^{-1} \le \im^{-1}(\eta)$ by \eqref{eq:ST-estimate}.
\begin{eqnarray*}
\var\left( \frac 1{\eta +\frac 1n \tr \Delta G}\right) 
&\stackrel{(a)}\le&  \E \left|  \frac 1{\eta +\frac 1n \tr \Delta G} - \frac 1{\E\left(\eta +\frac 1n \tr \Delta G\right)}\right|^2\ ,\\
&\le& \frac 1{\im^4(\eta)} \var\left( \frac 1n \tr \Delta G\right)\quad =\quad \Oeta{\frac{\|\Delta\|^2}{n^2}}\ ,
\end{eqnarray*}
where $(a)$ follows from the fact that $\var(X)=\inf_a\E |X-a|^2$. 
Estimate \eqref{eq:var-frac-quadra} can be established similarly. 
\end{proof}


\subsection{Proof of Proposition \ref{prop:QR-interpolation}}
\label{app:proof-interpolation}

Notice first that it is sufficient to prove 
\[
\frac 1n \tr \E\, \boldsymbol{C} \left(\Res(z,\eta) - \Res^{\mathcal N}(z,\eta) \right) 
= \vOeta{\frac 1{\sqrt{n}}} \, , 
\]
where $\boldsymbol{C}$ is any of the following $2n\times 2n$ matrices:
$$\begin{pmatrix} I_n & 0 \\ 0 & 0 \end{pmatrix}\, , 
\ \begin{pmatrix} 0 & 0 \\ I_n & 0 \end{pmatrix}\, , 
\ \begin{pmatrix} 0 & I_n \\ 0 & 0 \end{pmatrix}\, ,
\ \begin{pmatrix} 0 & 0 \\ 0 & I_n \end{pmatrix}\, .
$$ 
We follow the approach used in \cite{tao-vu-(am)11} and perform an entry-by-entry interpolation between $Y$ and $Y^{\mathcal N}$. Let $\alpha : [n^2] \to [n]^2$ be the bijection
defined as $\alpha(m) = (i,j)$, where $m-1 = (i-1) n + j-1$ is the Euclidean 
division of $m-1$ by $n$. Given $m \in [n^2]$, define the $n\times n$ matrices
$Z^m = ( Z^m_{i,j} )$ and 
$W^m = ( W^m_{i,j})$ by
\[
Z^{m}_{i,j} := \left\{\begin{array}{ll} 
Y_{i,j}^{\mathcal N} & \text{if} \ \alpha^{-1}(i,j) <  m, \\ 
Y_{i,j}  & \text{if} \ \alpha^{-1}(i,j) \geq  m , 
\end{array}\right. 
 \quad \text{and} \quad 
W^{m}_{i,j} := \left\{\begin{array}{ll} 
Y_{i,j}^{\mathcal N} & \text{if} \ \alpha^{-1}(i,j) < m,  \\ 
0 & \text{if} \ \alpha^{-1}(i,j) = m,  \\ 
Y_{i,j}  & \text{if} \ \alpha^{-1}(i,j) >  m . 
\end{array}\right. 
\]
Notice that $Z^1=Y$. By convention, we denote $Z^{n^2+1} = Y^{\mathcal N}$.
Redenoting the resolvent $\Res(z,\eta)$ defined in~\eqref{def:res0} as 
$\Res_{Y}$ to express the dependence on $Y$ (thus, 
$\Res^{\mathcal N} = \Res_{Y^{\mathcal N}}$), we have  
\begin{align} 
\frac 1n \tr \E\, \boldsymbol{C} \left(\Res_{Y} - \Res_{Y^{\mathcal N}}\right) 
& =  \frac 1n \sum_{m=1}^{n^2}
  \tr \E\, \boldsymbol{C} \left(\Res_{Z^m} - \Res_{Z^{m+1}} \right) \nonumber \\
& =  \frac 1n \sum_{m=1}^{n^2}
  \tr \E\, \boldsymbol{C} \left(\Res_{Z^m} - \Res_{W^{m}} \right) - 
  \tr \E \, \boldsymbol{C} \left(\Res_{Z^{m+1}} - \Res_{W^{m}} \right) . 
\label{eq:telescope} 
\end{align} 
For $(i,j) = \alpha(m)$, the matrices $\bs\Delta_m := Z^m - W^m$ and 
$\bs\Delta^{\mathcal N}_m := Z^{m+1} - W^m$ are given by the equations
\[
\bs\Delta_m = \begin{bmatrix} 0 & Y_{ij} {\bs e}_i {\bs e}_j^* \\
 \bar Y_{ij} {\bs e}_j {\bs e}_i^* & 0 \end{bmatrix} 
\quad \text{and} \quad 
\bs\Delta_m^{\mathcal N} = 
\begin{bmatrix} 0 & Y_{ij}^{\mathcal N} {\bs e}_i {\bs e}_j^* \\
 \bar Y_{ij}^{\mathcal N} {\bs e}_j {\bs e}_i^* & 0 \end{bmatrix} , 
\]
where ${\bs e}_i$ is the $i^{\text{th}}$ canonical vector of $\C^n$. These
matrices are (at most) rank-two matrices that are both independent of 
$W^m$. We now use the identity 
$\Res_{Z^m} = \Res_{W^{m}} - \Res_{Z^m} \bs\Delta_m \Res_{W^{m}}$ three times 
to obtain 
\[
\E C(\Res_{Z^m} - \Res_{W^m}) = - \E C \Res_{W^m} \bs\Delta_m \Res_{W^m}  
 + \E \Res_{W^m} ({\bs\Delta}_m \Res_{W^m})^2 
 - \E \Res_{Z^m}  ({\bs\Delta}_m \Res_{W^m})^3 . 
\]
Similarly, 
\[
\E C(\Res_{Z^{m+1}} - \Res_{W^m}) 
 = - \E C \Res_{W^m} \bs\Delta_m^{\mathcal N} \Res_{W^m}  
 + \E \Res_{W^m} ({\bs\Delta}_m^{\mathcal N} \Res_{W^m})^2 
 - \E \Res_{Z^m}  ({\bs\Delta}_m^{\mathcal N} \Res_{W^m})^3 . 
\]
The first two terms at the right hand sides of these equations are
identical since the first two moments of $Y_{ij}$ and $Y_{ij}^{\mathcal N}$ 
are equal, thus, they cancel out in \eqref{eq:telescope}. Moreover, it is easy
to see from the general properties of the resolvents and from the expressions
of $\bs\Delta_m$ and $\bs\Delta_m^{\mathcal N}$ that the traces of the 
terms with the cubes are bounded by $C n^{-3/2} / \eta^4$. The result follows.

\subsection{Proof of estimate \eqref{eq:invertibility-I-AA}} \label{app:linear-algebra}
This estimate relies on the following standard proposition, whose proof can be found in \cite[Lemma 5.2]{hachem-et-al-2008}.
\begin{prop}\label{prop:folklore}
 Let $C\succcurlyeq 0$ be a $n\times n$ matrix and $\bs{u}=(u_{\ell}) \succ 0$ and $\bs{v}=(v_{\ell}) \succ 0$ two $n\times 1$ vector. Assume that the following equality holds true:
$
\bs{u} = C\bs{u} +\bs{v}\ .
$  
Then $\rho(C)<1$, matrix $I-C$ is invertible, $(I-C)^{-1} \succcurlyeq 0$ and 
$$
\lrn (I-C)^{-1}\rrn_{\infty} \ \le \ \frac{\max(u_{\ell}\, ; \ \ell \in [n])}{\min( v_{\ell}\, ; \ \ell\in [n])}\ . 
$$
\correction{\textcolor{red}{Moreover, $\det(I-C)>0$.}}
\end{prop}

\correction{
\textcolor{red}{
\begin{proof}
We only prove the last property, for the other properties' proof can be found in \cite[Lemma 5.2]{hachem-et-al-2008}. Let $t\in [0,1]$ and notice that $\rho(tC)=t\rho(C)<1$, hence $I-tC$ is invertible for all $t\in [0,1]$. Since this mapping $t\mapsto \det(I-tC)$ is continuous for $t\in [0,1]$ with value 1 at $t=0$, its sign must remain constant and positive for all $t\in [0,1]$, for otherwise $\det(I-t_0C)=0$ for some $t_0\in (0,1)$, which contradicts the invertibility of $I-t_0C$.
\end{proof}
}}
\correction{
\textcolor{red}{
\begin{lemma}\label{lemma:key-lemma}
Let $A$ and $B$ be two $n\times n$ matrices such that $\rho(A\odot\bar{A})<1$ and $\rho(B\odot \bar{B})<1$. Let $\bs{u}, \bs{v}, \bs{w},\bs{r}$ be $n\times 1$ vectors, then 
\begin{enumerate}
\item matrix $I-A\odot B$ is invertible,
\item the following inequality holds true:
\begin{multline*}
\left| 
(\bs{u}\odot \bs{v})^\tran (I-A\odot B)^{-1} \bs{w}\odot\bs{r}
\right| \\
\le \ \sqrt{
(\bs{u}\odot \bar{\bs{u}})^\tran (I-A\odot \bar{A})^{-1} \bs{w}\odot\bar{\bs{w}}
}\ 
\sqrt{
(\bs{v}\odot \bar{\bs{v}})^\tran (I-B\odot \bar{B})^{-1} \bs{r}\odot\bar{\bs{r}}
}\ .
\end{multline*}
\item the following max-row norm estimate holds true
\begin{equation}\label{eq:estimate-maxrow}
\lrn (I-A\odot B)^{-1}\rrn_{\infty} \le 
\sqrt{\lrn (I-A\odot \bar{A})^{-1}\rrn_{\infty}}\ 
\sqrt{ \lrn (I-B\odot \bar{B})^{-1}\rrn_{\infty}} \ .
\end{equation}
\end{enumerate}
\end{lemma}
The proof of Lemma \ref{lemma:key-lemma} is postponed to Section \ref{sec:proof-key-lemma}.  
In order to study the properties of matrix ${\mathcal A}(\E\, \bs{\vec g}) \odot {\mathcal A}(\bs{p})$, we introduce two auxiliary systems. 
}}
Recall the definitions \eqref{def-upsilon} and \eqref{def:Delta} of $\bs{\Upsilon}(\bs{\vec b})$ and $\Delta(\bs{b})$, $\widetilde \Delta(\bs{\tilde b})$.
In order to study the properties of matrix ${\mathcal A}(\bs{p}) \odot {\mathcal A}(\bs{p})$, we introduce an auxiliary system. 
Since the $p_i$'s satisfy
\eqref{cve0}, we immediately obtain
\begin{multline*}
\im(p_i) = \frac{\left[ V^\tran \Imm(\bs{p})\right]_i\ |z|^2}{\left| \ |z|^2 - (\eta +[V\bs{\tilde p}]_i)(\eta +[V^\tran\bs{p}]_i)\right|^2} 
+\frac{ [V\Imm(\bs{\tilde p})]_i}{\big| - (\eta +[V\bs{\tilde p}]_i) +\frac{|z|^2}{\eta +[V^\tran\bs{p}]_i}\big|^2}  \\
+
\frac{  \Imm(\eta)}{\big| - (\eta +[V\bs{\tilde p}]_i) +\frac{|z|^2}{\eta +[V^\tran\bs{p}]_i}\big|^2}\left( \frac{|z|^2}{\left| \eta +[V^\tran\bs{p}]_i)\right|^2}+1\right) 
\end{multline*}
and its counterpart for $\im(\tilde p_i)$. Denote 
by $\bs{v}({\bs{\vec p}})$ the $2n\times 1$ vector defined by
$$
\left[\bs{v}({\bs{\vec p}})\right]_i = \frac{  \Imm(\eta)}{\big| - (\eta +[V\bs{\tilde p}]_i) +\frac{|z|^2}{\eta +[V^\tran\bs{p}]_i}\big|^2}\left( \frac{|z|^2}{\left| \eta +[V^\tran\bs{p}]_i)\right|^2}+1\right) 
$$
for $i\in [n]$ and
$$
\left[\bs{v}({\bs{\vec p}})\right]_i = 
\frac{  \Imm(\eta)}{\big| - (\eta +[V^\tran\bs{p}]_i) +\frac{|z|^2}{\eta +[V\bs{\tilde p}]_i}\big|^2}\left( \frac{|z|^2}{\left| \eta +[V\bs{\tilde p}]_i)\right|^2}+1\right) 
$$
for $i\in \{n+1,\cdots,2n\}$. Then the system satisfied by 
$\im({\bs{\vec p}})$ writes
\begin{equation}\label{eq:system-aux-1}
\im(\vec{\bs{p}}) = {\mathcal A}(\vec{\bs{p}}) \odot \overline{{\mathcal A}(\vec{\bs{p}})} \, \im(\vec{\bs{p}}) + \bs{v}(\vec{\bs{p}})\ ,
\end{equation}
where matrix ${\mathcal A}(\vec{\bs{p}})$ has been defined in \eqref{eq:def-A}. Since matrix ${\mathcal A}(\vec{\bs{p}}) \odot \overline{{\mathcal A}(\vec{\bs{p}})}$ has nonnegative entries, we will rely on Proposition \ref{prop:folklore} to evaluate
$\lrn \left(I-{\mathcal A}(\vec{\bs{p}}) \odot \overline{{\mathcal A}(\vec{\bs{p}})}\right)^{-1}\rrn_{\infty}\ .
$
We need to check that $\im(\vec{\bs{p}}),\bs{v}(\vec{\bs{p}}) \succ 0$, to upper bound $\im(p_{\ell}),\im(\tilde p_{\ell})$ and to lower bound $\left[\bs{v}(\vec{\bs{\theta}})\right]_i$. Since $p_{\ell}$ and $\tilde p_{\ell}$ are Stieltjes transform, we have $
|\im(p_{\ell})| \vee | \im(\tilde p_{\ell})| \le  \im(\eta)^{-1}$ which we write in short $\im(\vec{\bs p}) \prec \im(\eta)^{-1}$.
Now, if $i\le n$  
$$
\left[\bs{v}(\vec{\bs{p}})\right]_i \ge \frac{\Imm (\eta)}{\big| - (\eta +[V\bs{\tilde p}]_i) +\frac{|z|^2}{\eta +[V^\tran\bs{ p}]_i}\big|^2}$$
with
$$
\frac 1{\Imm(\eta)}\left| - (\eta +[V\bs{\tilde p}]_i) +\frac{|z|^2}{\eta +[V^\tran\bs{ p}]_i}\right|^2
\ \le\  2\left( \frac{|\eta|^2}{\Imm(\eta)} +\frac {\smax^4}{\im^3(\eta)} +\frac{|z|^4}{\im^3(\eta)}\right) \ =\  \Oeta{1}\ .
$$
The case where $n+1\le i\le 2n$ being handled similarly, we finally get
\begin{equation}\label{eq:bound-2}
\min_i \left[\bs{v}(\vec{\bs{\theta}})\right]_i\ \ge\ \frac 1{\Oeta{1}} \ .
\end{equation}
Hence, $\bs{v}(\vec{\bs{\theta}})\succ 0$ and $\left[\bs{v}(\vec{\bs{\theta}})\right]_i$ is lower bounded away from zero.

In order to prove $\vec{\bs{p}}\succ 0$, we argue as follows: the $p_i$'s are Stieltjes transforms of probability measures $\mu_i$.
These probability measures are tight, see for instance Proposition \ref{prop:deteq}-(v). In particular, there exists a real number $K$ such that 
$
\mu_i([-K,K]) \ge \frac{1}2\ .
$
Hence,
\begin{eqnarray}
\im(p_i) &=&\im(\eta) \int_{\mathbb{R}} \frac{\mu_i(d \lambda)}{|\lambda - \eta|^2} 
\ \ge\   \im(\eta) \int_{-K}^K \frac{\mu_i(d \lambda)}{|\lambda - \eta|^2} \nonumber \ \ge\  \frac{\mu_{i}([-K,K])}{2(K^2 + |\eta|^2)}\ge \frac{1}{4(K^2+|\eta|^2)}\ .
\end{eqnarray}
We are now in position to apply Proposition \ref{prop:folklore}. This proposition yields in particular that $\rho({\mathcal A}(\vec{\bs{p}}) \odot \overline{{\mathcal A}(\vec{\bs{p}})})<1$ and gathering the estimates $\im(\vec{\bs p}) \prec \im(\eta)^{-1}$ and \eqref{eq:bound-2}, we obtain
\begin{equation}\label{eq:estimate-AA}
\lrn\left(I-{\mathcal A}(\vec{\bs{p}}) \odot \overline{{\mathcal A}(\vec{\bs{p}})}\right)^{-1}\rrn_{\infty} = \Oeta{1}\ .
\end{equation}
We can now conclude. Notice that along the imaginary axis,
${\bs p}(\ii t)=\ii {\bs r}(t)$ and $\tilde {\bs p}(\ii t)=\ii \tilde {\bs r}(t)$ by Proposition \ref{prop:deteq}-\eqref{SD-prop3}.  Hence, straightforward computations yield that  
$$
I- {\mathcal A}(\bs{\vec p})\odot {\mathcal A}(\bs{\vec p}) =\begin{pmatrix}
A & B\\
C & D
\\
\end{pmatrix}\quad \textrm{and}\quad 
I- {\mathcal A}(\bs{\vec p})\odot \overline{{\mathcal A}(\bs{\vec p})} =\begin{pmatrix}
A & -B\\
-C & D
\\
\end{pmatrix}\ .
$$
By standard block inversion formulas, the same structure occurs for the inverses
$$
\left( I- {\mathcal A}(\bs{\vec p})\odot {\mathcal A}(\bs{\vec p})\right)^{-1}=
\begin{pmatrix}
\widetilde{A} & \widetilde{B}\\
\widetilde{C} & \widetilde{D}
\\
\end{pmatrix}\quad  \textrm{and}\quad 
\left(  I- {\mathcal A}(\bs{\vec p})\odot \overline{{\mathcal A}(\bs{\vec p})}\right)^{-1}=\begin{pmatrix}
\widetilde{A} & -\widetilde{B}\\
-\widetilde{C} & \widetilde{D}
\\
\end{pmatrix}\, ,
$$
from which we immediately deduce that $\lrn\left(I-{\mathcal A}(\vec{\bs{p}}) \odot \overline{{\mathcal A}(\vec{\bs{p}})}\right)^{-1}\rrn_{\infty}
= \lrn\left(I-{\mathcal A}(\vec{\bs{p}}) \odot {\mathcal A}(\vec{\bs{p}})\right)^{-1}\rrn_{\infty}$ by either optimizing over $\vec{\bs u}=({\bs u},\tilde{\bs u})$ with 
$\| \vec{\bs u}\|_{\infty}=1$ or  $({\bs u},-\tilde{\bs u})$. Estimate \eqref{eq:invertibility-I-AA} is established.

\correction{
\textcolor{red}{
If one considers now the perturbed system \eqref{eq:Gtilde-ii}-\eqref{eq:G-ii} satisfied by $\E\, \vec{\bs{g}}$, one obtains similarly
\begin{equation}\label{eq:system-aux-2}
\im(\E \vec{\bs{g}}) = {\mathcal A}(\E \vec{\bs{g}}) \odot \overline{{\mathcal A}(\E \vec{\bs{g}})} \ \im(\E \vec{\bs{g}}) + \bs{v}(\E \vec{\bs{g}})
+\Oeta{\frac 1{n^{3/2}}}\ .
\end{equation}
By arguing as before (notice in particular that there is no impact of the residual term $\vec{\mathcal O}_{\eta}\left(n^{-3/2}\right)$), we obtain 
\begin{equation}\label{eq:estimate-BB}
\rho\left({\mathcal A}(\E \vec{\bs{g}}) \odot \overline{{\mathcal A}(\E \vec{\bs{g}})}\right) \ <\ 1\quad \text{and}\quad 
\lrn (I-{\mathcal A}(\E \vec{\bs{g}}) \odot \overline{{\mathcal A}(\E \vec{\bs{g}})})^{-1}\rrn_{\infty} = \Oeta{1}\ .
\end{equation}
It now remains to bound 
$$
\lrn \left(I-{\mathcal A}( \vec{\bs{p}}) \odot {{\mathcal A}(\E \vec{\bs{g}})}\right)^{-1}\rrn_{\infty}
$$ 
given estimates \eqref{eq:estimate-AA} and \eqref{eq:estimate-BB}. Lemma \ref{lemma:key-lemma}-(3) provides the appropriate estimate.
Consider now matrices ${\mathcal A}( \vec{\bs{p}})$ and ${\mathcal A}(\E \vec{\bs{g}})$ as defined in \eqref{eq:def-A}. We have already proved that $\rho({\mathcal A}(\vec{\bs{p}}) \odot \overline{{\mathcal A}(\vec{\bs{p}})})<1$ and $\rho({\mathcal A}(\E \vec{\bs{g}}) \odot \overline{{\mathcal A}(\E \vec{\bs{g}})})<1$ hence 
$I-{\mathcal A}( \vec{\bs{p}}) \odot {{\mathcal A}(\E \vec{\bs{g}})}$ is invertible. Plugging estimates \eqref{eq:estimate-AA} and \eqref{eq:estimate-BB} into \eqref{eq:estimate-maxrow}, we obtain:
$$
\lrn (I-{\mathcal A}( \vec{\bs{p}}) \odot {{\mathcal A}(\E \vec{\bs{g}})})^{-1}\rrn_{\infty} =\Oeta{1}\ ,
$$
which is the desired result. Proposition \ref{prop:inversion} is proved. 
}}

\correction{
\textcolor{red}{
\subsection{Proof of Lemma \ref{lemma:key-lemma}}\label{sec:proof-key-lemma}  
The proof of parts $(1)$ and $(2)$ proceeds by induction on the dimension $n$ and by computing the Schur complement of matrices under investigation.
\subsubsection*{The induction assumption, its verification for $n=1$ and some notations} The induction assumption is the following: let $A$ and $B$ be $n\times n$ matrices with $\rho(A\odot \bar{A})<1$ and $\rho(B\odot \bar{B})<1$, then $\det(I-A\odot B)\neq 0$ (or equivalently $I -A\odot B)$ is invertible) and the inequality of the lemma holds true. 
We first verify the induction assumption for $n=1$. In this case, the matrices are scalar.
$$
A=a\, ,\ B=b\, ,\ A\odot B=ab\, ,\ A\odot \bar{A}=|a|^2\, ,\ B\odot \bar{B} = |b|^2\, ,\ |a|<1\, ,\ |b|<1\ .
$$ 
Then $\det(I-A\odot B)=1-ab \neq0$ and 
$$
\frac{|uvwr|}{|1-ab|}\quad \le\quad  \sqrt{\frac{|uw|^2}{1-|a|^2}}\sqrt{\frac{|vs|^2}{1-|b|^2}}\ ,
$$
which is the desired inequality.
We now assume the induction assumption at step $n$ and prove it at step $n+1$. We first introduce some notations. The tilded quantities refer to step $n+1$ (either $(n+1)\times (n+1)$ matrices or $(n+1)\times 1$ vectors). The untilded quantities refer to their counterparts at step $n$. Plain lowercase letters are scalars.
$$
\tilde A=\left(
\begin{array}{cc}
A & \bs{a}\\
\bs{\alpha}^\tran&a
\end{array}\right)\, ,\
\tilde B=\left(
\begin{array}{cc}
B & \bs{b}\\
\bs{\beta}^\tran&b
\end{array}\right)\, ,\ 
\tilde{\bs{u}}=\left(
\begin{array}{c}
\bs{u}\\
u
\end{array}\right)\, ,\
\tilde{\bs{v}}=\left(
\begin{array}{c}
\bs{v}\\
v
\end{array}\right)\, ,\
\tilde {\bs{w}}=\left(
\begin{array}{c}
\bs{w}\\
w
\end{array}\right)\, ,\
\tilde {\bs{s}}=\left(
\begin{array}{c}
\bs{s}\\
s
\end{array}\right)\, .
$$
We will denote by
$$
\Gamma_{AB} = (I- A\odot  B)^{-1}\ ,\quad \Gamma_A = (I-A\odot \bar A)^{-1}\ ,\quad  \Gamma_B = (I-B\odot \bar B)^{-1}
$$
and consider the following notation for quadratic forms:
\begin{equation}\label{eq:notation-quadratic}
{\mathcal Q}(\bs{u},A,\bs{w}) \ =\  (\bs{u}\odot \bar{\bs{u}})^\tran \Gamma_{A}\,  \bs{w}\odot \bar{\bs{w}}\ ,\quad 
{\mathcal Q}(\bs{v},B,\bs{s}) \ =\ (\bs{v}\odot \bar{\bs{v}})^\tran \Gamma_{B}\,  \bs{s}\odot \bar{\bs{s}}\ ,
\end{equation}
where $\bs{u},\bs{v},\bs{w}$ and $\bs{s}$ are generic $n\times 1$ vectors.
\subsubsection*{Invertibility of $I-\tilde A\odot \tilde B$}
By assumption, $\rho(\tilde A\odot \bar{\tilde A})<1$ and $\rho(\tilde B\odot \bar{\tilde B})<1$. Since $A\odot \bar{A}$ is a principal submatrix of $\tilde A\odot \bar{\tilde A}\succcurlyeq 0$, we have
$$
\rho(A\odot \bar{A}) \le \rho(\tilde A\odot \bar{\tilde A})<1 
$$
and similarly, $\rho(B\odot\bar{B})<1$. By the induction assumption, $I-A\odot B$ is invertible. In particular, $\det(I-A\odot B)\neq 0$.
Consider now the following decomposition
$$
I-\tilde A\odot \tilde B =\left( 
\begin{array}{cc}
I-A\odot B& -\bs{a}\odot\bs{b}\\
-\bs{\alpha}\odot \bs{\beta} & 1-ab
\end{array}\right)\ .
$$
The Schur complement of matrix $I-A\odot B$ in $I- \tilde A \odot \tilde B$ writes:
$$
{\mathcal S}_{AB} = 1-ab - (\bs{\alpha}\odot \bs{\beta})^\tran \Gamma_{AB}\,  \bs{a}\odot \bs{b}\ .
$$
We similarly consider the Schur complement ${\mathcal S}_A$ (resp. ${\mathcal S}_B$) of matrix $I-A\odot \bar{A}$ (resp. matrix $I-B\odot \bar{B}$) in $I- \tilde A \odot \bar{\tilde A}$ (resp. $I- \tilde B \odot \bar{\tilde B}$):
\begin{eqnarray*}
{\mathcal S}_{A} &=& 1-|a|^2 - (\bs{\alpha}\odot \bar{\bs{\alpha}})^\tran \Gamma_{A}\,  \bs{a}\odot \bar{\bs{a}}\ ,\\
{\mathcal S}_{B} &=& 1-|b|^2 - (\bs{\beta}\odot \bar{\bs{\beta}})^\tran \Gamma_{B}\,  \bs{b}\odot \bar{\bs{b}}\ .
\end{eqnarray*}
By Proposition \ref{prop:folklore}, the determinants of $I-\tilde A\odot \bar{\tilde A}$ and $I-A\odot \bar{A}$ are positive. 
By the determinantal formula \cite[Section 0.8.5]{book-horn-johnson} involving ${\mathcal S}_A$, we obtain:
$$
\det(I-\tilde A\odot \bar{\tilde A}) = \det(I-A\odot \bar{A})\times {\mathcal S}_A\ .
$$
Hence ${\mathcal S}_A>0$; similarly, ${\mathcal S}_B>0$.
We now prove that $|{\mathcal S}_{AB}| \neq 0$.
Applying the induction assumption, we get
\begin{eqnarray}
|{\mathcal S}_{AB} | &\ge & 1 - ab -\sqrt{\Qa}\sqrt{\Qb}\nonumber\\
&\ge & \sqrt{1-|a|^2} \sqrt{1-|b|^2} -\sqrt{\Qa}\sqrt{\Qb}\nonumber\\
&\ge & \sqrt{{\mathcal S}_A}\sqrt{{\mathcal S}_B} >0\ ,\label{eq:min-Sab}
\end{eqnarray}
where the last inequality follows from the elementary inequality $\sqrt{ab} - \sqrt{cd}\ge \sqrt{a-c}\sqrt{b-d}$, valid for $(a-c)\wedge(b-d)\ge 0$.
Using again the determinantal formula, we get
$$
\det(I-\tilde A \odot \tilde B)= \det(I - A\odot B) \times {\mathcal S}_{AB} \ \neq\  0\ .
$$
Hence $I-\tilde A \odot \tilde B$ is invertible.
\subsubsection*{The inequality at step $n+1$}
Using the Schur decomposition
$$
\left( I - \tilde A\odot \tilde B\right)^{-1}
=\left( 
\begin{array}{cc}
I &\Gamma_{AB}\, \bs{a}\odot \bs{b}\\
0 &1
\end{array}\right)
\left( \begin{array}{cc}
\Gamma_{AB}&0\\
0&{\mathcal S}^{-1}_{AB}
\end{array}\right)
\left( \begin{array}{cc}
I &0\\
(\bs{\alpha}\odot\bs{\beta})^\tran \Gamma_{AB} &1
\end{array}\right)\ ,
$$
we obtain
\begin{eqnarray}
\lefteqn{(\tilde{\bs{u}}\odot\tilde{\bs{v}})^\tran  \left( I - \tilde A\odot \tilde B\right)^{-1} \tilde{\bs{w}}\odot\tilde{\bs{s}}}\nonumber \\
&=&
(\bs{u}\odot\bs{v})^\tran  \Gamma_{AB} \bs{w}\odot\bs{s} 
+\frac{(\bs{u}\odot\bs{v})^\tran  \Gamma_{AB} \bs{a}\odot\bs{b}\times 
(\bs{\alpha}\odot\bs{\beta})^\tran  \Gamma_{AB} \bs{w}\odot\bs{s}}{{\mathcal S}_{AB}}\nonumber\\
&&\quad + uv\frac{
(\bs{\alpha}\odot\bs{\beta})^\tran  \Gamma_{AB} \bs{w}\odot\bs{s}}{{\mathcal S}_{AB}} + ws \frac{(\bs{u}\odot\bs{v})^\tran  \Gamma_{AB} \bs{a}\odot\bs{b}}{{\mathcal S}_{AB}} + \frac{uvws}{\Sab}\ .\label{eq:schur-decomp}
\end{eqnarray}
A similar computation yields (recall notation \eqref{eq:notation-quadratic})
\begin{eqnarray}
\lefteqn{{(\tilde{\bs{u}}\odot\bar{\tilde{\bs{u}}})^\tran  \left( I - \tilde A\odot \bar{\tilde A}\right)^{-1} \tilde{\bs{w}}\odot\bar{\tilde{\bs{w}}}}}\nonumber\\
&=& {\mathcal Q}(\bs{u},A,\bs{w}) +\frac{{\mathcal Q}(\bs{u},A,\bs{a}){\mathcal Q}(\bs{\alpha},A,\bs{w}) }{\Sa} +|u|^2\frac{{\mathcal Q}(\bs{\alpha},A,\bs{w})}\Sa
+|w|^2 \frac{{\mathcal Q}(\bs{u},A,\bs{a})}\Sa + \frac{|uw|^2}\Sa\ .\nonumber\\
\label{eq:schur-decomp-A}
\end{eqnarray}
Using the induction assumption at step $n$ and the inequality \eqref{eq:min-Sab} over the Schur complements, we can majorize the r.h.s. of \eqref{eq:schur-decomp}:
\begin{eqnarray*}
\lefteqn{\left| 
(\tilde{\bs{u}}\odot\tilde{\bs{v}})^\tran  \left( I - \tilde A\odot \tilde B\right)^{-1} \tilde{\bs{w}}\odot\tilde{\bs{s}}
\right|} \\
&\le &\sqrt{{\mathcal Q}(\bs{u},A,\bs{w})}\sqrt{{\mathcal Q}(\bs{v},B,\bs{s})}
+\frac{\sqrt{{\mathcal Q}(\bs{u},A,\bs{a})}\sqrt{{\mathcal Q}(\bs{v},B,\bs{b})} \sqrt{{\mathcal Q}(\bs{\alpha},A,\bs{w})}\sqrt{{\mathcal Q}(\bs{\beta},B,\bs{s})}}
{\sqrt{\Sa\, \Sb}}\\
&&\quad + |uv|\frac{\sqrt{{\mathcal Q}(\bs{\alpha},A,\bs{w})}\sqrt{{\mathcal Q}(\bs{\beta},B,\bs{s})}}
{\sqrt{\Sa\, \Sb}} 
+ |ws| \frac{\sqrt{{\mathcal Q}(\bs{u},A,\bs{a})}\sqrt{{\mathcal Q}(\bs{v},B,\bs{b})}}
{\sqrt{\Sa\, \Sb}} + 
\frac{|uwvs|}{\sqrt{\Sa\, \Sb}}\ .
\end{eqnarray*}
Using Cauchy-Schwarz inequality $\sum_i \sqrt{x_i y_i} \le (\sum_i x_i)^{1/2} (\sum_i y_i)^{1/2}$ together with \eqref{eq:schur-decomp-A} and its counterpart for 
$(\tilde{\bs{v}}\odot\bar{\tilde{\bs{v}}})^\tran  \left( I - \tilde B\odot \bar{\tilde B}\right)^{-1} \tilde{\bs{s}}\odot\bar{\tilde{\bs{s}}}$, we finally obtain the desired result at step $n+1$.
Parts (1) and (2) are proved.
Part (3) is a simple corollary of the previous inequality. Notice that for a $n\times n$ matrix $C$,
\begin{eqnarray*}
\lrn C\rrn_{\infty}&=& \max_i \Big\{ \sum_{j} |C_{ij}| \Big\} \ =\  \max_i \Big\{ \sum_{j} C_{ij}x_j \ ; \|\bs{x}\|_{\infty} \le 1\Big\}\ ,\\ 
&=& \max_i \Big\{ \sum_{j} C_{ij}x_j y_j \ ; \|\bs{x}\|_{\infty}, \|\bs{y}\|_{\infty} \le 1\Big\} \ .
\end{eqnarray*}
Specializing $\bs{u}=\bs{v}=\bs{e}_i$, $\|\bs{w}\|_{\infty}\le 1$ and $\|\bs{r}\|_{\infty} \le 1$ in the lemma, we obtain:
$$
\left| \sum_{j} \left[(I-A\odot B)^{-1}\right]_{ij} w_j r_j\right| \ \le\ \sqrt{\sum_{j} \left[(I-A\odot \bar{A})^{-1}\right]_{ij} |w_j|^2} \ \sqrt{\sum_{j} \left[(I-B\odot \bar{B})^{-1}\right]_{ij}|r_j|^2}
$$
for $i\in [n]$. Noticing that $\|\bs{w}\|_{\infty} \le 1$ implies that $\|\bs{w}\odot \bar{\bs{w}}\|_{\infty} \le 1$, it remains to optimize over $\bs{w}$ and $\bs{r}$ to conclude.
Lemma \ref{lemma:key-lemma} is proved.}}

\subsection{Proof of Lemma \ref{cvP}}\label{proof:cvP}

Assume without loss of generality that $\zeta$ is a probability measure. Let $\varphi\in C_c(\C)$. Then $
\E\, \varphi(f_n(\cdot,z) - g(z)) \xrightarrow[n\to\infty]{} \varphi(0)
$ for $\zeta$-almost all $z\in\C$ since the convergence in probability induces the convergence in distribution.
Thus, by the dominated convergence and Fubini's theorems,
\[
\int_{\Omega\times \C} \varphi(f_n(\omega, z) - g(z)) \, (\PP \otimes \zeta)(d\omega\times dz) 
\xrightarrow[n\to\infty]{} \varphi(0) \, .  
\]
In other words, $f_n-g$ converges to $0$ in distribution, hence in probability,  for 
the probability measure $\PP \otimes \zeta$. As a consequence (see for instance \cite[Lemma 3.11]{book-kallenberg-second-edition}), 
$$
\int_{\C} |g(z)|^{1+\alpha} \zeta(dz) \ =\  \int_{\Omega\times \C} |g(z)|^{1+\alpha} 
(\PP \otimes \zeta)(d\omega\times dz)\  \le\  C\ .
$$
By \eqref{eq:unif-integrability}, the sequence $(f_n)$ is $\mathbb{P}\otimes \zeta$-uniformly integrable, hence 
$
\int_{\Omega\times \C} |f_n(\omega,z) - g(z)| (\PP \otimes \zeta)(d\omega\times dz)\xrightarrow[n\to\infty]{} 0\ ,
$
see for instance \cite[Proposition 3.12]{book-kallenberg-second-edition}. Convergence \eqref{eq:conv-probab} follows from Markov's inequality. 

\subsection{Proof of Lemma \ref{h:pot}}\label{proof:h:pot}
Let $(\psi_k)_{k\ge 1}$ be a sequence of smooth compactly supported functions, dense in $C_c(\C)$ for the supremum norm $\|\psi\|_{\infty} = \sup_{z\in \C} |\psi(z)|$. By the diagonal extraction procedure, one can find a subsequence $(\zeta_{n'})$ 
such that with probability one $(\zeta_{n'})$ is tight and  
$\int \psi_k d\,\zeta_{n'} \xrightarrow[n'\to\infty]{} - \frac1{2\pi} \int \Delta\psi_k(z) h(z) \,\ell(dz)$
for all $k\ge 1$. Thus, on this set of probability one, the tight sequence 
$(\zeta_{n'})$ has a unique non-random limit point $\zeta$, and this limit point satisfies
$
\zeta = - (2\pi)^{-1} \Delta h
$
in $\mathcal D'(\C)$, the set of Schwartz distributions. With this at hand, we get from the assumption that
\[
\int \psi_k(z) \, \zeta_n(dz) 
\xrightarrow[n\to\infty]{{\mathcal P}} 
\int \psi_k(z) \, \zeta(dz) 
\]
for all $k\ge 1$. By a density argument, we thus get that 
\[
\int \varphi(z) \, \zeta_n(dz) 
\xrightarrow[n\to\infty]{{\mathcal P}} 
\int \varphi(z) \, \zeta(dz) 
\]
for every $\varphi \in C_c(\C)$.

\end{appendix}

\bibliographystyle{abbrv}
\bibliography{math}

\def\cprime{$'$} \def\cdprime{$''$} \def\cprime{$'$} \def\cprime{$'$}
  \def\cprime{$'$} \def\cprime{$'$}
\begin{thebibliography}{10}

\bibitem{Ahmadian:2015xw}
Y.~Ahmadian, F.~Fumarola, and K.~D. Miller.
\newblock Properties of networks with partially structured and partially random
  connectivity.
\newblock {\em Physical Review E}, 91(1):012820, 2015.

\bibitem{ajanki2015universality}
O.~Ajanki, L.~Erd\H{o}s, and T.~Kr\"uger.
\newblock Universality for general wigner-type matrices.
\newblock {\em arXiv preprint arXiv:1506.05098}, 2015.

\bibitem{AEK16:mde}
O.~Ajanki, L.~Erdos, and T.~Kr{\"u}ger.
\newblock Stability of the matrix dyson equation and random matrices with
  correlations.
\newblock Preprint arxiv:1604.08188.

\bibitem{AEKquadequations}
O.~{Ajanki}, L.~{Erdos}, and T.~{Kr{\"u}ger}.
\newblock {Quadratic vector equations on complex upper half-plane}.
\newblock {\em ArXiv e-prints}, June 2015.

\bibitem{ARS}
J.~Aljadeff, D.~Renfrew, and M.~Stern.
\newblock Eigenvalues of block structured asymmetric random matrices.
\newblock {\em J. Math. Phys.}, 56(10):103502, 14, 2015.

\bibitem{ASS}
J.~Aljadeff, M.~Stern, and T.~Sharpee.
\newblock Transition to chaos in random networks with cell-type-specific
  connectivity.
\newblock {\em Phys. Rev. Lett.}, 114:088101, Feb 2015.

\bibitem{allesina2015predicting}
S.~Allesina, J.~Grilli, G.~Barab{\'a}s, S.~Tang, J.~Aljadeff, and A.~Maritan.
\newblock Predicting the stability of large structured food webs.
\newblock {\em Nature communications}, 6, 2015.

\bibitem{AlTa12:nature}
S.~Allesina and S.~Tang.
\newblock Stability criteria for complex ecosystems.
\newblock {\em Nature}, 483(7388):205--208, 2012.

\bibitem{Allesina:2015ux}
S.~Allesina and S.~Tang.
\newblock The stability--complexity relationship at age 40: a random matrix
  perspective.
\newblock {\em Population Ecology}, 57(1):63--75, 2015.

\bibitem{AEKgram}
J.~{Alt}, L.~{Erd{\H o}s}, and T.~{Kr{\"u}ger}.
\newblock {Local law for random Gram matrices}.
\newblock {\em arXiv preprint arXiv:1606.07353}, June 2016.

\bibitem{alt2016local}
J.~Alt, L.~Erd{\"o}s, and T.~Kr{\"u}ger.
\newblock Local inhomogeneous circular law.
\newblock {\em arXiv preprint arXiv:1612.07776}, 2016.

\bibitem{AnZe:band}
G.~W. Anderson and O.~Zeitouni.
\newblock A {CLT} for a band matrix model.
\newblock {\em Probab. Theory Related Fields}, 134(2):283--338, 2006.

\bibitem{bai-sil-book}
Z.~Bai and J.~W. Silverstein.
\newblock {\em Spectral analysis of large dimensional random matrices}.
\newblock Springer Series in Statistics. Springer, New York, second edition,
  2010.

\bibitem{Bai93a}
Z.~D. Bai.
\newblock Convergence rate of expected spectral distributions of large random
  matrices. {I}. {W}igner matrices.
\newblock {\em Ann. Probab.}, 21(2):625--648, 1993.

\bibitem{Bai93b}
Z.~D. Bai.
\newblock Convergence rate of expected spectral distributions of large random
  matrices. {II}. {S}ample covariance matrices.
\newblock {\em Ann. Probab.}, 21(2):649--672, 1993.

\bibitem{Bai-circular-law-1997}
Z.~D. Bai.
\newblock Circular law.
\newblock {\em Ann. Probab.}, 25(1):494--529, 1997.

\bibitem{bai-sil-rmta12}
Z.~D. Bai and J.~W. Silverstein.
\newblock No eigenvalues outside the support of the limiting spectral
  distribution of information-plus-noise type matrices.
\newblock {\em Random Matrices: Theory and Applications}, 1(1), 2012.

\bibitem{BDG:heavytail}
S.~Belinschi, A.~Dembo, and A.~Guionnet.
\newblock Spectral measure of heavy tailed band and covariance random matrices.
\newblock {\em Comm. Math. Phys.}, 289(3):1023--1055, 2009.

\bibitem{BCC:heavytail}
C.~Bordenave, P.~Caputo, and D.~Chafa{\"{\i}}.
\newblock Spectrum of non-{H}ermitian heavy tailed random matrices.
\newblock {\em Comm. Math. Phys.}, 307(2):513--560, 2011.

\bibitem{BCCT}
C.~Bordenave, P.~Caputo, D.~Chafa{\"\i}, and K.~Tikhomirov.
\newblock On the spectral radius of a random matrix.
\newblock Preprint arXiv:1607.05484, 07 2016.

\bibitem{2012-bordenave-chafai-circular}
C.~Bordenave and D.~Chafa{\"{\i}}.
\newblock Around the circular law.
\newblock {\em Probab. Surv.}, 9:1--89, 2012.

\bibitem{Cook:ssv}
N.~A. Cook.
\newblock Lower bounds for the smallest singular value of structured random
  matrices.
\newblock arXiv:1608.07347, 2016.

\bibitem{Cook:thesis}
N.~A. Cook.
\newblock {\em Spectral properties of non-Hermitian random matrices}.
\newblock PhD thesis, University of California, Los Angeles, 2016.

\bibitem{cook-et-al-in-preparation}
N.~A. Cook, W.~Hachem, J.~Najim, and D.~Renfrew.
\newblock Deterministic equivalents of non-hermitian random matrices:
  Properties and examples.
\newblock {\em In preparation}, 2018.

\bibitem{dumont-et-al-10}
J.~Dumont, W.~Hachem, S.~Lasaulce, P.~Loubaton, and J.~Najim.
\newblock On the capacity achieving covariance matrix for {R}ician {MIMO}
  channels: an asymptotic approach.
\newblock {\em IEEE Trans. Inform. Theory}, 56(3):1048--1069, 2010.

\bibitem{ESY:local2}
L.~Erd{\H{o}}s, B.~Schlein, and H.-T. Yau.
\newblock Local semicircle law and complete delocalization for {W}igner random
  matrices.
\newblock {\em Comm. Math. Phys.}, 287(2):641--655, 2009.

\bibitem{ESY:local1}
L.~Erd{\H{o}}s, B.~Schlein, and H.-T. Yau.
\newblock Semicircle law on short scales and delocalization of eigenvectors for
  {W}igner random matrices.
\newblock {\em Ann. Probab.}, 37(3):815--852, 2009.

\bibitem{EYY:generalizedWigner}
L.~Erd{\H{o}}s, H.-T. Yau, and J.~Yin.
\newblock Bulk universality for generalized {W}igner matrices.
\newblock {\em Probab. Theory Related Fields}, 154(1-2):341--407, 2012.

\bibitem{FeZe97}
J.~Feinberg and A.~Zee.
\newblock Non-{H}ermitian random matrix theory: method of {H}ermitian
  reduction.
\newblock {\em Nuclear Phys. B}, 504(3):579--608, 1997.

\bibitem{Ginibre-1965}
J.~Ginibre.
\newblock Statistical ensembles of complex, quaternion, and real matrices.
\newblock {\em J. Mathematical Phys.}, 6:440--449, 1965.

\bibitem{girko1985circular}
V.~L. Girko.
\newblock Circular law.
\newblock {\em Theory of Probability and Its Applications}, 29(4):694--706,
  1985.

\bibitem{girko-canonical-equations-I}
V.~L. Girko.
\newblock {\em Theory of stochastic canonical equations. {V}ol. {I}}, volume
  535 of {\em Mathematics and its Applications}.
\newblock Kluwer Academic Publishers, Dordrecht, 2001.

\bibitem{gotze2010circular}
F.~G{\"o}tze and A.~Tikhomirov.
\newblock The circular law for random matrices.
\newblock {\em The Annals of Probability}, 38(4):1444--1491, 2010.

\bibitem{Guionnet:band}
A.~Guionnet.
\newblock Large deviations upper bounds and central limit theorems for
  non-commutative functionals of {G}aussian large random matrices.
\newblock {\em Ann. Inst. H. Poincar{\'e} Probab. Statist.}, 38(3):341--384,
  2002.

\bibitem{MR2831116}
A.~Guionnet, M.~Krishnapur, and O.~Zeitouni.
\newblock The single ring theorem.
\newblock {\em Ann. of Math. (2)}, 174(2):1189--1217, 2011.

\bibitem{hachem-et-al-2007}
W.~Hachem, P.~Loubaton, and J.~Najim.
\newblock Deterministic equivalents for certain functionals of large random
  matrices.
\newblock {\em Ann. Appl. Probab.}, 17(3):875--930, 2007.

\bibitem{hachem-et-al-2008}
W.~Hachem, P.~Loubaton, and J.~Najim.
\newblock A {CLT} for information-theoretic statistics of gram random matrices
  with a given variance profile.
\newblock {\em Ann. Appl. Probab.}, 18(6):2071--2130, 2008.

\bibitem{har-01}
L.~A. Harris.
\newblock Fixed points of holomorphic mappings for domains in {B}anach spaces
  [mr1981265].
\newblock In {\em Proceedings of the {I}nternational {C}onference on
  {F}ixed-{P}oint {T}heory and its {A}pplications}, pages 261--274. Hindawi
  Publ. Corp., Cairo, 2003.

\bibitem{HKR16:local}
Y.~He, A.~Knowles, and R.~Rosenthal.
\newblock Isotropic self-consistent equations for mean-field random matrices.
\newblock Preprint arXiv:1611.05364.

\bibitem{speicher-et-al-2007}
J.~W. Helton, R.~Rashidi~Far, and R.~Speicher.
\newblock Operator-valued semicircular elements: solving a quadratic matrix
  equation with positivity constraints.
\newblock {\em Int. Math. Res. Not. IMRN}, (22):Art. ID rnm086, 15, 2007.

\bibitem{hor-joh-topics}
R.~A. Horn and C.~R. Johnson.
\newblock {\em Topics in matrix analysis}.
\newblock Cambridge University Press, Cambridge, 1994.
\newblock Corrected reprint of the 1991 original.

\bibitem{book-horn-johnson}
R.~A. Horn and C.~R. Johnson.
\newblock {\em Matrix analysis}.
\newblock Cambridge University Press, Cambridge, second edition, 2013.

\bibitem{Janaband}
I.~Jana and A.~Soshnikov.
\newblock Esd of singular values of random band matrices; marchenko-pastur law
  and more.
\newblock {\em arXiv preprint arXiv:1610.02153}, 2016.

\bibitem{book-kallenberg-second-edition}
O.~Kallenberg.
\newblock {\em Foundations of modern probability}.
\newblock Probability and its Applications (New York). Springer-Verlag, New
  York, second edition, 2002.

\bibitem{KoSi:survey}
J.~Koml{{\'o}}s and M.~Simonovits.
\newblock Szemer{\'e}di's regularity lemma and its applications in graph
  theory.
\newblock In {\em Combinatorics, {P}aul {E}rd{\H o}s is eighty, {V}ol.\ 2
  ({K}eszthely, 1993)}, volume~2 of {\em Bolyai Soc. Math. Stud.}, pages
  295--352. J{\'a}nos Bolyai Math. Soc., Budapest, 1996.

\bibitem{Lovasz:book}
L.~Lov\'asz.
\newblock {\em Large networks and graph limits}, volume~60 of {\em American
  Mathematical Society Colloquium Publications}.
\newblock American Mathematical Society, Providence, RI, 2012.

\bibitem{may1972will}
R.~M. May.
\newblock Will a large complex system be stable?
\newblock {\em Nature}, 238:413--414, 1972.

\bibitem{mehta1967}
M.~L. Mehta.
\newblock {\em Randon matrices and the statistical theory of energy level}.
\newblock Academic Press, 1967.

\bibitem{najim-yao-2016}
J.~Najim and J.~Yao.
\newblock Gaussian fluctuations for linear spectral statistics of large random
  covariance matrices.
\newblock {\em Ann. Appl. Probab.}, 26(3):1837--1887, 2016.

\bibitem{pan-zhou-circular-2010}
G.~Pan and W.~Zhou.
\newblock Circular law, extreme singular values and potential theory.
\newblock {\em J. Multivariate Anal.}, 101(3):645--656, 2010.

\bibitem{pas-livre}
L.~Pastur and M.~Shcherbina.
\newblock {\em Eigenvalue distribution of large random matrices}, volume 171 of
  {\em Mathematical Surveys and Monographs}.
\newblock American Mathematical Society, Providence, RI, 2011.

\bibitem{Pastur73}
L.~A. Pastur.
\newblock Spectra of random self adjoint operators.
\newblock {\em Uspehi Mat. Nauk}, 28(1(169)):3--64, 1973.

\bibitem{RaAb}
K.~Rajan and L.~F. Abbott.
\newblock Eigenvalue spectra of random matrices for neural networks.
\newblock {\em Phys. Rev. Lett.}, 97:188104, Nov 2006.

\bibitem{rudelson-zeitouni-2016}
M.~Rudelson and O.~Zeitouni.
\newblock Singular values of {G}aussian matrices and permanent estimators.
\newblock {\em Random Structures Algorithms}, 48(1):183--212, 2016.

\bibitem{book-seneta}
E.~Seneta.
\newblock {\em Non-negative matrices and {M}arkov chains}.
\newblock Springer Series in Statistics. Springer, New York, 2006.
\newblock Revised reprint of the second (1981) edition [Springer-Verlag, New
  York; MR0719544].

\bibitem{shlyakhtenko-96}
D.~Shlyakhtenko.
\newblock Random {G}aussian band matrices and freeness with amalgamation.
\newblock {\em Internat. Math. Res. Notices}, (20):1013--1025, 1996.

\bibitem{SCS}
H.~Sompolinsky, A.~Crisanti, and H.-J. Sommers.
\newblock Chaos in random neural networks.
\newblock {\em Phys. Rev. Lett.}, 61(3):259--262, 1988.

\bibitem{Tao:outliers}
T.~Tao.
\newblock Outliers in the spectrum of iid matrices with bounded rank
  perturbations.
\newblock {\em Probab. Theory Related Fields}, 155(1-2):231--263, 2013.

\bibitem{tao-vu-(am)11}
T.~Tao and V.~Vu.
\newblock Random matrices: Universality of local eigenvalue statistics.
\newblock {\em Acta Mathematica}, 206(1):127, Mar 2011.

\bibitem{tao2010random}
T.~Tao, V.~Vu, and M.~Krishnapur.
\newblock Random matrices: universality of esds and the circular law.
\newblock {\em The Annals of Probability}, 38(5):2023--2065, 2010.

\bibitem{TaVu:circ}
T.~Tao and V.~H. Vu.
\newblock Random matrices: the circular law.
\newblock {\em Commun. Contemp. Math.}, 10(2):261--307, 2008.

\bibitem{Wegner:bounds}
F.~Wegner.
\newblock Bounds on the density of states in disordered systems.
\newblock {\em Z. Phys. B}, 44(1-2):9--15, 1981.

\bibitem{book-weidmann}
J.~Weidmann.
\newblock {\em Linear operators in {H}ilbert spaces}, volume~68 of {\em
  Graduate Texts in Mathematics}.
\newblock Springer-Verlag, New York-Berlin, 1980.
\newblock Translated from the German by Joseph Sz{\"u}cs.

\bibitem{Wood:sparse}
P.~M. Wood.
\newblock Universality and the circular law for sparse random matrices.
\newblock {\em Ann. Appl. Probab.}, 22(3):1266--1300, 2012.

\end{thebibliography}

\noindent {\sc Nicholas Cook}\\
UCLA Mathematics Department\\
520 Portola Plaza,\\
Los Angeles, CA 90095-1555\\
e-mail: {\tt nickcook@math.ucla.edu}\\

\noindent {\sc Walid Hachem, Jamal Najim},\\
Laboratoire d'Informatique Gaspard Monge, UMR 8049\\
CNRS \& Universit\'e Paris Est Marne-la-Vall\'ee\\
5, Boulevard Descartes,\\
Champs sur Marne,
77454 Marne-la-Vall\'ee Cedex 2, France\\
e-mail: {\tt \{walid.hachem,jamal.najim\}@u-pem.fr}\\

%
\noindent {\sc David Renfrew}\\
Department of Mathematical Sciences\\
Binghamton University (SUNY)\\
Binghamton, NY 3902-6000\\
e-mail: {\tt renfrew@math.binghamton.edu}\\

\end{document}